\documentclass[letterpaper,11pt]{article}
\usepackage{graphicx}
\usepackage{amsmath}
\usepackage{amsthm}
\usepackage{amssymb} 
 
\usepackage{amsfonts}
\usepackage{hyperref}
\usepackage{amsthm}
 \usepackage{amscd}
\usepackage{mathrsfs}
\usepackage{caption}
\usepackage{subcaption}
\usepackage{float}
\usepackage[top=1in,left=1in,right=1in,bottom=1in]{geometry}

\newtheorem{theorem}{Theorem}[section]
\newtheorem{proposition}[theorem]{Proposition}

\newtheorem{corollary}[theorem]{Corollary}

\newtheorem{remark}[theorem]{Remark}

\hypersetup{
	colorlinks = true,
	linkcolor  = red,
	citecolor  = green,
	filecolor  = cyan,
	urlcolor   = magenta,}

\newcommand{\ds}{\displaystyle}
\newcommand\beq{\begin{equation}}
\newcommand\ene{\end{equation}}

\numberwithin{equation}{section}

\title{Galaxy dynamics, gravitational  Vlasov-Poisson system, Landau damping, and scattering theory\thanks{2020 MSC, 35P25; 35Q83; 35Q85; 47A40; 85A05. }}
\author{
Ricardo Weder\thanks{weder@unam.mx.}\thanks{ Ricardo Weder is an Emeritus National Researcher of SNII-SECIHTI, M\'exico}\\
Departamento de F\'\i sica Matem\'atica\\
Instituto de Investigaciones en
Matem\'aticas Aplicadas y en Sistemas\\
Universidad Nacional Aut\'onoma de M\'exico\\
Apartado Postal 20-126, IIMAS-UNAM\\
 Ciudad de M\'exico, CP 01000, M\'exico}

\date{}

\begin{document}
\baselineskip=15 pt

\maketitle

\begin{abstract} 
\noindent  We consider the gravitational Vlasov-Poisson system linearized around    
 steady states  that are extensively used to study the dynamics of  galaxies, or of clusters of galaxies. Namely,  polytropes and   King  steady states.  We develop a complete stationary scattering theory for the selfadjoint, strictly  positive, Antonov operator that governs the plane-symmetric linearized dynamics. We identify  the absolutely continuous spectrum of  the Antonov operator.   Moreover,  we prove that the part of the   singular spectrum  of the Antonov operator that is embedded in its  absolutely continuous spectrum  is contained in a closed set of measure zero, that we characterize. We construct the generalized Fourier maps, and we  prove that the wave operators exist and are complete. Moreover,  we obtain  stationary formulae for the wave operators, and we prove that Birman's invariance principle holds. Using these results we obtain a precise description of the dynamics of  the stars in the galaxies, or of the galaxies in the clusters of galaxies, for large times. Namely, we prove  that   the distribution function of the solutions to the linearized  gravitational Vlasov-Poisson system with initial data in the absolutely continuous subspace  of the Antonov operator   are asymptotic, for large times, to the solutions to the unperturbed  linearized  gravitational Vlasov-Poisson system.  This implies that they are asymptotic
  to the trajectories of the solutions to Newton's equation with the gravitational potential of the steady state, in the sense that they are transported along these trajectories. Moreover,  for these initial states  the gravitational Landau damping holds. Namely, we  prove that the gravitational force and its time derivative, as well as the gravitational potential and its time derivative, tend to zero  for large times.

\end{abstract}


\section{Introduction}\label{intro}
Galaxy dynamics, generally speaking,  is the study of the dynamics of selfgravitating  matter that consists of the stars in a galaxy, or in a larger scale of clusters of galaxies. Since the early days of this subject it was realized that it is appropriate   to take a statistical point of view and to describe the dynamics by means of a distribution function $f(t,x,v)$  that at time $t$ gives the density of stars, or of galaxies,  at position $x$ in space that have  velocity $v.$ For this purpose Jeans introduced  in 1915 \cite{jeans} an equation  that was later independently considered by Vlasov \cite{vlasov} in the related problem in plasma physics. This  equation came to be known as the Vlasov equation.  To have a selfconsistent  system of equations one has to add the Poisson equation that gives the gravitational potential induced by the self-gravitating matter. This is the gravitational Vlasov-Poisson system that we study in this paper. In three dimensions  it is the following system

\begin{eqnarray}\label{0.1}
\partial_t f(t,x,v)+v\cdot \nabla_x f(t,x,v)- \nabla_x U(t,x)\cdot\nabla_v f(t,x,v)=0,\\
\Delta U(t,x)= 4\pi \rho(t,x),\label{0.1b}\\
\label{0.1c} \rho(t,x):= \int f(t,x,v)\, dv.
\end{eqnarray} 
We have taken the mass of a typical star and Newton's  constant $G$   equal to one. 
 Further, $U(t,x)$ is the potential induced by the self-gravitating matter, and $\rho(t,x)$ is the density of matter. 
The  system  \eqref{0.1}-\eqref{0.1c} is the fundamental system used in astrophysics to study Newtonian galactic  dynamics \cite{bt}. It  has steady-state solutions, i.e. solutions that are independent of time. These solutions are of great importance in the dynamics of galaxies \cite{bt}. In  this paper we are interested in steady states where the distribution function is a  function of the microscopic energy $E(x,v)= \frac{1}{2} v^2+U_0(x).$  In particular, we consider  polytropes and in King steady states, that are extensively used in astrophysics \cite{bt}.
The stability of  steady states is a major concern in astrophysics. There is a large literature on the stability of  steady states, including  nonlinear orbital stability. We comment on these result in Subsection~\ref{preres}.  However, much less  seems to be known about the asymptotic stability of  steady states. We consider a steady-state solution to \eqref{0.1}-\eqref{0.1c}, given by $ (f_0(x,v), U_0(x)),$ with $f_0(x,v)= \varphi(E),$ for an appropriate function $\varphi.$ See  Subsection~\ref{subps}. We linearize the  system \eqref{0.1}-\eqref{0.1c} around $f_0,$ as $ f=f_0+ \varepsilon \delta f.$  
We introduce this expansion in  \eqref{0.1}-\eqref{0.1c}, and keeping only the linear terms in $\varepsilon$ we get,

\begin{eqnarray}\label{0.2}
\partial_t \delta f(t,x,v)+ \tilde{\mathcal D}\delta  f(t,x,v)+ \nabla_x U(t,x) v |\varphi'(E)|=0,\\\label{0.3}
 \Delta U(t,x)= 4\pi\rho(t,x),\\
\rho (t,x):= \int \delta f(t,x,v)\, dv.     \label{0.4}
\end{eqnarray}
The operator  $\tilde{\mathcal D}$ is the transport operator associated with the characteristic flow of the steady state,
\beq\label{2.8}
\tilde{\mathcal D}:= v \partial_x- U'_0(x) \partial_{v}.
\ene
 The behaviour for large times of the macroscopic quantities 
associated to \eqref{0.2}-\eqref{0.4} is referred to as gravitational relaxation \cite{bt}, \cite{lin1}, \cite{lin2}, and \cite{lin3}. In particular, the gravitational Landau damping  is the decay to zero  of the gravitational force $\mathbf  F=-\nabla U.$ 
We consider the system  \eqref{0.2}-\eqref{0.4} under the assumption that the 
phase-space density is {\it plane symmetric}. See Subsection~\ref{vp}. 
The assumption of plane symmetry is extensively used in galaxy dynamics. In particular, in the case of disk galaxies. See, for example  \cite{ant2}, \cite{araki}, \cite{fried},  \cite{hrs}, \cite{kalnajs}, \cite{louis}, \cite{mark}, \cite{mathur}, \cite{wainberg}  and the references quoted there.  We consider the case of planar symmetry, that is somehow simpler,  to concentrate on the main aspects of our  method.   Note, however, that our method  applies in the case of  spherical symmetry  and also if there is an external gravitational potential. In fact,  our method in stationary scattering theory  is general and  it can be used for other problems, like, for example, the plasma physics case.

To study the large time asymptotic of the solutions to the  system \eqref{0.2}-\eqref{0.4} we follow a well established method in galaxy dynamics \cite{ant}, \cite{ant1}, \cite{bt}. We introduce the Antonov operator  that is a selfadjoint, strictly positive operator in a natural Hilbert space $\mathcal H,$ and we consider the Antonov wave equation that is equivalent to the   system  \eqref{0.2}-\eqref{0.4}. See
Subsection~\ref{subant}. In this paper we develop a complete stationary scattering theory for the Antonov operator. 
Our  results on the gravitational Landau damping are given in Theorem~\ref{damping} and  Corollary~\ref{decaypot} where we prove that  the gravitational force,  its time derivative, the gravitational potential and  its time derivative,  tend to zero  for large times. Moreover,  we obtain 
in Section~ \ref{ladab}  our results on the large time asymptotics of   the distribution function. 
 These  results give a precise description of the  dynamics of the  stars in the  galaxies, or of the  galaxies in the clusters of galaxies, for large times that
 to  the best of our knowledge is new.
     
 Stationary scattering theory is a powerful method that has been used in many problems in mathematical physics. See, for example,  \cite{aw}, \cite{kur}, \cite{rs4}, \cite{rs3}, \cite{rw91}, \cite{ya}  and the references quoted there. It appears that it is in this paper where stationary scattering theory  is used in galactic dynamics for the first time. 

The paper is organized as follows. In Subsection~\ref{preres} we discuss the previous results. In Section~\ref{prel} we   introduce the notations and the definitions that we use. We state  preliminary results that we  use in later sections. Further, we  consider the gravitational Vlasov-Poisson system with planar symmetry,   we  linearize it around a steady state, and we  introduce the Antonov operator  \cite{ant}, \cite{ant1}.
Moreover, we introduce a spectral representation of the unperturbed Antonov operator in terms of energy-angle variables.
    In Section~\ref{specan} we consider the spectral theory of the Antonov operator. In particular,
    we prove that the absolutely continuous spectrum of the Antonov operator coincides with its essential spectrum and with the spectrum of the unperturbed Antonov operator and that the part of the singular spectrum of the Antonov operator that is embedded in the absolutely continuous spectrum is contained in a closed set of measure zero.
In Section~\ref{fouma} we construct the generalized Fourier maps and  we prove that they
diagonalize the Antonov operator. 
 In Section~\ref{wave} we prove that the wave operators  exist,  are complete, 
   and that   they are they given by the stationary formulae .
    In Section~\ref{ladab} we obtain our results on Landau damping and on the large time asymptotic of the distribution function.
In Section~\ref{scat} we  develop the scattering theory for the Antonov wave equation.
  In Section~\ref{concl} we give our conclusions.  Finally, in Appendix~\ref{apex} we obtain auxiliary results on the space of H\"older continuous functions and on the energy-angle variables that we need.  
\subsection{Previous results} \label{preres} 
The stability of steady-state solutions to the gravitational Vlasov-Poisson system has been extensively studied.
 In the seminal work of Antonov \cite{ant}-\cite{ant2} a linear  bound for spectral stability was identified. This Antonov bound plays a fundamental role in the study of the stability of the steady states of the gravitational Vlasov-Poisson system, even in  the proofs of orbital stability in the nonlinear case. See, for example, \cite{lem1}, \cite{lem2}, and the review \cite{mouhot2} where  a great deal of information can be found. Concerning asymptotic stability, in the physics literature it has been  suggested in  \cite{lin1}, \cite{lin2}, and \cite{lin3} that  in the linear case mixing could lead to damping. However, it  has been remarked that  that there could be oscillations in the linear case, see \cite{mathur}, Chapter 5 of \cite{bt}, and \cite{fried}.  Recently, in the mathematics literature criteria for the existence and for the absence of eigenvalues of the Antonov operator  where obtained \cite{hrs}, \cite{hrss}, \cite{ku}, and \cite{vdb}. These eigenvalues correspond to oscillatory solutions to the linearized gravitational Vlasov-Poisson system. The results of  \cite{hrs}, \cite{hrss}, \cite{ku}, and \cite{vdb} were obtained by means of a Birman-Schwinger principle, in the spherical symmetric case under different conditions. Moreover, in \cite{hrs}  the case of planar symmetry is considered. Furthermore,  in \cite{hrss}, in the case where the are no bound states, it is proved that   the time average over a time interval $T$ of the time derivative of the gravitational force tends to zero as $T$ tends to infinity.
  Further, \cite{vdb}   proves that  the absolutely continuous spectrum of the Antonov operator and of the unperturbed Anonov operator are the same, by means of a trace class criterium. In the linear case, the articles \cite{chalu1}, \cite{chalu2}, \cite{hrss2}, \cite{vdb2}, and \cite{riosar} consider the Vlasov equation with an external potential. They obtain large-time decay estimates. Note however, that as pointed out in \cite{hrss2} this problem is not equivalent to the linearized Vlasov- Poisson system. Further, \cite{chalu2} also considers the nonlinear gravitational Vlasov-Poisson system. In this paper, for small initial data,  solutions are constructed for a finite time interval, that depends on the smallness of the initial data, and estimates for the solutions during the interval of existence are obtained.  
 
 For the Vlasov-Poisson system in the plasma  physics case the Landau damping, namely the decay of the electric field for large times was discovered in the fundamental work of L. D Landau \cite{landau}. It has been extensively studied. A major breakthrough was the work of Mouhot and Villani  \cite{villani}, who proved  the Landau damping in the nonlinear case. See also \cite{bmm} and \cite{grn}.  When the Vlasov-Poisson system in the plasma physics case is considered with an external constant magnetic field the Landau Damping disappears and the electric field is oscillatory in time, as was shown by Bernstein \cite{bernstein}. For the study of this problem in the mathematics literature see \cite{bedro} and 
 \cite{cdrw}. In \cite{cdrw} it was also proved that there are time independent solutions. For recent reviews of these results see  \cite{bedro1} and \cite{nyu}. These results on Landau damping in the plasma physics case are for perturbations of homogeneous steady states. The proof of Landau damping in the plasma physics case for perturbations of non homogeneous steady states was pioneered by Despr\'es in \cite{bruno1} \cite{bruno2}, who  used scattering theory methods.  See also \cite{hm}. 
 
\section{Preliminaries}\label{prel}
\subsection{Notations and Definitions} \label{not}
For any set $O \subset \mathbb R$ we denote $\chi_O(x)$ the characteristic function of $O.$
For a separable Hilbert space $ \mathcal G,$  for a bounded open interval $I \subset \mathbb R,$ and for $0< \alpha \leq 1,$  we denote by  $C_\alpha(I, \mathcal G)$ the Banach space of all functions defined on $I,$ with values in $\mathcal G,$ that are H\"older continuous with exponent $\alpha.$ See page 116 of \cite{ya} for the definiton.
Further, we designate by $\hat{C}_\alpha(I, \mathcal G)$ the completion of 
$C^\infty(I, \mathcal G)$ in the norm of $C_\alpha(I, \mathcal G).$
Moreover, we designate by  $\hat{C}_{\alpha,0}(I, \mathcal G)$ the set of all functions in $\hat{C}_\alpha(I, \mathcal G)$ with compact support in $I.$
Let    $O$ be a set of real numbers, $\mu$ a sigma-finite complete measure on $O,$ and $\mathcal G$ a Hilbert space. By $L^2(O, \mathcal G; \mu)$ we denote the Hilbert space of all  functions from $O$ into $\mathcal G$ that are measurable and square integrable (Bochner)  with respect to $\mu.$
For  the definition of these spaces see \cite{hp}. If $\mu$ is the Lebesgue measure we use the notation  $L^2(O, \mathcal G),$ and if $\mathcal G= \mathbb C$ we use the notation $L^2(O).$  
We denote by $H_j((0,1)),$ for $j=1,2$ the Sobolev space  \cite{adams}. We denote by $H_{1,{\rm p}}((0,1))$ the closed subspace of $H_1((0,1))$  of all the functions in $H_1((0,1))$ that satisfy $g(0)=g(1).$ Furthermore, let us  designate by $H_{2,{\rm p}}((0,1))$ the closed subspace of $H_2((0,1))$ of all the functions in $H_2((0,1))$ that satisfy  $g(0)=g(1)$ and $g'(0)=g'(1).$
For  Banach spaces $\mathcal B_1$ and  $\mathcal B_2$  we denote by $B(\mathcal B_1, \mathcal B_2 )$ the Banach space of all bounded linear operators from  $\mathcal B_1$ into $\mathcal B_2.$
 For the  definitions below see Section 5 of Chapter 6 and Section 1 of Chapter 10 of \cite{kato}. Let $A$ be a selfadjoint operator in a Hilbert space $\mathcal G.$ Let $E_A(\lambda)$ be the, right continuous,  spectral family of $A$ such that by the spectral theorem
   $
A= \int_\mathbb R \lambda\, dE_A(\lambda).
$
For any Borel set $O \subset \mathbb R,$ let $E_A(O)$ be the spectral projector  for $O.$ We denote by $\mathcal G_{\rm p}(A)$ the discontinuous subspace of $A$, or the point subspace of $A.$ It is the closed subspace of $\mathcal G$ generated by all the eigenvectors of $A.$  We designate by $\mathcal G_{\rm ac}(A)$ the absolutely continuous subspace of $A$ and by
 $\mathcal G_{\rm sc}(A)$ the singular continuous subspace of $A.$  
 For any selfadjoint operator $A$ in $\mathcal G$ we denote the resolvent set of $A$ by $\rho (A),$  and by $\sigma (A)$ the spectrum of $A.$ Further, we denote by 
 $R_A(z)$ the resolvent of $A,$ namely,
 $
 R_A(z):= (A-z)^{-1},z \in \rho(A). 
 $
The point spectrum of $A,$ denoted by $\Sigma_{\rm p}(A),$ is the set of all the eigenvalues of $A.$
  The absolutely continuous spectrum of $A,$ is designated   by $\sigma_{\rm ac}(A).$  Further, the singular continuous spectrum of $A,$ is designated   by $\sigma_{\rm sc}(A).$  The singular spectrum of $A,$  denoted by $\sigma_{\rm sing}(A),$ is the union of the closure of the set of all the eigenvalues of $A$ with the singular continuous spectrum of $A,$ i.e. $\sigma_{\rm sing}(A):=
 \overline{\Sigma_{\rm p}(A)} \cup \sigma_{\rm sc}(A).$
  Moreover, the discrete spectrum of $A,$  that is denoted by $\sigma_{\rm dis}(A),$ consists of all the eigenvalues of $A$ of finite multiplicity that are isolated points of $\sigma(A).$ Further, the  essential spectrum of $A,$ denoted by $\sigma_{\rm ess}(A),$ is the complement in $\sigma(A)$ of the discrete spectrum, $\sigma_{\rm ess}(A):=\sigma(A) \setminus\sigma_{\rm dis}(A).$ 
 We denote by $C$ a generic constant that does not have  to take the same value when it appears in different places.
 Let $ \mathcal P\equiv \mathbb R\times \mathbb R$ be the phase space for one dimensional particles with  position $x \in \mathbb R$ and velocity $ v\in \mathbb R.$
 We use the convention that the  functions $f(x,v)$ defined in a set  $O\subset \mathcal P$  
  are extended   by zero to $\mathcal P\setminus O.$ This allows us to simplify the notation for the domain of integration of several functions below. 

\subsection{The gravitational Vlasov-Poisson system}\label{vp}
We study the Vlasov-Poisson system  under the assumption that the phase-space density is {\it plane symmetric}.
This means the following. We denote $x=(x_1, x_2,x_3), v=(v_1,v_2,v_3), $ and $
\hat{v}:= (v_2,v_3).$  
The phase-space density is {\it plane symmetric} along $x_1$ if it  depends only on the first Cartesian coordinate $x_1$ and it is symmetric for reflections in $x_1$ and $v_1,$
\beq\label{1.2b}
f(t,x,v)= f(t,x_1,v_1, \hat{v})= f(t,-x_1, -v_1, \hat{v}).
\ene
As the phase-space density and the induced gravitational  potential only depend on the first coordinate $x_1$ we drop the subscript one and consider the  system  \eqref{0.1}-\eqref{0.1c} for $x\in \mathbb R,$ and $ v \in \mathbb R^3.$ With this notation the {\it plane-symmetric} system  \eqref{0.1}-\eqref{0.1c} can be written as follows,

\begin{eqnarray}\label{1.3b}
\partial_t f(t,x,v)+v_1\partial_x f(t,x,v)- \partial_x U(t,x)\,\partial_{v_1} f(t,x,v)=0,\\\label{1.4b}
 U(t,x)= 2\pi \int |x-y| \rho(t,y)\, dy,\\
\rho (t,x):= \int  f(t,x,v)\, dv,       \label{1.5b}
\end{eqnarray}
and the condition of {\it planar symmetry} reads,
\beq\label{1.5bb}
f(t,x,v)=f(t,x,v_1, \hat{v})= f(t,-x, -v_1, \hat{v}).
\ene
Note that in  \eqref{1.3b}  and \eqref{1.5bb} the transversal velocity $\hat{v}$ appears as a fixed parameter. In fact,  the system \eqref{1.3b}-\eqref{1.5bb} is equivalent to the one dimensional  gravitational Vlasov-Poisson system,
\begin{eqnarray}\label{1.3}
\partial_t f(t,x,v)+v\partial_x f(t,x,v)- \partial_x U(t,x)\,\partial_{v} f(t,x,v)=0,\\\label{1.4}
 U(t,x)= 2\pi \int |x-y| \rho(t,y)\, dy,\\
\rho (t,x):= \int f(t,x,v)\, dv,       \label{1.5}
\end{eqnarray}
 where  $t,x, v \in \mathbb R,$ 
with the assumption of planar symmetry,
 \beq\label{1.2}
f(t,x,v)= f(t,-x, -v).
\ene
Actually, given a solution, $f(t,x,v)$ to  \eqref{1.3b}-\eqref{1.5bb} the function 
$$
g(t,x,v):= \int  f(x,v, \hat{v})\, d\hat{v}
$$
is a solution to   \eqref{1.3}-\eqref{1.2}. On the other hand, given a solution $ g(t,x,v)$ to 
 \eqref{1.3}-\eqref{1.2} and taking a function $ \phi(\hat{v})$ defined on $\mathbb R^2,$ such that
 $
 \int \phi(\hat{v})\, d\hat{v}=1,
 $ 
the function
$$
f(t,x,v_1, \hat{v}):= g(t,x,v_1) \phi(\hat{v}),
$$
is a solution to \eqref{1.3b}-\eqref{1.5bb}. Hence, in what follows for simplicity  we consider the one-dimensinal {\it plane-symmetric} Vlasov-Poisson system  \eqref{1.3}-\eqref{1.2}.

Note that the {\it planar symmetry} \eqref{1.2} implies that 
the mass density is even,
$
\rho(t,-x)=\rho(t, x).
$
Then, by \eqref{1.4} the potential $U$ is also even,
$
U(t,x)= U(t,-x).
$
Observe  that by the equation for $U$ in \eqref{1.4}
\beq\label{1.8}
\partial_x U(t,x)=  2\pi \int_{\mathbb R} \text{\rm sign}(x-y) \rho(t,y)\, dy= 4\pi \int_0^x \,\rho(y)\, dy.
\ene
Then,  by \eqref{1.8}
\beq\label{1.9b}
\partial_x U(t,0)=0.
\ene

\subsection{Steady states}\label{subps}
Steady states for the Vlasov-Poisson system have been extensively studied. See, for example, \cite{bt}, \cite{fried}, \cite{lem2}, \cite{rein}, and the references mentioned there.
We consider steady states $(f_0(x,v), U_0(x))$  that are time independent solutions   to the {\it plane-symmetric}  system \eqref{1.3}-\eqref{1.2}.
 We are interested in steady states of the form
\beq\label{1.11}
f_0(x,v)= \varphi(E)
\ene
where $E$ is the microscopic energy,
\beq\label{1.10}
E(x,v)= \frac{1}{2} v^2+U_0(x).
\ene
We take  a cutoff energy $E_0 > 0 $ and we require $\varphi$ to be either a polytrope
\beq\label{1.12}
 \varphi(E)= \left\{\begin{array}{l}
(E_0-E)^k, \qquad E_0-E >0,\\
0, \qquad E_0-E \leq 0,
\end{array} \right. 
\ene
where $ k \geq 1,$
or a  King steady state 
\beq\label{1.12b}
 \varphi(E)= \left\{\begin{array}{l}
e^{(E_0-E)}-1,\qquad  E_0-E >0,\\
0, \qquad E_0-E \leq 0.
\end{array} \right. 
\ene
In Proposition2.3 of \cite{hrs} (see also \cite{rein}) it is proved that for the Ansatz   \eqref{1.11}-\eqref{1.12b}
and for each $M_0 >0$ there is an unique {\it plane symmetric } steady state $(f_0,U_0)$ such that $M_0:= \int_{\mathbb R} \rho_0(x)\, dx.$ Moreover, the  mass density $\rho_0$ and the potential $U_0$ are  even,
$ \rho_0(-x)= \rho_0(x),$ and $ U_0(-x)=U_0(x), x \in \mathbb R.$ Further,  $\rho_0\in C^1(\mathbb R),$  it has compact support in $[-R_0,R_0],$ for some $R_0>0,$ and it is strictly decreasing on $[0,R_0].$
 The potential $U_0$ belongs to $C^3(\mathbb R),$ is  convex on $\mathbb R$ and strictly increasing on $[0,\infty].$ Moreover, $U_0(x)= 2\pi M_0\, x$ for $x \geq R_0,$ and $U_0(x)=- 2\pi M_0 \, x$ for $x \leq R_0.$
The potential $U_0(x)$ is positive for $ x \in \mathbb R,$ and $ U_0(0)= \min_{x \in \mathbb R} U_0(x).$
Moreover,  the cutoff energy that appears in \eqref{1.12} and \eqref{1.12b} satisfies
$
E_0= U_0(R_0)=2\pi R_0 M_0.$

Let us  consider the characteristic equations for $(x, v)$ in the steady state,
\beq\label{2.1}
v(t)=x'(t) , \qquad v'(t)= - U'_0(x(t)),
\ene
or equivalently, Newton's equation with the gravitational potential of the steady state
$$
x''(t)=  - U'_0(x(t)).
$$
For every global solution $(x(t), v(t)), t \in\mathbb R,$ to \eqref{2.1} the energy $E=E(x(t),v(t)):=\frac{1}{2}v^2+U_0(x) $ is independent of $t$. Furthermore, for every $ E=E(x(t), v(t)) > U_0(0)$ the solution  $(x(t), v(t))$  to     \eqref{2.1} oscillates periodically between the turning points $x_-(E)$ and $x_+(E)$  
that satisfy  $ x_-(E) < 0 < x_+(E)$ and   $U_0(x_\pm(E))= E.$
The following results are  proved  in \cite{guorein}, \cite{lem2}, and  Lemma 2.4 of \cite{hrs}. We have that  $x_\pm(E_0)= \pm R_0.$ Moreover,
for all $ E_0 \geq E > U_0(0)$ we have that
 $U_0(x) < E$ is equivalent to $ x_-(E) < x < x_+(E).$ Further,
the  $x_\pm$ are continuously differentiable on $(U_0(0), E_0]$ 
and
\beq\label{1.20}
x'_\pm(E)= \frac{1}{U'_0(x_\pm(E))}, \qquad E \in  (U_0(0), E_0].
\ene
Moreover,  $x_-(E)=- x_+(E),$  $x_+(E)$ is strictly increasing for $E\in (U_0(0), E_0],$  and  $x_-(E)$ is strictly decreasing for 
$E \in (U_0(0), E_0].$ Furthermore, 
$
\lim_{E\to U_0(0)} x_\pm(E)= 0. $
Hence, we set $ x_\pm(U_0(0)=0.$

The period $ T(E)$  for  the solution to travel from $x_-(E)$ to   $x_+(E)$ and back from 
 from $x_+(E)$ to   $x_-(E)$ is given by \cite{bt}
 \beq\label{2.3}
 T(E)= 2 \int_{x_-(E)}^{x_+(E)} \frac{1}{ \sqrt{2(E- U_0(x(t)))}}\, dx= 4 \int_{0}^{x_+(E)} \frac{1}{ \sqrt{2(E- U_0(x(t)))}}\, dx.
 \ene
It is proved in  \cite{cw} and  Lemma 2.6  and Propositions 2.7 and 2.8 of \cite{hrs} that 
the period function $T(E)$ is continuously differentiable and strictly increasing, $T'(E)>0,$ for $E \in (U_0(0), E_0].$ 
Further, 
for $ E \in (U_0(0), E_0).$
\beq\label{2.4}
0 < \frac{2\pi}{\sqrt{U''_0(0)}} =\sqrt{\frac{\pi}{\rho_0(0)}}= \lim_{E \downarrow U_0(0)}T(E) < T(E) < T(E_0) < \infty.
\ene

\subsection{Linearization. The Antonov operator}\label{subant}
Let us denote by $\hat{\Omega}_0$ the interior of the support of the distribution function  $f_0$
of the steady-state,
\beq\label{1.18}\begin{array}{l}
\hat{\Omega}_0:= \{ (x,v)\in \mathbb R^2 : f_0(x,v) \neq 0\} = \{  (x,v)\in \mathbb R^2  : E(x,v) < E_0 \}.
\end{array}
\ene
The finite energy cutoff  $E_0$ assures that $\hat{\Omega}_0$ is an open bounded set.
Following \cite{hrs}  we linearize the  system \eqref{1.3}-\eqref{1.5} around $f_0.$ We take $ f=f_0+ \varepsilon \delta f.$  Note that since $\delta f$ has to be small with respect to $f_0$ the support of $\delta f$ has to be contained in the support of $f_0,$ that is equal to $\overline{\hat{\Omega}_0}.$
We introduce this expansion in  \eqref{1.3}-\eqref{1.5}, and keeping only the linear terms in $\varepsilon$ we get,

\begin{eqnarray}\label{2.5}
\partial_t \delta f(t,x,v)+ \tilde{\mathcal D}\delta  f(t,x,v)+ \partial_x U(t,x) v |\varphi'(E)|=0,\\\label{2.6}
 U(t,x)= 2\pi \int  |x-y| \rho(t,y)\, dy,\\
\rho (t,x):= \int \delta f(t,x,v)\, dv,     \label{2.7}
\end{eqnarray}
where $ \varphi'(E)$ denotes the derivative with respect to $E$ of  $\varphi(E).$The operator  $\tilde{\mathcal D}$ is defined in \eqref{2.8}. 
To introduce  the Antonov wave equation \cite{ant}, \cite{ant1}, and \cite{bt}
we  decompose $\delta f$ as the sum of its even and odd parts in $v$, 
\beq\label{2.9}
\delta f= \delta f_++\delta f_-,
\ene
 where
\beq\label{2.10}
\delta f_\pm(x,v)= \frac{1}{2}\left(   \delta f(x,v) \pm  \delta f(x,-v) \right).
\ene
Introducing \eqref{2.9} and \eqref{2.10} into \eqref{2.5}-\eqref{2.7}, taking into account the symmetry in $v,$ and observing  that $\tilde{\mathcal D}$ sends even, respectively odd, functions of $v$ into   odd, respectively even, functions of $v,$   we obtain the following system of equations,
\beq\label{2.11}
\partial_t \delta f_+ +\tilde{\mathcal D} \delta f_-=0,
\ene
\begin{eqnarray}\label{2.12}
\partial_t \delta f_-(t,x,v)+ \tilde{\mathcal D}\delta  f_+(t,x,v)+ \partial_x U(t,x) v |\varphi'(E)|=0,\\\label{2.13}
 U(t,x)= 2\pi \int  |x-y| \rho(t,y)\, dy,\\
\rho (t,x)= \int  \delta f_+(t,x,v)\, dv.     \label{2.14}
\end{eqnarray}
Taking the derivative with respect to $t$ of \eqref{2.12} and \eqref{2.14},  taking the derivative with respect to $x$ and to $t$ of \eqref{2.13}, and using \eqref{2.11} we obtain,
\begin{eqnarray}\label{2.15}
\partial_t^2 \delta f_-(t,x,v)- \tilde{\mathcal D}^2\delta  f_-(t,x,v)+ \partial^2_{t,x} U(t,x) v |\varphi'(E)|=0,\\\label{2.16}
 \partial^2_{t,x}U(t,x)= 2\pi \int \text{\rm sign} (x-y) \partial_t \rho(t,y)\, dy,\\
\partial_t \rho (t,x)=- \int \tilde{\mathcal D}\delta f_-(t,x,v)\, dv.     \label{2.17}
\end{eqnarray}
Further, if $\delta f_- (t,x,v) \in C^1_0(\tilde{\Omega}_0),$
\beq\label{2.17b}
\int \tilde{\mathcal D}\delta f_-(t,x,v)\, dv  = \int  v \partial_x \delta f_-(t,x,v)\, dv.
\ene
 Introducing \eqref{2.17} into \eqref{2.16} and using \eqref{2.17b} we obtain,
\beq\label{2.18}
 \partial^2_{x,t}U(t,x)=- 4 \pi \int v \delta f_-(x,v)\, dv.
 \ene
 Hence, by \eqref{2.18} equation \eqref{2.15} can be written as  follows
 \beq\label{2.19}
 \partial^2_t \delta f_-+ \tilde{\mathcal A} \delta f_-=0,
 \ene
 where $\tilde{\mathcal A}$ is the plane-symmetric Antonov operator,
 \beq\label{2.20}
 \tilde{\mathcal A}:= -\tilde{\mathcal D}^2-\tilde{\mathcal B},
 \ene
 with
 \beq\label{2.21}
\left(\tilde{\mathcal B}g\right)(x,v):=  4 \pi  v |\varphi'(E)|   \int u g(x,u)\, du.
\ene
Remark that if $\delta f$ is plane symmetric, i.e. if it satisfies \eqref{1.2}, also the $\delta f_\pm$ defined in \eqref{2.9},
\eqref{2.10} are plane symmetric. Moreover, for plane-symmetric $f$ we have that $\delta f_+$ is separately symmetric in $ x$ and in $v,$
\beq\label{2.21b}
\delta f_+(x,v)= \delta f_+ (-x,v), \qquad  \delta f_+(x,v)= \delta f_+ (x,-v),
 \ene
 and $\delta f_-$ is separately antisymmetric in $x$ and in $v$
 
 \beq\label{2.21c}
\delta f_-(x,v)= -\delta f_- (-x,v), \qquad  \delta f_-(x,v)= -\delta f_- (x,-v).
 \ene
 Moreover, if $\delta f_+$ is separately symmetric in $x$ and in $v,$  it is also  plane symmetric. Similarly,  if
 $\delta f_-$ is  is separately antisymmetric in $x$ and in $v,$  it is furthermore,  plane-symmetric. Furthermore, if
  $\delta f_-$ is separately antisymmetric in $x$ and in $v,$  the function  $\delta f_+$ defined as in \eqref{2.11}
  will be   separately symmetric in $x$ and in $v,$ if we take the  value of $\delta f_+$ at some initial time 
 separately symmetric in $x$ and in $v.$ Then, we can  consider the Antonov wave equation \eqref{2.19}  in the case of  solutions $\delta f_-$ that are separately antisymmetric in $x$ and in $v.$ Remark that  from the initial data at $t=0$ of the solution  $f(t,x,v)$ to the  system   \eqref{2.5}-\eqref{2.7} we obtain the initial  data at $t=0$  to solve the Antonov wave equation  \eqref{2.19} as follows. We have,  $\delta f_-(0)= \frac{1}{2}(f(0,x,v)-f(0,x,-v)).$ Further,  by \eqref{2.12} at $t=,0$ we get   $\partial_t \delta f_-(0,x,v)=- \tilde{\mathcal D}\delta  f_+(0,x,v)- \partial_x U(0,x) v |\varphi'(E)|, $ with   $\partial_x U(0,x) $ given by \eqref{2.13} and \eqref{2.14} at $t=0,$ and where  $\delta f_+(0,x,v)=  \frac{1}{2}(f(0,x,v)+f(0,x,-v)).$

 Below we properly define the Antonov operator $ \tilde{\mathcal A}$ as a selfadjoint operator in an appropriate Hilbert space of  phase-space density functions  that are separately antisymmetric in $x$ and in $v.$ 

To simplify the notation  we denote $\delta_-f(x,v)= g(x,v),$ and we write \eqref{2.19} as follows
 \beq\label{2.22}
 \partial^2_t g+ \tilde{\mathcal A} g=0,
 \ene
 for solutions $g$ that are separately antisymmetric in $x$ and in $v,$
\beq\label{2.23}
g(-x,v)= - g(x,v), \qquad g(x,-v)= -g(x,v).
\ene
Remark that the operators  $\tilde{\mathcal D}^2$ and  $\tilde{\mathcal B}$ send function that satisfy \eqref{2.23} into functions that fulfill \eqref{2.23}.

As in \cite{hrs} and \cite{ku}  we introduce a convenient Hilbert space to study \eqref{2.22}. Let us denote by $\tilde{\mathcal H}$ the Hilbert space of all complex-valued measurable functions  $g(x,v),$ defined on $\hat{\Omega}_0$ that satisfy \eqref{2.23} and
\beq\label{2.24}
\|g\|_{\tilde{\mathcal H}}:= \left[ \int_{\hat{\Omega}_0} |g(x,v)|^2  \frac{1}{|\varphi'(E)|}   dx dv\right]^{1/2} < \infty.
\ene 
Note  that $\varphi'(E) < 0$  on $\hat{\Omega}_0.$ The quadratic form of $\tilde{\mathcal A}$ as an operator in $\tilde{\mathcal H},$
$$
\left(\tilde{\mathcal A} f,f\right)_{\tilde{\mathcal H}},
$$
coincides with the Antonov functional in the celebrated Antonov stability bound \cite{ant}, \cite{ant1}, \cite{bt}. See Theorem 1.2 in page 10 of \cite{ku}. We find it convenient to extend \eqref{2.22} to complex-valued functions, but note that if the initial data in \eqref{2.22} is real valued the solution remains real valued. 

We denote by $\tilde{\mathcal H}_{\rm even}$ the Hilbert space of all complex-valued measurable functions $g(x,v),$ defined on $\hat{\Omega}_0$ that are separately symmetric in $x$ and in $v,$
\beq\label{2.25b}
g(-x,v)=  g(x,v), \qquad g(x,-v)= g(x,v),
\ene
with the norm \eqref{2.24}.

Let us denote by $C^n_{0,{\rm o}}(\hat{\Omega}_0)$ for $n=1,\dots$  the set of all the functions in $C^n_0(\hat{\Omega}_0)$ that fulfill  \eqref{2.23}. We have that $ \tilde{\mathcal D }$ is a closable operator from $D[\tilde{\mathcal D}]:= C^1_{0, {\rm o}}(\hat{\Omega}_0) \subset \tilde{\mathcal H}$ into $\tilde{\mathcal H}_{\rm even}.$ To prove this let us consider $ f_n \in  C^1_{0,{\rm o}}(\hat{\Omega}_0)$ such that $ f_n \to 0$ strongly in $\tilde{\mathcal H}$  as $ n  \to \infty$ and $  \tilde{\mathcal D} f_n \to g$ strongly in $\tilde{\mathcal H}_{\rm even}$  as $ n  \to \infty$ .Then, for every $ h \in C^1_{0}(\hat{\Omega}_0)$ that satisfies \eqref{2.25b}
$$ 
\left( g, h\right)_{\tilde{\mathcal H}_{\rm even}}=\lim_{n\to \infty} \left(\tilde{\mathcal D} f_n, h\right)_{\tilde{\mathcal H}_{{\rm even}}}=
- \lim_{n\to \infty} \left( f_n, \tilde{\mathcal D} h\right)_{\tilde{\mathcal H}}=0,
$$
and it follows that $g=0.$ This proves that $\tilde{\mathcal D }$ is closable. We denote by $\tilde{\mathcal D}_{\rm o}$ the closure of $\tilde{\mathcal D}.$ We define
\beq\label{2.31c}
\tilde{\mathcal A}_0 =  \tilde{\mathcal D}_{\rm o}^\dagger \tilde{\mathcal D}_{\rm o}.
\ene
By Von Neumann's theorem (Theorem 3.24 in page 275 of \cite{kato}) the operator $\tilde{\mathcal A}_0$ is selfadjoint in $\tilde{\mathcal H}$  and $D[\tilde{\mathcal A}_0]$ is a core for $\tilde{\mathcal D}_{\rm 0}.$  We call $\tilde{\mathcal A}_0$ the unperturbed Antonov operator. Further, observe that  for $ f \in C^2_{0,{\rm o}}(\hat{\Omega}_0)$ we have  $- \tilde{\mathcal D}^2 f= \tilde{\mathcal A}_0 f.$

By Lemma~4.8 of  of \cite{hrs} $\tilde{\mathcal B}$ is bounded and selfadjoint in $\tilde{\mathcal H}$ Then, by the Kato-Rellich theorem, see Theorem 4.3 in page 287 of \cite{kato}, the operator
\beq\label{2.32}
\tilde{\mathcal A }:= \tilde{\mathcal A}_0- \tilde{\mathcal B},
\ene
with domain
\beq\label{2.33}
D\left[ \tilde{\mathcal A}\right]=  D\left[  \tilde{\mathcal A}_0\right]
\ene
is selfadjoint in $\tilde{\mathcal H}.$

By Theorem~7.9 of \cite{hrs} $\tilde{{\mathcal A}}$ is positive with bounded inverse. Hence, the  Antonov wave equation \eqref{2.22}
 fits into the framework of abstract wave equations  in $\tilde{\mathcal H}.$  See \cite{kato2}
and  Section  10 in Chapter XI of\cite{rs3}.

\begin{remark}\label{rem} {\rm In the derivation of the Antonov wave equation \eqref{2.19}, \eqref{2.22} we used \eqref{2.17b}.
Note that  \eqref{2.17b} holds also for $ \delta f_- \in D[\tilde{\mathcal D}_{\text{\rm o}}].$ See Remark~ 5.18 and equation A.4 in 
\cite{hrs}.}
\end{remark} 

\subsection{The Antonov operator in energy-angle variables}
Action-angle variables \cite{arnold} are extensively used in the study of the Vlasov-Poisson system. See, for example, \cite{bsdn}, \cite{bt}, \cite{mo1}, \cite{mo2}, and \cite{ku}. Here we use the closely related energy-angle variables that  play an important role in later sections. We denote,
\beq\label{2.34}
E_{\text{\rm min}}:= {\text{\rm min}}_{(x,v)  \in \mathbb R^2} E(x,v)= E(0,0)=U_0(0).
\ene
Let us  introduce the following  angle variable \cite{arnold}, \cite{hrs},
\beq\label{2.36} 
\theta(x,E) := \frac{1}{T(E)} \int_{x_-(E)}^x \frac{1}{\sqrt{2(E-U_0(y))}}\, dy,\qquad  E>   E_{\rm min},\, x_-(E)\leq x \leq
x_+(E),
\ene
where  $T(E)$ is defined in \eqref{2.3}.  Note that $ \theta \in [0, 1/2].$
We define,
\beq\label{2.38}
I_0:= (E_{\rm min}, E_0),
\ene
and
\beq\label{2.38.b}
\Omega_0:=   \mathbb S_1 \times I_0,
\ene
where $\mathbb S_1$ is the circle, i.e.,  $[0,1]$ with $0$ and $1$ identified. 
Let $\mathcal M$ \cite{hrs} be the following mapping from $\hat{\Omega}_{0}\setminus\{(0,0)\}$ onto $\Omega_0$ 
\beq\label{2.38c}\begin{array}{l}
\mathcal M(x, v)= (\theta(x,E), E),          \qquad v \geq 0, \\[5pt]
\mathcal M(x, v)= (1- \theta(x,E), E),\qquad v < 0, E=\frac{1}{2}v^2 +U_0(x).
 \end{array}
 \ene
The mapping $\mathcal M$  is a bijection from $\hat{\Omega}_{0}\setminus\{(0,0)\}$ onto $\Omega_0.$
To construct the inverse of $\mathcal M$ let us  proceed as follows \cite{hrs}. For a fixed $ E\in I_0,$ let
\beq\label{2.40}
 t \in \mathbb R \to (X(t,E) ,V(t,E))
 \ene be the unique solution to the characteristic equations \eqref{2.1} that satisfies the initial condition
\beq\label{2.40b}
(X(0,E),V(0,E))= (x_-(E), 0).
\ene
For $(\theta, E) \in \Omega_0$
the quantity $(x,v)$ is given by,
\beq\label{2.41}
(x(\theta,E),v(\theta,E))= (X(\theta T(E),E),V(\theta T(E), E)),
\ene
where for $\theta  \in [0,  1/2], v \geq 0,$  for  $\theta  \in [1/2, 1], v \leq 0,$ and $v=0,$ for $\theta= 1/2.$
Furthermore,
\beq\label{2.42}
dx\,  dv= T(E)\, d\theta\, dE,
\ene
 For $g \in C_0(\hat{\Omega}_{0}\setminus\{ (0,0)\} )$  we define
\beq\label{2.46}
\left( {\mathbf U} g  \right)( \theta, E)= g(x(\theta, E), v(\theta, E) ).
\ene
Note, see Remark 5.13 of \cite{hrs}, that $g$ satisfies \eqref{2.23} if and only if $ \mathbf U g $ fulfllls

\begin{eqnarray}
\label{2.59}
\left( \mathbf U g  \right)(\theta, E)= - \left( \mathbf U g  \right)(1-\theta, E),\qquad  \theta \in [0,1], \\
\label{2.60} \left( \mathbf U g  \right)( \theta, E)= - \left( \mathbf U g  \right)(1/2-\theta,E),\qquad  \theta \in [0,1/2],\\\label{2.61} \left( \mathbf U g  \right)(\theta, E)= - \left( \mathbf U g  \right)( -\theta+3/2,E), \qquad \theta \in [1/2,1],
\end{eqnarray}
for all $ E \in I_0.$ 
We  designate by $\mathcal{ H}$ the Hilbert space of all complex-valued measurable functions $g(\theta, E)$ defined on   ${\Omega}_0$ that satisfy

\begin{eqnarray}
\label{2.62}
 g(\theta, E)= -  g(1-\theta, E),\qquad \theta \in [0,1], \\
\label{2.63a} g( \theta, E)= -  g(1/2-\theta,E), \qquad \theta \in [0,1/2],\\\label{2.64}
 g(\theta, E)= -  g( -\theta+3/2,E),\qquad  \theta \in [1/2,1],
\end{eqnarray}
 for all $ E \in I_0,$ with the norm
\beq\label{2.44}
\|g\|_{{\mathcal H}}:= \left[ \int_{{\Omega}_0} |g(\theta, E)|^2  \frac{1}{|\varphi'(E)|}  \, T(E) \,d\theta 
\,dE  \right]^{1/2} < \infty.
\ene 
By \eqref{2.42} for $g \in C_0(\hat{\Omega}_{0}\setminus\{ (0,0)\} )$ that satisfies \eqref{2.23} we have,
\beq\label{247}
\| {\mathbf U}  g\|_{{\mathcal H}}= \| g\|_{\tilde{{\mathcal H}}},
\ene
and $\mathbf U $ extends into a unitary operator from $\tilde{{\mathcal H}}$ onto ${\mathcal H}.$

We designate by $L^2_{\rm o}(S_{1})$ the closed subspace of $L^2(S_1)$ consisting of all the functions $g\in L^2(S_1)$ such that,
\begin{eqnarray}
\label{2.56c}
 g(\theta)= -  g(1-\theta), \theta \in [0,1], \\
\label{2.63ab} g( \theta)= -  g(1/2-\theta), \theta \in [0,1/2],\\\label{2.64b}
 g(\theta)= -  g( -\theta+3/2), \theta \in [1/2,1].
\end{eqnarray}
We have,
\beq\label{2.47b}
\mathcal H= \ds L^2\left(I_0,L^2_{\rm o}(S_{1}); \frac{T(E)}{|\varphi'(E)|} \, dE \right).
\ene
By $F_{S_1}$ we denote the Fourier series on $L^2(S_1)$ \cite{gr}
\beq\label{2.55}
({F_{S_1}} g)_n= \int_0^1 e^{-2\pi i n\theta } g(\theta)\, d\theta, \qquad n \in \mathbb Z.
\ene
The Fourier series $ F_{S_1}$ is an unitary  operator from $L^2(S_1)$  onto the standard Hilbert space $l^2(\mathbb Z)$ of complex-valued square-summable sequences $\{g_n \}, n \in \mathbb Z.$    The inverse operator $ F_{S_1}^{-1}$ is given by,
 \beq\label{2.56b}
  F_{S_1}^{-1}\{ g_n\}= \sum_{n \in \mathbb Z} e^{2\pi i n\theta} g_n,\qquad \{g_n\} \in l^2(\mathbb Z).
   \ene   
It follows from a simple calculation that if $ g\in L^2_{\rm o}(S_{1}),$
\begin{eqnarray}\label{2.65}
&(F_{S_1} g)_{2l+1}=0, \qquad l \in \mathbb Z, \\\label{2.66}
&( F_{S_1} g)_{2l}  = - ({ F_{S_1}}g)_{-2l}, \qquad l \in \mathbb N,\\ \label{2.67}
&( F_{S_1} g)_{0}=0.
\end{eqnarray}
Let us designate by $l^2(\mathbb N)$ the standard Hilbert space of  complex-valued square-summable sequences $\{g_l \}, l \in \mathbb N.$ 
For $ g \in L^2_{\rm o} (S_1)$ we denote,
\beq\label{2.74}
 \left(F_{S_{1},{\rm o}}g\right)_l= -i \sqrt{2} \int_0^1 \sin(4\pi l\theta) g(\theta) d\theta, \qquad  l \in \mathbb N.
 \ene 
 As $F_{S_1}$ is unitary from $L^2(S_1)$ onto $l^2(\mathbb Z),$ it follows from \eqref{2.65}-\eqref{2.67} that $F_{S_{1},{\rm o}}$ is unitary from $ L^2_{\rm o}(S_{1})$  onto $l^2(\mathbb N).$
Further if $\{ g_{l}\}_{l=1}^\infty\in l^2(\mathbb N),$
\beq\label{2.74b}
F_{S_{1,{\rm o}}}^{-1} \{ g_{l}\} =  i \sqrt{2} \sum_{l=1}^\infty \sin(4\pi  l\theta) g_{l},\qquad \{g_l\}\in l^2(\mathbb N).
\ene
 Let $\mathcal Q$ be the Hilbert space of all complex-valued measurable functions $g( E)$ defined on   $I_0$  that satisfy 
\beq\label{2.52}
\|g\|_{{\mathcal Q}}:= \left[ \int_{I_0 } |g( E)|^2 \, \frac{1}{|\varphi'(E)|}  \, T(E) 
\,dE  \right]^{1/2} < \infty.
\ene 
 Let us denote
 \beq\label{2nn.1}
 \hat{\mathcal H}= \oplus_{l=1}^\infty  \mathcal Q.
 \ene
 for $ g \in \mathcal H$ we define
\beq\label{2.55bb}
({\mathbf F} g)_l(E)=\left( F_{S_{1, {\rm o}} }g(\cdot, E)\right)_l, \qquad l \in \mathbb  \in \mathbb N.
\ene
 The operator $\mathbf F$ is unitary from $\mathcal H$  onto $ \hat{\mathcal H}.$
 
 Below we unitarily transform $\tilde{\mathcal A}_0$ to $\mathcal H$ and to $\hat{\mathcal H}.$ For this purpose we prepare the following proposition.

\begin{proposition}\label{lap}
Let $b_0$ be the following operator in $ L_{\rm o}^2(S_{1})$
\beq\label{new.1}
b_0 g(\theta):= - \frac{d^2}{d \theta^2}g(\theta),
\ene
with domain,
\beq\label{new2}
D[b_0]= H_2((0,1))\cap H_{1,{\rm p}}((0,1))\cap L^2_{\rm o}(S_{1}).
\ene
Then, $b_0$ is selfadjoint and moreover,
\beq\label{new.3}
b_0= \ds F_{S_{1,{\rm o}}}^{-1} (4\pi l)^2 F_{S_{1,{\rm o}}},
\ene
where $(4 \pi l)^2$ denotes  the operator of multiplication by $(4 \pi l)^2$ in $l^2(\mathbb N).$ Namely, it is 
the operator $\{g_l\}\to \{  (4\pi l)^2  g_l\},$ defined on the domain of all  $\{g_l\} \in l^2(\mathbb N)$ such that
  $\{ (4\pi l)^2  g_l\} \in l^2(\mathbb N).$
\end{proposition}
\begin{proof} Since the functions in $H_{2}((0,1))$ as well as its first derivatives extend continuously to zero and to one, it follows from \eqref{2.56c} that  for $g \in  D[b_0]$ we have, $ g(0)=g(1)=0,$ and $g'(0)= g'(1).$ Then,
$D[b_0] \subset  H_{2,{\rm p}}((0,1)).$ Integrating by parts we get
\beq\label{new.4}
\left( F_{S_{1,{\rm o}}} b_0\, g\right)_l=  (4\pi l)^2  \left( F_{S_{1,{\rm o}}} g\right)_l, \qquad g\in D[b_0], l=1,\dots.
\ene
Hence, $b_0 \subset F_{S_{1,{\rm o}}}^{-1} (4\pi l)^2 F_{S_{1,{\rm o}}}.$ Using \eqref{2.74b} we prove that 
$ F_{S_{1,{\rm o}}}^{-1} (4\pi l)^2 F_{S_{1,{\rm o}}}\subset b_0.$ Then,   \eqref{new.3} follows and $b_0$ is selfadjoint as it is unitarily equivalent to the operator of multiplication by
$(4 \pi l)^2$ in $l^2(\mathbb N)$ that is selfadjoint.
\end{proof}
Let $q_0$ be the quadratic form in $L^2_{\rm o}(S_{1})$
\beq\label{new.5}
q_0(f,g)= (\partial_\theta f, \partial_\theta g)_{L^2(S_1)}, \qquad f,g \in D[q_0]= H_{1,{\rm p}}((0,1))\cap L_{\rm o}^2(S_{1}).
\ene
We have\beq\label{new.6}
q_0(f,g)= \sum_{l=1}^\infty     (4\pi l) \left( F_{S_{1, {\rm o}}}f\right)_l  (4\pi l) \overline{\left(F_{S_{1, {\rm o}}} g\right)_l}. 
\ene
Hence, by \eqref{new.3}, \eqref{new.6}, and Theorem 2.1 in page 322 of \cite{kato} $q_0$ is the quadratic form of
$b_0.$  

We denote by  $\tilde{\mathcal H}_{\rm ex}$  the Hilbert space of all complex-valued measurable functions $g(x,v)$ defined in $\hat{\Omega}_0$  with the norm \eqref{2.24}.  Similarly, we   designate by $\mathcal{ H}_{\rm ex}$ the Hilbert space of all complex-valued measurable functions $g(\theta, E)$ defined on   ${\Omega}_0$  with the norm \eqref{2.44}. Note that the functions in $\tilde{\mathcal H}_{\rm ex}$ and in $\mathcal H_{\rm ex}$ are not restricted by any symmetry assumption. By \eqref{2.42} the operator $\mathbf U$ defined in \eqref{2.46} extends to an unitary operator from   $\tilde{\mathcal H}_{\rm ex}$ onto  $\mathcal H_{\rm ex}.$ We use the same notation for $\mathbf U$ as an operator from  $\tilde{\mathcal H}$ onto    ${\mathcal H}$ and as an operator  from $\tilde{\mathcal H}_{\rm ex}$ onto  ${\mathcal H}_{\rm ex}.$ It will be clear from the
context to which spaces we are applying this operator.

 We remark that by the chain rule for derivatives, for $ f\in C^1_{0}(\hat{\Omega}_0\setminus \{(0,0)\})$ 
 \beq\label{new.5b}
 \left( \mathbf U  \tilde{\mathcal D} f\right)(\theta, E)= \frac{1}{T(E)} \partial_\theta  \left(\mathbf U f\right)(\theta, E).
 \ene
Hence, for $ f \in D[\tilde{\mathcal D}_{\rm 0}]$ and $g \in C^1_{0}(\hat{\Omega}_0\setminus \{(0,0)\}),$

 \beq\label{new.7}\begin{array}{l}
 \left(\mathbf U \tilde{\mathcal D}_{\rm o} f, \mathbf U g \right)_{{\mathcal H}_{\rm ex}}=
 \left( \tilde{\mathcal D}_{\rm o} f, g \right)_{\tilde{\mathcal H}_{\rm ex}}= - \left(f, \tilde{\mathcal D} g \right)_{\tilde{\mathcal H}_{\rm ex}}= \\  - \left( \mathbf U f,  \frac{1}
{T(E)}\partial_\theta \mathbf Ug \right)_{{\mathcal H}_{\rm ex}}=
 \left( \frac{1}
{T(E)} \partial_\theta\mathbf U f,\mathbf U g \right)_{{\mathcal H}_{\rm ex}}.
  \end{array}
 \ene
 By \eqref{new.7}
 \beq\label{new.9}
 \mathbf U \tilde{\mathcal D}_{\rm o} f = \frac{1}
{T(E)} \partial_\theta \mathbf U f, \qquad f \in D[ \tilde{\mathcal D}_{\rm o}].
 \ene
 We denote by $\mathcal A_0$ the following selfadjoint operator in $\mathcal H,$
 \beq\label{2.50}
{\mathcal A}_0:= \frac{1}{T(E)^2} b_0,
\ene
with the domain 
\beq\label{2.51}
D[{\mathcal A}_0]=\left \{ \begin{array}{l} g \in \mathcal H : \hbox{\rm for a.e.}\, E \in I_0,
g(\cdot, E) \in D[b_0] \, \hbox{\rm and}\\[5pt]
\int_{I_0} \|( g)(\cdot, E)\|^2_{H_2([0,1])} \,\frac{1}{  T(E)^3 |\varphi'(E)|}    \,dE   < \infty
\end{array}\right\}.
\ene 
By Theorem 2.1 in page 322 of \cite{kato} the quadratic form of $\mathcal A_0$ is given by,
\beq\label{new.10}
a_0(f,g):= \int_{I_0}  q_0(f(\cdot,E),g(\cdot,E))  \,\frac{1}{  T(E)|\varphi'(E)|}       dE,
\ene
 with domain,
 \beq\label{new.11}
 D[{a}_0]=\left \{ \begin{array}{l} g \in \mathcal H : \hbox{\rm for a.e.}\, E \in I_0,
g(\cdot, E) \in D[q_0]  \,\hbox{\rm and}\\[5pt]
\int_{I_0} q_0(g,g)\,\frac{1}{  T(E)|\varphi'(E)|}     \,dE   < \infty
\end{array}\right\}.
\ene 
 Moreover, by Theorem 2.1 in page 322 of \cite{kato} the quadratic form of $\tilde{\mathcal A}_0$ is given by,
 \beq\label{new.12}
 \tilde{a}_0(f,g):= \left( \tilde{\mathcal D}_{\rm o} f,   \tilde{\mathcal D}_{\rm o} g \right)_{\tilde{\mathcal H}_{\rm even}},
 D[\tilde{a}_0]:= D[\tilde{\mathcal D}_{\rm o}].
 \ene
 Further, by \eqref{new.9} and \eqref{new.10}-\eqref{new.12}, for $ f \in D[\tilde{\mathcal A}_0]$ and $ g \in D[\tilde{\mathcal D}_{\rm o}]$ we have,
 \beq\label{new.13}
 \left( \mathbf U \tilde{\mathcal A}_0 f, \mathbf U g \right)_{\mathcal H}=\left( \tilde{\mathcal A}_0f, g \right)_{\tilde{\mathcal H}}= \tilde{a}_0(f,g)= a_0(\mathbf Uf,\mathbf U g).
 \ene
By\eqref{new.13} and  Theorem 2.1 on page 322 of \cite{kato}  $\mathbf U f \in D[{\mathcal A}_0]$ and
\beq\label{new.14}
{\mathcal A}_0 \mathbf U f = \mathbf U \tilde{\mathcal A}_0 f, \qquad f \in D[\tilde{\mathcal A}_0].
\ene
Then,
\beq\label{new.15}
{\mathcal A}_0= \mathbf U \tilde{\mathcal A}_0 \mathbf U^{-1}.
\ene
With $\mathbf F$ defined in \eqref{2.55bb} we define,
\beq\label{new.16}
\hat{\mathcal A}_0:= \mathbf F {\mathcal A}_0  \mathbf F^{-1}.
\ene
 Moreover, by \eqref{new.3}  and  \eqref{2.50}, according to  the direct sum in \eqref{2nn.1} 
  \beq\label{2.57}
 \hat{\mathcal A}_0:= \mathbf F  {\mathcal A}_0 \mathbf F^{-1}= \oplus_{l \in \mathbb N} \frac{(4\pi l)^2}{T(E)^2 }
,\ene
where by $ \frac{1}{T(E)^2 }(4\pi l)^2$ we denote the operator of multiplication by $ \frac{1}{T(E)^2 }(4\pi l)^2.$
As  $T(E)$ is continuously differentiable and strictly increasing,  it follows from  \eqref{2.57} that the spectrum of $ \hat{\mathcal A}_0$ is  absolutely continuous and it is given by
\beq\label{2.58}
\sigma(\hat{\mathcal A}_0)=\sigma_{\text{\rm ac}}(\hat{\mathcal A}_0)= \cup_{l \in \mathbb N}
\left\{  \frac{(4\pi l)^2}{T(E)^2 }
 : E \in [E_{\rm min}, E_0]\right\}.
\ene
Further, by \eqref{2.57} the spectrum of $\mathcal A_0$ is also absolutely continuous and it is given by the right hand side of \eqref{2.58}.

 We denote by  $R_0(z)$ the resolvent ${\mathcal A}_0$  for  $z \in \rho({\mathcal A}_0).$ 
 By \eqref{2.74}, \eqref{2.74b}, and  \eqref{2.57} $R_0(z)$ is an integral operator,
\beq\label{2.75}
\left(R_0(z)g\right)(\theta, E)= \int_{\mathbb S_1} R_0(z, E,\theta,  \theta_1, ) g(\theta_1,E) d\theta_1,
\ene 
 where the integral kernel $R_0(z, E,\theta,  \theta_1)$ is given by,
 \beq\label{2.76}
 R_0(z, E,\theta, \theta_1):= 2  \sum_{l =1}^\infty \sin(4\pi l \theta) \left( \frac{(4\pi l)^2}{T(E)^2} -z\right)^{-1} \sin(4\pi l \theta_1).
 \ene
 In  Theorem 5.19 of \cite{hrs} it is proved that   $\tilde{\mathcal B}$ is relatively compact with respect to  $\tilde{\mathcal A}_0$ and that by Weyl's theorem the essential spectrum of $\tilde{\mathcal A}$ is given by,
\beq\label{2.80}
\sigma_{\rm ess}(\tilde{\mathcal A})= \cup_{l\in \mathbb N}
\left\{  \frac{(2\pi l)^2}{T(E)^2 }
 : E \in [E_{\rm min}, E_0]\right\}.
\ene
We denote,
\beq\label{2.81a}
\mathcal A:= \mathbf U \tilde{\mathcal A}\, \mathbf U^{-1}= {\mathcal A}_0- \mathcal B,
\ene
where,
\beq\label{2.81b}
\mathcal B:= \mathbf U \tilde{\mathcal B} \mathbf U^{-1}.
\ene
Then,
 \beq\label{2.81c}
 \sigma_{\rm ess}({\mathcal A}) = \cup_{l \in \mathbb N }
\left\{  \frac{(4\pi l)^2}{T(E)^2 }
 : E \in [E_{\rm min}, E_0]\right\}.
\ene
By \eqref{2.81a}, equation \eqref{2.22} is unitarily equivalent to,
 \beq\label{2.81d}
 \partial^2_t g+ {\mathcal A} g=0.
 \ene

 \subsection{The spectral representation of $\mathcal A_0$}\label{subspec}
 We end this section constructing a spectral representation of the unperturbed Antonov operator $\mathcal A_0$ in terms of {\it trace maps} that plays a crucial role in the proof, that we give in  Theorem~\ref{propgf}, that the generalized Fourier maps are onto the Hilbert space that we use in this spectral represention. We denote,
  \beq\label{sp.1}
  \beta_l(E):= \frac{(4\pi l)^2}{ T^2(E)},\qquad
 l=1,\dots,
  \ene
 and
  \beq\label{sp.2}
 \beta_{l,\rm min}= \frac{(4\pi l)^2}{T(E_0)^2}, \qquad   \beta_{l,\rm max}= \frac{(4\pi l)^2}{T(E_{\rm min})^2}, \qquad
 l=1,\dots.
 \ene

   Since $T(E)$ is increasing the function $\beta_l(E)$ is invertible. 
   We denote by 
   \beq\label{sp.2b}
   E_l(\beta), \qquad \beta \in [\beta_{l,\rm min}, \beta_{l,\rm max}],
   \ene
    the inverse function that fulfills,
  \beq\label{s.42bc}
  E_l(\beta_l(E))= E, \qquad E \in [E_{\rm min}, E_0]. 
  \ene
    By the inverse function theorem,
   \beq\label{s.43}
 p_l(\beta):= E_l'(\beta)= -\frac{T^3(E_l(\beta))}{2 (4\pi l)^2 T'(E_l(\beta))}.
  \ene
 We denote by $\mathcal H_{\rm sp}$ the  Hilbert space
 \beq\label{sp.4}
 \mathcal H_{\rm sp}:= \oplus_{l=1}^\infty L^2((\beta_{l,\rm min}, \beta_{l,\rm max})).
 \ene
We define the following unitary operator $U$ from $\hat{\mathcal H}$ onto $\mathcal H_{\rm sp},$ 
 \beq\label{sp.5}
 \begin{array}{l}
 \left(  U \{g_{l}\}\right)_j(\beta):= 
 \ds  \frac{\sqrt{ T(E_j(\beta)) |p_j(\beta)}|}{\sqrt{|\varphi'(E_j(\beta)|}} g_{j} (E_j(\beta)), \beta \in( \beta_{j,\rm min}, \beta_{j,\rm max}),\{g_{l}\}\in \hat{\mathcal H}, \\  j=1,\dots.
\end{array}
 \ene  
Further, we define,
\beq\label{sp.5b}
\mathcal F:=  U \mathbf F,
\ene
where $\mathbf F$ is defined in \eqref{2.55bb}.
The operator $\mathcal F$ is unitary from $\mathcal H$ onto $\mathcal H_{\rm sp}.$
 
 We define
 \beq\label{sp.6}
 \mathcal A_{0,{\rm sp}}:=     \mathcal  F   \mathcal A_0 \,  \mathcal F^{-1}.
 \ene
  Then, by \eqref{2.57},
  \beq\label{sp.7}
\left( \mathcal A_{0, \rm{sp}}\{ g_l\}\right)_{j}(\beta)=\beta g_{j}(\beta),\qquad  \beta \in( \beta_{l,\rm min}, \beta_{l,\rm max}), \{g_l\} \in \mathcal H_{\rm sp},  j=1,\dots,.
\ene
In other words, under the direct sum \eqref{sp.4}
\beq\label{sp.8}
\mathcal A_{0,{\rm sp}}= \oplus_{l=1}^\infty \beta,
\ene
where $\beta$ is the operator of multiplication by $\beta$ in  $L^2((\beta_{l,\rm min}, \beta_{l,\rm max})).$ 
Note  moreover, that for any  Borel set $\Delta$
\beq\label{sp.8b}
\left(\mathcal F E_{\mathcal A_0 }(\Delta) f\right)_l(\beta)= \chi_\Delta(\beta) \left(\mathcal F f\right)_l(\beta), \qquad \beta \in( \beta_{l,\rm min} \beta_{l,\rm max}),\,\, l=1,\dots.
\ene

We now write the spectral representation of $\mathcal A_0$ in a convenient way in terms of {\it trace maps}. We write  $\sigma(\mathcal A_{0,{\rm sp}})$ in the following way
\beq\label{sp.9}
\sigma (\mathcal A{_0,\text{\rm sp}})= \cup_{n=1}^\infty \Delta_n,
\ene
where $\Delta_n, n=1,\dots,$ are bounded disjoint intervals, $\Delta_n \cap \Delta_j=\emptyset,$ for $n\neq j,$ such  that there is a positive integer $m_n$ and positive integers $1 \leq l_{n,1} < l_{n,2} <  \dots <l_{n, m_n}$ with the property that for every $\beta \in \Delta_n,$ we have,
\beq\label{sp.10}
 \beta \in \ds \cap_{l=l_{n,1}}^{l_{n,m_n}} [\beta_{l,{\rm min}}, \beta_{l,{\rm max}}], \qquad \beta\notin[ \beta_{l,{\rm min}}, \beta_{l,{\rm max}}], \qquad l \notin \{l_{n,1},\dots l_{n,m_n}\}.
\ene 
Moreover, we take $m_1=1,$ and $\Delta_1 \subset [ \beta_{1,{\rm min}}, \beta_{1,{\rm max}}].$

Note that $\mathcal H_{\rm sp}$ can be written as follows,
\beq\label{sp.11}
\mathcal H_{\rm sp}= \oplus_{n=1}^\infty L^2(\Delta_n, \mathbb C_{m_n}).
\ene
Furthermore under the direct sum \eqref{sp.11}  $\mathcal A_{0,{\rm sp}}$ is again given by,
\beq\label{sp.18}
\mathcal A_{0,{\rm sp}}= \oplus_{n=1}^\infty \beta,  
\ene
with $\beta$ the  operator of multiplication by $\beta$ in  $ L^2(\Delta_n, \mathbb C_{m_n}).$  Remark that in $ L^2(\Delta_n, \mathbb C_{m_n})$ the spectrum of $\mathcal A_{0,{\rm sp}}$ has multiplicity $m_n.$

Recall that   the space $\mathcal H$ can be realized as in \eqref{2.47b}.
Further,  by \eqref{sp.8b}, under the direct sum \eqref{sp.11}, for any Borel set $\tilde{\Delta}$ contained in the interior of $\Delta_n,$ and  any $f \in \mathcal H \cap C(I_0,(S_{1,0}))$, 
\beq\label{sp.19}
\left(\mathcal F E_{\mathcal A_0 }(\tilde{\Delta}) f\right)(\beta)= \chi_{\tilde{\Delta}}(\beta) \mathcal L_{n}(\beta) 
\mathbf F f,
\ene
where, for $\beta$ in the interior of $\Delta_n$ the {\it trace map}  operator   $\mathcal L_n(\beta)$ is defined as follows. We denote by $\mathcal H_{{\rm sp}, m_n}$ 
the vector subspace of $\mathcal H_{\rm sp}$  consisting of all the vectors in $\mathcal H_{\rm sp}$ with  components in  $L^2(\Delta_n, \mathbb C_{m_n})$  continuous. Then, $\mathcal L_n(\beta)$ is the {\it trace map} operator  from $\mathcal H_{{\rm sp}, m_n}$  into $\mathbb C_{m_n}$ is given by, 
\beq\label{sp.20}
\begin{array}{r}
\mathcal L_n(\beta) \{f_l\}_{l=1}^{m_n}= \ds  \left( \frac{\sqrt{ T(E_{l_{n,1}}(\beta)) |p_{l_{n,1}}(\beta)|}}{\sqrt{|\varphi'(E_{l_{n,1}}(\beta))|}} f_{l_{n,1}}(E_{l_{n,1}}(\beta)),   
\dots, \right.\\  \noalign{\medskip}\left. \ds \frac{\sqrt{ T(E_{l_{n,m_n}}(\beta)) |p_{l_{n,m_n}}(\beta)|}}{\sqrt{|\varphi'(E_{l_{n,m_n}}(\beta))|}}f_{l_{n,m_n}}(E_{l_{n,m_n}}(\beta)) \right).
\end{array}
\ene
We call $\mathcal L_n(\beta)$ a {\it trace map} because it takes the trace of  $ \{f_l\}_{l=1}^{m_n}$ at the value
$\beta$ of the spectral parameter.

\section{Spectral analysis of the Antonov operator}\label{specan}
By \eqref{2.81c}  the essential spectrum of the Antonov operator $\mathcal A$ coincides with the spectrum of the unperturbed Antonov operator $\mathcal A_0.$ In this section  we prove that the absolutely continuous spectrum of $\mathcal A$ coincides with its essential spectrum. Further, we prove that the part of the singular spectrum that is embedded  in the absolutely continuous spectrum, if any, is contained in a closed set of measure zero. 
  We first obtain some important results that we use later. We define the relevant operators in  appropriate spaces of H\"older continuous functions. We find it convenient to use spaces of H\"older continuous functions by the following reason. The spectral representation of the unperturbed Antonov operator is obtained by means of the energy-angle variables. The perturbation term in the Antonov operator in energy angle variables  is an integral operator (see \eqref{2.21}, \eqref{2.81a} and \eqref{2.81b}). However, the kernel of this integral operator is singular at the minimum  of the Hamiltonian for the motion in the potential of the steady state, the so called elliptic point of the Hamiltonian. This singularity has to be carefully analyzed, and it is possible to do this in spaces of  H\"older continuous functions, that  require less regularity than other spaces, like Sobolev spaces for example. 

 \subsection{The operator $\mathcal B R_0(z)$}\label{opg}
\label{sec2}
We denote by $G(z)$ the following operator
\beq\label{s.1}
G(z):= \frac{1}{4\pi  v(\theta, E)| \varphi'(E)|}\,\mathcal B R_0(z), \qquad z \in \rho(\mathcal A_0).
\ene
 We have.
\begin{theorem}\label{theo.s1}
For any $ 0 < \alpha_1 \leq 1,$ and $ 0< \alpha_2 < 1/2$  and $ z \in   \rho(\mathcal A_0)$ the following holds.
\begin{enumerate}
\item[{\rm (a)}]
    $G(z) \in  B({C}_{\alpha_1}(I_0, L^2_{\rm o}(S_{1})), \hat{C}_{\alpha_2}(I_0, L^2_{\rm o}(S_{1}))).$  
 \item[{\rm (b)}]
  The function  $ z \in  \rho(\mathcal A_0) \rightarrow G(z)$ is analytic in the operator norm  in \linebreak 
  $B({C}_{\alpha_1}(I_0, L^2_{\rm o}(S_{1})), \hat{C}_{{\alpha_2}}(I_0, L^2_{\rm o}(S_{1}))).$  
  
 \item[{\rm (c)}]
\beq\label{s.1b}
\ds \lim_{|{\rm Im}z|\to \infty}\|G(z)\|_{\ds B({C}_{\alpha_1}(I_0, L^2_{\rm o}(S_{1})),\hat{C}_{\alpha_2}(I_0, L^2_{\rm o}(S_{1})))}=0.
\ene
\item[{\rm (d)}]
For $ z \in \rho(\mathcal A_0),$ the operator $G(z)$ is compact from ${C}_{\alpha_1}(I_0, L^2_{\rm o}(S_{1}))$ into
$\hat{C}_{\alpha_2}(I_0, L^2_{\rm o}(S_{1})).$
\end{enumerate}
\end{theorem}
\begin{proof} By \eqref{2.21}, \eqref{2.75}, \eqref{2.76}, and \eqref{s.1} for $ f \in {C}_{\alpha_1}(I_0,L^2_{\rm o}(S_{1}) ),$
\beq\label{s.2}
(G(z) f)(\theta,E)=   g(x,E,z), \qquad x=x(\theta, E),
 \ene
 where
 \beq\label{s.3}
 g(x,E,z):=2 \int_{U_0(x)}^{E_0}  \sum_{l =1}^\infty \,d\lambda \,\sin(4\pi l \theta(x,\lambda)) \left( \frac{(4\pi l)^2}{T(\lambda)^2} -z\right)^{-1}  f_l(\lambda),
 \ene
 with
 \beq\label{s.4}
  f_l(E):= \int_0^1\sin(4\pi l \theta_1) f(E,\theta_1) d\theta_1.
  \ene
  It follows that,
\beq\label{s.5}
\left| (G(z) f)(\theta, E) \right| \leq C \|f\|_{{C}_{\alpha_1}(I_0, L^2_{\rm o}(S_{1}))}, \qquad \theta \in S_1, E\in I_0,
\ene
where we used that as $ z \in \rho(\mathcal A_0),$ we have   $|\frac{(4\pi l)^2}{T(\lambda)^2} -z| \geq \delta >0,$ for $l \in \mathcal N,$ and $ \lambda \in I_0.$   Suppose that $ U_0(x_1) \leq U_0(x_2).$ We denote,
\beq\label{s.5b}\begin{array}{c}
 D(x_1, x_2)=2 \int_{U_0(x_2)}^{E_ 0}  \sum_{l =1}^\infty \,d\lambda \,[\sin(4\pi l \theta(x_2,\lambda))-[\sin(4\pi l \theta(x_1,\lambda))] \\\left( \frac{(4\pi l)^2}{T(\lambda)^2} -z\right)^{-1}  f_l(\lambda).
\end{array}
\ene
For any  $ 0 \leq  \delta \leq 1$ there is a constant $C$ such that,
\beq\label{s.5fxx}
\left|\sin(4\pi l \theta(x_2,\lambda))-\sin(4\pi l \theta(x_1,\lambda))  \right| \leq C  l^\delta   \left|\theta(x_2,\lambda)- \theta(x_1,\lambda)\right|^\delta.
\ene
Hence, by \eqref{2.4}, \eqref{2.36}, and \eqref{s.5fxx}, for $ 0 < \delta <1,$ we have,
\beq\label{s.5f}
|D(x_1,x_2)| \leq C  \int_{U_0(x_2)}^{E_ 0}    d\lambda \left|\int_{x_1}^{x_2} 
\frac{1}{\sqrt{\lambda- U_0(y)}} \,dy\right|^\delta \|f\|_{{C}_{\alpha_1}(I_0, L^2_{\rm o}(S_{1}))}.  
\ene
Further,
\beq\label{s.5fx}\begin{array}{c}
|D(x_1,x_2)| \leq C  \int_{U_0(x_2)}^{E_ 0}    d\lambda \frac{1}{(\lambda- U_0(x_2))^{\delta/2}} |x_1-x_2|^\delta
\|f\|_{{C}_{\alpha_1}(I_0, L^2_{\rm o}(S_{1}))}  \leq   \\[10pt] C   |x_1-x_2|^\delta
\|f\|_{{C}_{\alpha_1}(I_0,  L^2_{\rm o}(S_{1}))}.
\end{array}
\ene
 Moreover,  for $ x \geq 0,$ if $ E= U_0(x),$ we have that $ x=x_+(E).$  Then,  it follows  from  \eqref{1.20}
 \beq\label{s.8}\begin{array}{c}
 \left | \int_{U_0(x_1)}^{U_0(x_2)} 2 \sum_{l =1}^\infty \,d\lambda \,\sin(4\pi l \theta(x_1,\lambda)) 
 \left( \frac{(4\pi l)^2}{T(\lambda)^2} -z\right)^{-1}  f_l(\lambda)\right | \\
  \leq |x_1-x_2|  
  \|f\|_{{C}_{\alpha_1}(I_0, L^2_{\rm o}(S_{1}))},
  \end{array}
 \ene
 where we changed the variable of integration to $x_\pm(\lambda).$  Note that as  $ U_0$ is even we can always assume that $x_1\geq 0$ and $x_2 \geq 0$ in the left-hand side of \eqref{s.8}.
 Then, by \eqref{s.2}, \eqref{s.3},  \eqref{s.5b}, \eqref{s.5fx}, \eqref{s.8}, and \eqref{der.1}, for any $ 0 <\delta < 1,$
 \beq\label{s.9} 
 \left|   G(\theta, E_1)-G(\theta, E_2)\right | \leq C |E_1-E_2|^{\delta/2} \|f\|_{{C}_{\alpha_1}(I_0, L^2_{\rm o}(S_{1}))}.
 \ene
 By \eqref{s.5} and \eqref{s.9} $G(z)$ is bounded from ${C}_{\alpha_1}(I_0,L^2_{\rm o}(S_{1}))$ into
 ${C}_{\delta/2}(I_0,L^2_{\rm o}(S_{1})).$ Let us take $ \delta/2 > \alpha_2.$ Since by  Proposition 
 ~\eqref{propa.1}   ${C}_{\delta/2}(I_0,L^2_{\rm o}(S_{1})) \subset \hat{C}_{\alpha_2}(0,L^2_{\rm o}(S_{1})),$ and by \eqref{a.16b} the imbedding is continuous, we  have that  $G(z)$ is bounded from   ${C}_{\alpha_1}(I_0, L^2_{\rm o}(S_{1})),$ into  $\hat{C}
_{\alpha_2}(I_0, L^2_{\rm o}(S_{1})).$ This proves (a). Item (b) is proved as in the proof of (a) differentiating \eqref{s.3} under the integral sign  with respect  to $z.$
 We prove (c) arguing as in the proof of (a) and using that for any $ 1 > \varepsilon >0$
 $$
 \left| \left( \frac{(4\pi l)^2}{T(\lambda)^2} -z\right)^{-1}  \right| \leq  \left| \left( \frac{(4\pi l)^2}{T(\lambda)^2} -z\right)^{-1}  \right|^{1-\varepsilon} \, \frac{1}{|{\rm im}\, z|^\varepsilon}.
 $$
 Let us now  prove (d).We define the operator $G_N(z)$ as follows.
 
\beq\label{s.10}
(G_N(z) f)(\theta, E)=  h_N(x,E,z), \qquad x=x(\theta, E),
 \ene
 where
 \beq\label{s.11}
 h_N(x,E,z):=2\int_{U_0(x)}^{E_0}  \sum_{l =N+1}^\infty \,d\lambda \,\sin(4\pi l \theta(x,\lambda)) \left( \frac{(4\pi l)^2}{T(\lambda)^2} -z\right)^{-1}  f_l(\lambda),
 \ene
 with $f_l(E)$ as in \eqref{s.4}. Arguing as in the proof that $G(z)$ is bounded from  $C_{\alpha_1}(I_0, L^2_{\rm o}(S_{1}))$ into   $\hat{C}_{\alpha_2}(I_0, L^2_{\rm o}(S_{1})),$ we prove that
$G_N(z)$ is bounded from  $C_{\alpha_1}(I_0, L^2_{\rm o}(S_{1}))$ into   $\hat{C}_{\alpha_2}(I_0, L^2_{\rm o}(S_{1})),$ and that,
\beq\label{s.12}
\ds \lim_{N\to \infty} \| G_N(z) \|_{\ds B( {C}_{\alpha_1}(I_0, L^2_{\rm o}(S_{1})),\hat{C}_{\alpha_2}(I_0, L^2_{\rm o}(S_{1})))}=0.
\ene
Then, it is enough to prove that the operator $K_N(z)$ is compact from  $C_{\alpha_1}(I_0, L^2_{\rm o}(S_{1}))$ into   $\hat{C}_{\alpha_2}(I_0, L^2_{\rm o}(S_{1})),$ where,
\beq\label{s.13}
\left(K_N(z) f\right)(\theta, E):= k_N(x,E,z), \qquad x=x(\theta, E),
 \ene
 with
 \beq\label{s.14}
 k_N(x,E,z):=2 \int_{U_0(x)}^{E_0}  \sum_{l =1}^N \,d\lambda \,\sin(4\pi l \theta(x,\lambda)) \left( \frac{(4\pi l)^2}{T(\lambda)^2} -z\right)^{-1}  f_l(\lambda),
 \ene
 with $f_l(E)$ as in \eqref{s.4}. For this purpose, let $f^{(q)}, q=1, \dots,$ be a bounded sequence in ${C}_{\alpha_1}(I_0, L^2_{\rm o}(S_{1})).$ Then, the vector valued functions $(f_1^{(q)}(E),\dots, f_N^{(q)}(E)  ), q=1, \dots$ are uniformly bounded and 
 uniformly equicontinuous functions of $ E \in I_0 $ into $\mathbb C^N$  and by the Arzel\`a--Ascoli theorem,
 Theorem 7.25 of \cite{ru}, it has a subsequence, that we denote also by   $(f_1^{(q)}(E),\dots, f_N^{(q)}(E)  ), q=1, \dots$ that converges uniformly  to a continuous vector valued function function $(g_1(E),\dots, g_N(E)).$ As in the proof that $G(z)$ 
  is bounded from ${C}_{\alpha_1}(I_0, L^2_{\rm o}(S_{1}))$ into   $\hat{C}_{\alpha_2}(I_0, L^2_{\rm o}(S_{1})),$ we prove that
 \beq\label{s.15}
 \begin{array}{c}
\lim_{q\to \infty}\left\| K_N f^{(q)}-  2 \int_{U_0(x)}^{E_0}  \sum_{l =1}^N \,d\lambda\ds \,\sin(4\pi l \theta(x,\lambda)) \left( \frac{(4\pi l)^2}{T(\lambda)^2} -z\right)^{-1} \right. \\ \left.g_l(\lambda)  \right\|_{\hat{C}_{\alpha_2}(I_0, L^2_{\rm o}(S_{1}))} =0.
\end{array}
\ene
This proves that   $K_N f^{(q)}$  is convergent in $\hat{C}_{\alpha_2}(I_0, L^2_{\rm o}(S_{1})).$ This completes the proof of the compactness of $G(z).$

\end{proof}
We denote
\beq\label{s.16}
\Gamma:= \sigma(\mathcal A_0)\setminus  \cup_{l \in \mathbb N }
\left\{  \frac{(4\pi l)^2}{T(E_0)^2 }
 , \frac{(4\pi l)^2}{T(E_{\rm min})^2 }
\right\}.
\ene

In the next theorem we prove that $G(z)$ extends continuously to $\Gamma.$

\begin{theorem} \label{theo.s2}
For any $ 0 < \alpha_1, \alpha_2 < 1/4$  the following holds.
\begin{enumerate}
\item[{\rm (a)}]
 For any $\gamma \in \Gamma$ the  limits
\beq\label{s.17}
G(\gamma \pm i0):= \lim_{\varepsilon \downarrow 0} G(\gamma\pm i\varepsilon)
\ene
exist in the uniform operator topology   in  $B({C}_{\alpha_1}(I_0, L^2_{\rm o}(S_{1})), \hat{C}_{\alpha_2}(I_0, L^2_{\rm o}(S_{1}))).$  
Moreover, the convergence is uniform for $\gamma$  in  any set $\Gamma_K$  that is a finite union of bounded and closed intervals contained in $\Gamma.$
\item[{\rm (b)}]
The functions 
\beq\label{s.18}
 G_\pm(\gamma):=\left\{\begin{array}{c} G(\gamma), \qquad \gamma \in \rho( \mathcal A_0),\\
 G(\gamma \pm i0), \qquad \gamma \in \Gamma,
 \end{array}\right.
 \ene
defined for  $ \gamma  \in \rho(\mathcal A_0) \cup \Gamma $ with values in  $B({C}_{\alpha_1}(I_0, L^2_{\rm o}(S_{1})), \hat{C}_{\alpha_2}(I_0, L^2_{\rm o}(S_{1})))$   are analytic for $ \gamma \in \rho(\mathcal A_0)$  and are H\"older continuous with exponent $\alpha,$ for any 
 $0 < \alpha <  \alpha_1(1-4\alpha_2),$   in compact sets of
 \beq\label{s.18ba}
 {\mathbb C_\pm }\cup \Gamma_K.
 \ene
 \item[{\rm (c)}]
 The operators  $G_\pm(\gamma)$ are compact from  ${C}_{\alpha_1}(I_0, L^2_{\rm o}(S_{1}))$, into  $\hat{C}_{\alpha_2}(I_0, L^2_{\rm o}(S_{1}))$  
for $ \gamma \in \rho(\mathcal A_0) \cup \Gamma.$
 \end{enumerate}
\end{theorem}
\begin{proof} Let $\Gamma_K$ be any 
finite union of bounded and closed intervals contained in $\Gamma.$  We define,
\beq\label{s.19}
J:= \left\{ l \in \mathbb N :  \Gamma_K \cap \left(    \frac{(4\pi l)^2}{T(E_0)^2 }
 , \frac{(4\pi l)^2}{T(E_{\rm min})^2}\right) \neq\emptyset\right\}, \qquad \tilde{J}:= \mathbb N \setminus J.
\ene   
Note that $J$ is a finite set.

We decompose $G(z)$ as follows,
\beq\label{s.20}
G(z)= G_{\tilde{J}}(z)+ \sum_{l\in J}  G_l(z),
\ene
where for $f \in {{C}_{\alpha_1}(I_0, L^2_{\rm o}(S_{1}))} ,$ and with $f_l(E)$ as in \eqref{s.4},
\beq\label{s.20b}
(G_{\tilde{J}}(z)f)(\theta, E):=   g_{\tilde{J}}(x,z), \qquad x=x(\theta, E),
 \ene
 with
 \beq\label{s.21}
 g_{\tilde{J}}(x,z):= 2 \int_{U_0(x)}^{E_0}  \sum_{l \in \tilde{J}}\,d\lambda \,\sin(4\pi l \theta(x,\lambda)) \left( \frac{(4\pi l)^2}{T(\lambda)^2} -z\right)^{-1}  f_l(\lambda),
 \ene
and
\beq\label{s.22}
(G_{l}(z)f)(\theta, E):=   g_{l}(x,z), \qquad x=x(\theta, E),\qquad l \in J,
 \ene
 where
 \beq\label{s.23}
 g_{l}(x,z):= 2\int_{U_0(x)}^{E_0} d\lambda \,\sin(4\pi l \theta(x,\lambda)) \left( \frac{(4\pi l)^2}{T(\lambda)^2} -z\right)^{-1}  f_l(\lambda), l \in J.
 \ene
Note  that  the distance
$$
d:= {\rm dist}\left( \Gamma_K,  \ds \cup_{l \in \tilde{J} }
\left[  \frac{(4\pi l)^2}{T(E_0)^2 }
 , \frac{(4\pi l)^2}{T(E_{\rm min})^2 }
\right ]\right)
$$
is strictly positive. Hence  as in the proof of Theorem~\ref{theo.s1} we prove that the limits
\beq\label{s.24}
 \lim_{\varepsilon \downarrow 0}G_{\tilde{J}}(\gamma+i\varepsilon)=  \lim_{\varepsilon \downarrow 0}G_{\tilde{J}}(\gamma-i\varepsilon):= G_{\tilde{J}}(\gamma) , \gamma \in \Gamma_K,
\ene
exist in the uniform operator topology  in  $B({C}_{\alpha_1}(I_0, L^2_{\rm o}(S_{1})), \hat{C}_{\alpha_2}(I_0, L^2_{\rm o}(S_{1})))$  
 and that the converge is uniform for $\gamma \in \Gamma_K.$ 
 Moreover, $G_{\tilde{J}}(\gamma)$ is analytic for $\gamma \in  \rho(\mathcal A_0)\cup \Gamma_k.$ This proves (a)  and (b) for $G_{ \tilde J}(z).$ Then, it is enough to prove  (a) and (b) for each  one of the $G_l(z), l \in J.$ Recall that  when $ E=U_0(x)$ we have $ x=x_+(E),$ for $x \geq 0,$ and $x=x_-(E),$ for $ x \leq 0.$ Further, by \eqref{2.36} $\theta(x_+(E), E)= 1/2,$ and  $\theta(x_-(E), E)= 0.$ Then
denoting
\beq\label{s.25}
a_l(x,E)= \sin(4\pi l \theta(x,E)),
\ene
we have,
\beq\label{s.26}
a_l(x, U_0(x))=0.
\ene
Note that for each $ x \in [x_-(E_0), x_+(E_0)],$ the angle $\theta(x,E)$ is defined for $ E \in [ U_0(x), E_0].$ Then, 
by \eqref{s.26} we can extend $a_l(x, E)$ continuously to $ E \in [E_{\rm min}, U_0(x)]$ by zero, i.e.
 \beq\label{s.26b}
 a_l(x,E)=0,\qquad  E \in [E_{\rm min}, U_0(x)].
 \ene
  Hence, we can write \eqref{s.23} as follows,
\beq\label{s.27}
 g_{l}(x,z):= \int_{E_{\rm min}}^{E_0} d\lambda \,a_l(x,\lambda) \left( \frac{(4\pi l)^2}{T(\lambda)^2} -z\right)^{-1}  f_l(\lambda), \qquad  l \in J.
 \ene
We designate by 
 \beq\label{s.28}
 d_l:= {\rm dist}\left(\left[ \Gamma_K\cap  \left(  \frac{(4\pi l)^2}{T(E_0)^2} 
 , \frac{(4\pi l)^2}{T(E_{\rm min})^2 }\right)\right],   \left\{  \frac{(4\pi l)^2}{T(E_0)^2 }
 , \frac{(4\pi l)^2}{T(E_{\rm min})^2 } \right \}  \right).
 \ene
 Clearly, $d_l >0.$ Let $ \varphi \in C^1(\mathbb R)$ satisfy  
 \beq\label{s.29}
 \varphi(y)= \left\{ \ds \begin{array}{l}
 0, y \leq   \frac{(4\pi l)^2}{T(E_0)^2} + \frac{d_l}{4},\\[5pt]
  1,  \frac{(4\pi l)^2}{T(E_0)^2} + \frac{d_l}{2} \leq y \leq \frac{(4\pi l)^2}{T(E_{\rm min})^2} - \frac{d_l}{2}, \\[5pt]
 0,  y \geq \frac{(4\pi l)^2}{T(E_{\rm min})^2} - \frac{d_l}{4}.
 \end{array}\right.
 \ene
  We decompose $G_l(z)$ as follows
 \beq\label{s.30}
 G_l(z)= G_{l,1}(z)+G_{l,2}(z), 
 \ene
 with
\beq\label{s.30b}
(G_{l,1}(z)f)(\theta,E)= g_{l,1}(x,z), \qquad x= x(\theta,E),
\ene
with
 \beq\label{s.32}\begin{array}{r}
 g_{l,1}(x,z):=2 \ds\int_{E_{\rm min}}^{E_0}\,a_l(x,\lambda) \ds\left( \frac{(4\pi l)^2}{T(\lambda)^2} -z\right)^{-1} \\[10pt]  \varphi\left( \frac{(4\pi l)^2}{T(\lambda)^2} \right) f_l(\lambda) \,  d\lambda , \qquad    l \in J,
 \end{array}
 \ene
 and
\beq\label{s.34b}
(G_{l,2}(z)f)(\theta,E)= g_{l,2}(x,z), \qquad x= x(\theta,E),
\ene
where
 \beq\label{s.34}\begin{array}{r}
 g_{l,2}(x,z):=2 \ds\int_{E_{\rm min}}^{E_0} \,a_l(x,\lambda) \ds\left( \frac{(4\pi l)^2}{T(\lambda)^2} -z\right)^{-1} \\[14pt]  \left[1-  \varphi\left( \frac{(4\pi l)^2}{T(\lambda)^2} \right)\right] f_l(\lambda) \,d\lambda  , \qquad l \in J.
 \end{array}
 \ene
 Note that in the support of the integrand in the right-hand side of \eqref{s.34} $ |\frac{(4\pi l)^2}{T(\lambda)^2} -\gamma|\geq \delta_l >0,$ for $\gamma \in \Gamma_K.$ Hence, as in the proof of Theorem~\ref{theo.s1} we prove that the  limits 
 \beq\label{s.35}
 \lim_{\varepsilon \downarrow 0}G_{l,2}(\gamma+i\varepsilon)=  \lim_{\varepsilon \downarrow 0}G_{l,2}(\gamma-i\varepsilon):= G_{l,2}(\gamma),\qquad  \gamma \in \Gamma_K, l \in J,
 \ene 
 exist in the uniform operator topology in $B({C}_{\alpha_1}(I_0, L^2_{\rm o}(S_{1})), \hat{C}_{\alpha_2}(I_0, L^2_{\rm o}(S_{1})))$  
 and that the converge is uniform for $\gamma \in \Gamma_K.$
  Moreover,   $ G_{l,2}(\gamma)$ analytic for $\gamma$  in  $ \rho(\mathcal A_0)\cup\Gamma_K.$  This proves (a) and (b) for $ G_{l,2}(\gamma).$  Finally, it only remains to study the term $G_{l,1}(z)$ that is the one where the singularity lies.
 Since  the period $T(E)$ is strictly increasing, in the support of the integrand in the right-hand side of \eqref{s.32} $\lambda \geq E_{\rm min}+ \rho_l,$ for some $\rho_l >0,$ such that $ E_{\rm min}+\rho_l < E_0.$ Then, we can write \eqref{s.32} as
  \beq\label{s.37}\begin{array}{c}
 g_{l,1}(x,z)= \ds\int_{E_{\rm min}+ \rho_l}^{E_0} \,a_l(x,\lambda) \left( \frac{(4\pi l)^2}{T(\lambda)^2} -z\right)^{-1} \\[11pt]  \varphi\left( \frac{(4\pi l)^2}{T(\lambda)^2}\right) f_l(\lambda) \,d\lambda  , \qquad l \in J.
 \end{array}
 \ene
 Recalling \eqref{sp.1}, \eqref{sp.2}, \eqref{sp.2b}, \eqref{s.42bc}, and \eqref{s.43}  and changing the variable of integration from $E$ to $\beta$ in \eqref{s.37} we get,
    \beq\label{s.340}\begin{array}{l}
 g_{l,1}(x,z)= \ds\int_{\beta_{\rm min}}^{\beta_{\rm max}-\tilde{\rho}_l}  \, p_l(\beta) \,a_l(x,\ \lambda_l(\beta) ) \left( \beta -z\right)^{-1} \\[10pt]  \varphi\left( \beta \right)  f_l(\lambda_l(\beta))\, d\beta, \qquad l \in J,
 \end{array}
 \ene
 for some $\tilde{\rho}_l>0.$
  By \eqref{2.42abc} and
 as we defined $a_l(x,E)=0,$ for  $ E \in [E_{\rm min},E_0],$ there is a constant $C$ such that,
 \beq\label{s.36b}
 \left| a_l(x,E_2)- a_l(x,E_1)  \right| \leq C   \sqrt{|E_1- E_2|},  E_1,E_2 \in [E_{\rm min}, E_0], x \in [x_-(E_0), x_+(E_0)].
 \ene
  It follows from Proposition~\ref{last}  that  for $ 0<\alpha <1/2,$  there is a constant $C$ such that,
  \beq\label{s.38}
  |a_l(x_2,E)- a_l(x_1,E)| \leq  C |x_1-x_2|^{\alpha},  x_1,x_2 \in [x_-(E_0), x_+(E_0)], E \in [E_{\rm min}+ \delta, E_0].
  \ene
Note that since $T(E)$ is strictly increasing and   continuously differentiable  for $E \in (U_0(0), E_0],$ 
  by \eqref{2.4} and \eqref{s.43} $p_l(\beta)$ is bounded for $ \beta \in [\beta_{\rm min}, \beta_{\rm max}-\tilde{\rho}_l].$ Hence,
  \beq\label{s.44}
|\lambda_l(\beta_{1})- \lambda_l(\beta_{2})|\leq C |\beta_{1}-\beta_{2}|, \qquad \beta_{1}, \beta_{ 2} \in [ \beta_{l,\min}, \beta_{l,\max}-\tilde{\rho}_l].
\ene
 By item (b) in Proposition~\ref{propder} $ T(E)\in C^2([E_{\rm min}+\delta, E_0]) $ for $\delta >0$ with 
 $ E_{\rm min}+\delta < E_0.$ Hence, by \eqref{s.43} and \eqref{s.44},
\beq\label{s.46}
|p_l(\lambda(\beta_{1})- p_l(\lambda(\beta_{2})| \leq C |\beta_{1}-\beta_{2}|, \qquad \beta_{1}, \beta_{2} \in [ \beta_{\min}, \beta_{\max}-\tilde{\rho}_l].
\ene
With slight abuse of notation let us denote,
\beq\label{s.46b}
g_{l,1}(\theta,E,z):= g_{l,1}(x(\theta, E),z), \qquad l \in J.
\ene
We designate,
\beq\label{s.46c}
b_l(\theta, E, \beta):=  p_l(\beta) \,a_l(x(\theta,E),\ \lambda_l(\beta) )  \varphi\left( \beta \right), \qquad l \in J.
\ene
Clearly,
\beq\label{s.46d}
|b_l(\theta, E, \beta)| \leq C, \\ \qquad \theta \in S_{1}, E\in I_0, \beta 
\in [\beta_{\rm min}, \beta_{{\rm max}-\tilde{\rho}_l}].
\ene

Note that by \eqref{s.29},
\beq\label{s.46e}
b_l(\theta, E, \beta_{\rm min})= b_l(\theta, E,\beta_{{\rm max}-\tilde{\rho}_l})=0.
\ene
Moreover, by \eqref{s.36b}, \eqref{s.38}, \eqref{s.44}, \eqref{s.46}, and \eqref{der.1}, for any $0 < \alpha < 1/4, $
\beq\label{s.46f}\begin{array}{c}
\left |  b_l(\theta, E_1, \beta_1) -b_l(\theta, E_2, \beta_2) \right |\leq
 C \left ( |E_1-E_2|+ |\beta_1-\beta_2|\right )^{\alpha},\\[5pt] \qquad \theta \in S_{1}, E_1,E_2 \in I_0, \beta_1,\beta_2 
\in [\beta_{\rm min}, \beta_{{\rm max}-\tilde{\rho}_l}].
\end{array}
\ene
Further, by \eqref{s.340} and \eqref{s.46c}
    \beq\label{s.46g}\begin{array}{l}
 g_{l,1}(\theta, E ,z):= \ds\int_{\beta_{\rm min}}^{\beta_{\rm max}-\tilde{\rho}_l}  \,b_l(\theta, E, \beta) ) \left( \beta -z\right)^{-1} \,    f_l(\lambda_l(\beta))\,  d\beta, \qquad l \in J.
 \end{array}
 \ene
Then, by \eqref{s.30b}, \eqref{s.44},  \eqref{s.46d}- \eqref{s.46g}, and Theorem 5 and Remark 6 in  Chapter  IV of 
\cite{kur} (see also Theorem 6 in Section 2 of  Chapter 1 of \cite{ya}),   the following limits exist,
\beq\label{s.47} 
(G_{l,1}(\gamma\pm i0)f)(\theta, E):= \lim_{\varepsilon \downarrow 0} G_{l,1}(\gamma\pm i\varepsilon)(\theta, E),\qquad \gamma \in \Gamma_K, \theta \in S_{1},
\ene
for $ f \in C_{\alpha_1}(I_0, L^2_{\rm o}(S_{1})),$ with $ 0 < \alpha_1 < 1/4,$
 and they are given by
 \beq\label{s.48}\begin{array}{r}
  (G_{l,1}(\gamma\pm i0)f)(\theta, E):= {\rm P.V.}\int_{\beta_{\rm min}}^{\beta_{\rm max}-\tilde{\rho}_l }\,    d\beta \, b_l(\theta, E, \beta)\left( \beta -\gamma\right)^{-1}  f_l(\lambda_l(\beta))\, \pm  \\[5pt] \pi  i \, b_l(\theta, E, \gamma)   f_l(\lambda_l(\gamma)) , \qquad  \gamma \in \Gamma_K, l \in J,
 \end{array}
 \ene
 where ${\rm P.V.}$ denotes the principal value of the integral. Further,  the converge is uniform for $\gamma \in \Gamma_K.$ 
 We denote,
 \beq\label{s.48b}
 G_{l,1,\pm}(\gamma):=\left\{\begin{array}{c} G_{l,1}(\gamma), \qquad \gamma \in \rho( \mathcal A_0),\\
 G_{l,1}(\gamma \pm i0), \qquad \gamma \in \Gamma.
 \end{array}\right.
 \ene
 Using again  Theorem 5 and Remark 6 in  Chapter  IV of 
\cite{kur} (see also Theorem 6 in Section  2 of Chapter 1 of \cite{ya} and also Section 19 of Chapter 2 of \cite{mu})  we obtain
that for every compact set $  \mathcal K_\pm \in \mathbb C_\pm \cup \Gamma,$ 
 \beq\label{s.48z}
 \left| (G_{l,1,\pm}(\gamma)f)(\theta,E) \right| \leq C\ds \|f\|_{\ds C_{\alpha_1}(I_0, L^2_{\rm o}(S_{1}))}, \qquad \gamma \in \mathcal K_\pm ,
 \ene
and
\beq\label{s.48zz}\begin{array}{c}
 \left| (G_{l,1,\pm}(\gamma_1)f)(\theta,E)- (G_{l,1,\pm}(\gamma_2)f)(\theta,E) \right| \leq \ds C \ds |\gamma_1-\ds \gamma_2|^{\alpha_1}   \|f\|_{\ds  C_{\alpha_1}(I_0,  L^2_{\rm o}(S_{1}))},\\[8pt] \gamma_1, \gamma_2  \in \mathcal K_\pm.
 \end{array}
 \ene
For the clarity in the estimates below we designate,
\beq\label{s.48c}
h_{l,E}(\theta, \beta):= b_l(\theta, E, \beta), \qquad  \theta \in S_{1}, E  \in I_0, \beta
\in [\beta_{\rm min}, \beta_{{\rm max}-\tilde{\rho}_l}].
\ene
 Then, by \eqref{s.30b} and \eqref{s.46g}
 \beq\label{s.48d}
  (G_{l,1,\pm}(\gamma)f)(\theta, E)=  \ds\int_{\beta_{\rm min}}^{\beta_{\rm max}-\tilde{\rho}_l}  \,h_{l,E}(\theta, \beta) ) \left( \beta -z\right)^{-1} \,  f_l(\lambda_l(\beta))\, d\beta, \qquad l \in J,  \gamma \in \mathbb C_\pm.
  \ene
  Let us take $ 0< \alpha_1, \alpha_2 < 1/4,$  and  $ \alpha_3 < 1/4$ that satisfy $ \alpha_1 < \alpha_3$ and
 $ \alpha_2 < \alpha_3.$   Take $0 < \delta <1$ such that $ \alpha_2= \delta \alpha_3.$ By \eqref{s.46f} with $ \alpha= \alpha_3,$
 $$
  \begin{array}{l}
 \left| b_l(\theta, E_2,\beta_2)-  b_l(\theta, E_1,\beta_2) - ( b_l(\theta, E_2,\beta_1)-  b_l(\theta, E_1,\beta_1) )\right|   
 \leq\\[8pt] C |E_2-E_1|^{\delta \alpha_3}\,  |\beta_2-\beta_1|^{(1-\delta)\alpha_3}, \qquad \theta \in S_{1}, E_1,E_2 \in I_0, \beta_1,\beta_2 
\in [\beta_{\rm min}, \beta_{{\rm max}-\tilde{\rho}_l}],
\end{array}
$$
 and then,
  \beq\label{s.48e}\begin{array}{l}
 \left| b_l(\theta, E_2,\beta_2)-  b_l(\theta, E_1,\beta_2) - ( b_l(\theta, E_2,\beta_1)-  b_l(\theta, E_1,\beta_1) )\right|   
 \leq\\[8pt] C |E_2-E_1|^{\alpha_2}\,  |\beta_2-\beta_1|^{\alpha_3(1- \alpha_2/\alpha_3)}, \theta \in S_{1}, E_1,E_2 \in I_0, \beta_1,\beta_2 
\in [\beta_{\rm min}, \beta_{{\rm max}-\tilde{\rho}_l}].
\end{array}
\ene
Then, by \eqref{s.46d}, \eqref{s.46f}, and \eqref{s.48e}, for $ f \in {C}_{\alpha_1}(I_0, L^2_{\rm o}(S_{1}))$ we have,
\beq\label{s.48f}
\| (h_{l,E_2}-h_{l,E_1}) f_l\|_{\ds {C}_{\alpha_1(1- \alpha_2/\alpha_3)}(I_0, L^2_{\rm o}(S_{1}))}\leq C |E_2-E_1|^{\alpha_2} \|f\|_{\ds {C}_{\alpha_1}(I_0, L^2_{\rm o}(S_{1}))},  E_1,E_2 \in I_0.
\ene
Using \eqref{s.48d}, \eqref{s.48f}, and applying  Theorem 5 and Remark 6 in  Chapter  IV of 
\cite{kur} (see also Theorem 6  in Section 2 of Chapter 1 of \cite{ya}), and  Section 19 of Chapter 2 of \cite{mu} to the function
$(h_{E_2}(\theta,\beta)-h_{E_1}(\theta,\beta)) f_l(\lambda_l(\beta))$ for each fixed $E_1,E_2$ we get,
\beq\label{s.48g}\begin{array}{c}
\left|  (G_{l,1,\pm}(\gamma)f)(\theta, E_1)-(G_{l,1,\pm}(\gamma)f)(\theta, E_2) \right| \leq\ds C  |E_2-E_1|^{\alpha_2}
\|f\|_{ {C}_{\alpha_1}(I_0, L^2_{\rm o}(S_{1}))},\\[5pt] \theta \in S_{1},  E_1,E_2 \in I_0, \gamma \in \mathcal K_\pm,
\end{array}
 \ene
 and
 \beq\label{s48.exxz}\begin{array}{c}
   \left | (G_{l,1,\pm}(\gamma_1)f)(\theta, E_1)-(G_{l,1,\pm}(\gamma_1)f)(\theta, E_2)-((G_{l,1,\pm}(\gamma_2)f)(\theta, E_1)-(G_{l,1,\pm}(\gamma_2)f)(\theta, E_2))\right| \\[7pt] \leq \ds C| E_1-E_2|^{\alpha_2} |\gamma_1-\gamma_2|^{\alpha_1(1-\alpha_2/\alpha_3)}
\| f\|_{\ds {C}_{\alpha_1}(I_0, L^2_{\rm o}(S_{1}))},\\[5pt] \theta \in S_{1}, E_1,E_2 \in I_0, \gamma_1,\gamma_2 \in \mathcal K_\pm.  
 \end{array}
 \ene
 Equations \eqref{s.48z} and \eqref{s.48e} imply,
 \beq\label{s.48fa}
 \| (G_{l,1,\pm}(\gamma)f\|_{\ds {C}_{\alpha_2}(I_0, L^2_{\rm o}(S_{1}))} \leq C \ds \|f\|_{\ds {C}_{\alpha_1}(I_0, L^2_{\rm o}(S_{1}))}, \qquad\gamma \in \rho(\mathcal A_0)\cup \Gamma,
 \ene
 with the constant $C$ uniform in compact sets of $\mathcal K_\pm.$ From \eqref{s.48fa} it follows that $G_{l,1,\pm}(\gamma)
 \in B({C}_{\alpha_1}(I_0, L^2_{\rm o}(S_{1})), {C}_{\alpha_2}(I_0, L^2_{\rm o}(S_{1}))).$ Moreover, taking $ 1/4 >\tilde{\alpha}_2  > \alpha_2,$ we also have, $G_{l,1,\pm}(\gamma)
 \in B({C}_{\alpha_1}(I_0, L^2_{\rm o}(S_{1})), {C}_{\tilde{\alpha}_2}(I_0, L^2_{\rm o}(S_{1}))).$ By  Proposition  ~\eqref{propa.1}   ${C}_{\tilde{\alpha}_2}(I_0,L^2_{\rm o}(S_{1})) \subset \hat{C}_{\alpha_2}(0,L^2_{\rm o}(S_{1})),$ and by \eqref{a.16b} the imbedding is continuous. Then,  $G _{l,1,\pm}(\gamma)
 \in B({C}_{\alpha_1}(I_0, L^2_{\rm o}(S_{1})), \hat{C}_{\alpha_2}(I_0, L^2_{\rm o}(S_{1}))).$
  As we already  know that the range of $G_{l,1,\pm}(\gamma)$ is contained in  $\hat{C}_{\alpha_2}(0,L^2_{\rm o}(S_{1})),$ equations \eqref{s.48zz} and \eqref{s48.exxz} imply that  $ G_{l,1,\pm}(\gamma)$ is H\"older continuous for  $\gamma $  in compact sets of $\mathbb C_\pm\cup \Gamma,$ with values in  $ B({C}_{\alpha_1}(I_0, L^2_{\rm o}(S_{1})), \hat{C}_{\alpha_2}(I_0, L^2_{\rm o}(S_{1}))),$
 and with exponent  $\alpha_1(1-\alpha_2/\alpha_3).$ For $ \alpha_1$ and $\alpha_2$ fixed, the quantity $\alpha_1(1-\alpha_2/\alpha_3),$ is increasing in 
 $\alpha_3$ and the maximum at $ \alpha_3=1/4$ is equal to $\alpha_1(1- 4 \alpha_2).$ Then, given any $ \alpha <\alpha_1(1- 4 \alpha_2)$ there is an $ \alpha_3 < 1/4,$ with  $ \alpha_1 < \alpha_3,$ and    $ \alpha_2 < \alpha_3,$ such that  $ \alpha <  \alpha_1(1-\alpha_2/\alpha_3).$ This completes the proof of (a) and  (b) for  $G_{l,1}(z).$
 Finally, the fact that  the operators  $G_\pm(\gamma)$ are compact in   $ B({C}_{\alpha_1}(I_0, L^2_{\rm o}(S_{1})), \hat{C}_{\alpha_2}(I_0, L^2_{\rm o}(S_{1}))) $
   for $ \gamma \in \rho(\mathcal A_0) \cup \Gamma,$ follows from Theorem~\ref{theo.s1}, recalling that the limit in operator norm of compact operators is compact. 
 This completes the proof of the theorem.
 \end{proof}

By \eqref{s.1}
\beq\label{nueva}
\mathcal B R_0(z)=  4\pi  v(\theta, E)| \varphi'(E)|\, G(z),\qquad  z \in \rho(\mathcal A_0).
\ene
 In the next theorem we obtain our results on the operator $\mathcal B R_0(z).$
 \begin{theorem} \label{theo.s2b}
Let $\alpha_1$ satisfy, $0< \alpha_1< 1/4.$  For the polytrope \eqref{1.12}  with $ k >1$ let $\alpha_2$ satisfy  $0 < \alpha_2 < k-1$ and $ \alpha_2 < 1/4.$ Further, for the polytrope  \eqref{1.12}  with $ k =1$ and for the King steady state \eqref{1.12b} let $\alpha_2$ satisfy $0 < \alpha_2 < 1/4.$  Then, the following holds.
 \begin{enumerate}
\item[{\rm (a)}]
 For any $\gamma \in \Gamma$ the  limits
\beq\label{s.17b}
(\mathcal BR_0)(\gamma \pm i0):= \lim_{\varepsilon \downarrow 0}   4\pi  v(\theta, E)| \varphi'(E)| G(\gamma\pm i\varepsilon)=   4\pi  v(\theta, E)| \varphi'(E)| G_\pm(\gamma),
\ene
exist in the uniform operator topology   in  $B({C}_{\alpha_1}(I_0, L^2_{\rm o}(S_{1})), \hat{C}_{\alpha_2}(I_0, L^2_{\rm o}(S_{1}))).$  
Moreover, the convergence is uniform for $\gamma$  in  any set $\Gamma_K$  that is a finite union of bounded and closed intervals contained in $\Gamma.$
\item[{\rm (b)}]
The functions 
\beq\label{s.18bx}
 \mathcal (B R_0)  _\pm(\gamma):=\left\{\begin{array}{c}\mathcal B R_0(\gamma), \qquad \gamma \in \rho( \mathcal A_0),\\[5pt]
 \mathcal (B R_0)(\gamma \pm i0), \qquad \gamma \in \Gamma,
 \end{array}\right.
 \ene
defined for  $ \gamma  \in \rho(\mathcal A_0) \cup \Gamma $ with values in  $B({C}_{\alpha_1}(I_0, L^2_{\rm o}(S_{1})), \hat{C}_{\alpha_2}(I_0, L^2_{\rm o}(S_{1})))$   are analytic for $ \gamma \in \rho(\mathcal A_0)$  and are H\"older continuous with exponent $\alpha,$ for any 
 $0 < \alpha <  \alpha_1(1-4\alpha_2),$   in compact sets of
 $
 {\mathbb C_\pm }\cup \Gamma_K.
 $
 \item[{\rm (c)}]
 The operators  $\mathcal (B R_0)_\pm(\gamma)$ are compact from  ${C}_{\alpha_1}(I_0, L^2_{\rm o}(S_{1}))$, into  $\hat{C}_{\alpha_2}(I_0, L^2_{\rm o}(S_{1}))$  
for $ \gamma \in \rho(\mathcal A_0) \cup \Gamma.$
 \end{enumerate}
\end{theorem}
\begin{proof} Let us take  $\tilde{\alpha}_2 > \alpha_2$  that fulfills the conditions imposed to $\alpha_2.$ Then, by \eqref{der.1a} $ |\varphi' (E)| v \in B({C}_{\tilde{\alpha}_2}(I_0, L^2_{\rm o}(S_{1})).$  By  Proposition  ~\eqref{propa.1}   ${C}_{\tilde{\alpha_2}}(I_0,L^2_{\rm o}(S_{1})) \subset \hat{C}_{\alpha_2}(0,L^2_{\rm o}(S_{1})),$ and by \eqref{a.16b}  the imbedding is continuous. Then, $  |\varphi' (E)| v \in  B(C_{\tilde{\alpha}_2}(I_0, L^2_{\rm o}(S_{1})), \hat{C}_{\alpha_2}(I_0, L^2_{\rm o}(S_{1}))).$  
   Hence, the theorem  follows from \eqref{nueva}, and by  Theorem~\ref{theo.s2} where we replace $\alpha_2$ by $\tilde{\alpha}_2.$
\end{proof}
 In the following proposition we prove that $\left(\mathcal BR_0\right)_\pm(\gamma)$  have the eigenvalue $1$ for the same values of $\gamma,$ for $\gamma \in \Gamma.$

\begin{proposition}\label{prop2.3}
  For the polytrope \eqref{1.12}  with $ k >1$ let $\alpha $ satisfy  $0 < \alpha < k-1$ and $ \alpha < 1/4.$ Further, for the polytrope  \eqref{1.12}  with $ k =1$ and for the King  steady state \eqref{1.12b} let $\alpha $ satisfy $0 < \alpha  < 1/4.$   For  the operator $\left(\mathcal BR_0\right)_\pm(\gamma) \in B({\hat{C}_{\alpha}(I_0, L^2_{\rm o}(S_{1}))}) $ we define,
\beq\label{s.51}
\mathcal N_\pm:= \{  \gamma \in  \Gamma : 1\, \text{\rm is an eigenvalue of} \,\left(\mathcal BR_0\right)_\pm(\gamma)\}.
\ene
Then,
\beq\label{s.52}
\mathcal N_+= \mathcal N_-:= \mathcal N,
\ene
and $\mathcal N$ is independent of $\alpha$ as long as $\alpha$ fulfills the conditions imposed in the proposition.
\end{proposition}
\begin{proof} 
Suppose that for some  $f_\pm \in \hat{C}_{\alpha}(I_0, L^2_{\rm o}(S_{1}))$ and some $\gamma \in \Gamma$ we have
\beq\label{s.52xx}
f_\pm= \left(\mathcal BR_0\right)_\pm(\gamma)f_\pm.
\ene
Then, by \eqref{s.52xx}  and as $\mathcal B$ is selfadjoint,  for $\varepsilon >0,$
\beq\label{s.53}\begin{array}{l}
0= {\rm Im}\left( \mathcal BR_0(\gamma\pm i\varepsilon) f_\pm , R_0(\gamma\pm i \varepsilon)f_\pm\right)_{\mathcal H}= \\[5pt]  {\rm Im}\left([\mathcal BR_0(\gamma\pm i\varepsilon) - \left(\mathcal BR_0\right)_\pm(\gamma)] 
f_\pm, \right.  \left.  R_0(\gamma\pm i \varepsilon)f_\pm\right)_{\mathcal H}
+ {\rm Im}\left(f_\pm, R_0(\gamma\pm i \varepsilon)f_\pm\right)_{\mathcal H}.
\end{array}
\ene
By Theorem~\ref{theo.s2}  
\beq\label{2.53b}
| G(\gamma\pm i \varepsilon)f_\pm-G_\pm(\gamma)f_\pm| \leq C |\varepsilon|^{\rho},
\ene
with $ 0 < \rho <  \alpha(1-4\alpha).$
Further, by  \eqref{2.76}, \eqref{nueva}, and \eqref{2.53b}
\beq\label{s.55}\begin{array}{l}
\left| \left([\mathcal BR_0(\gamma\pm i\varepsilon) - \left(\mathcal BR_0\right)_\pm(\gamma)] 
f_\pm, R_0(\gamma\pm i \varepsilon)f_\pm\right)_{\mathcal H}\right|\leq \\[5pt]
C \varepsilon^{\rho} \sum_{l =1}^\infty  \int_{E_{\rm min}}^{E_0} dE   \ds \frac{1} {\left|\frac{(4\pi l)^2}{T(E)^2} -\gamma\right|+\varepsilon } |(f_\pm)_l(E)|,
\end{array}
 \ene
where $(f_\pm)_l(E)$ is defined as in \eqref{s.4}. By \eqref{s.55}
\beq\label{s.56}
\lim_{\varepsilon \downarrow 0}  \left([\mathcal BR_0(\gamma\pm i\varepsilon) - \left(\mathcal BR_0\right)_\pm(\gamma)] f_\pm
, R_0(\gamma\pm i \varepsilon)f_\pm\right)_{\mathcal H}=0.
\ene
By \eqref{s.53} and \eqref{s.56}
\beq\label{s.57}
   \lim_{\varepsilon \downarrow 0} {\rm Im} 
\left(f_\pm
, R_0(\gamma\pm i \varepsilon)f_\pm\right)_{\mathcal H}=0.
\ene
Further, by \eqref{2.76}
\beq\label{s.57b}\begin{array}{l}
\left(f_\pm, R_0(\gamma\pm i \varepsilon)f_\pm\right)_{\mathcal H}= 2  \sum_{l =1}^\infty  \int_{E_{\rm min}}^{E_0}
 \frac{T(E)}{|\varphi'(E)|} \left|(f_\pm)_l(E)\right|^2 \left( \frac{(4\pi l)^2}{T(E)^2} -\gamma\mp i\varepsilon\right)^{-1} dE.
 \end{array}
\ene
Let us take $\Gamma_K$ as in Theorem~\ref{theo.s2}  with $K=[\gamma].$
Note that by  \eqref{nueva},  Theorem~\ref{s.2},  \eqref{s.52xx}, and \eqref{der.1a} we have that $ \frac{1}{|\varphi'(E)|} f_\pm \in \hat{C}_{\alpha}(I_0, L^2_{\rm o}(S_{1})).$
Then, by \eqref{s.57} and  \eqref{s.57b},  it follows from  Theorems 5 and 6 of Section 2 of Chapter 1 on 
\cite{ya} that
\beq\label{s.57c}
 \sum_{l \in J}b_l(\theta, E, \gamma)   f_l(\lambda_l(\gamma)) (f_\pm)_l(\gamma)=0, 
\ene
where $ b_l(\theta, E, \gamma)$ is defined in \eqref{s.46c} and
\beq\label{s.57d}
{J}:= \left\{ l \in \mathbb N : \left(    \frac{(4\pi l)^2}{T(E_0)^2 }
 , \frac{(4\pi l)^2}{T(E_{\rm min})^2)}\right) \cap \{\gamma\} \neq\emptyset\right\}.
\ene   
See the proof of Theorem~\ref{theo.s2}, in particular \eqref{s.48}, for a similar argument. Finally, arguing as in the proof of Theorem~\ref{theo.s2} we prove,
\beq\label{s.57e}
\left( \mathcal B R_0\right)_+(\gamma) f_\pm= \left( \mathcal B R_0\right)_-(\gamma) f_\pm.
\ene
This proves that $ \mathcal N_+= \mathcal N_-= \mathcal N.$  Let $\alpha_1$ and $ \alpha_2$ satisfy the conditions imposed  to $\alpha$ in the proposition.Then,  if $ f_\pm \in \hat{C}_{\alpha_1}(I_0, L^2_{\rm o}(S_1))$ 
fulfills \eqref{s.52xx} it follows from Theorem~\ref{theo.s2b} that  $ f_\pm \in \hat{C}_{\alpha_2}(I_0, L^2_{\rm o}(S_{1} )).$ This proves that  
 $\mathcal N$ is independent of $\alpha.$ 
\end{proof}
\subsection{The eigenvalues of the Antonov operator}
In the following Theorem  we obtain properties of the eigenvalues of the Antonov operator $\mathcal A.$
\begin{theorem}\label{eigen}
Let us define $ \mathcal N:= \mathcal N_+=\mathcal N_-$ as in Proposition~\ref{prop2.3}. Then.
\begin{enumerate}
\item[\rm (a)] The set
\beq\label{s.59}
\mathcal M:=\mathcal N  \cup_{l \in \mathbb N }
\left\{  \frac{(4\pi l)^2}{T(E_0)^2 }
 , \frac{(4\pi l)^2}{T(E_{\rm min})^2 }\right\}
\ene
is closed and of Lebesgue measure zero.
\item[\rm (b)]
\beq\label{s.59b}
\Sigma_p(\mathcal A) \cap \Gamma \subset \mathcal N.
\ene
  Moreover,  the eigenvalues of $\mathcal A$ in $\Gamma$ have finite multiplicity.  
 \item[\rm (c)] (Antonov bound \cite{ant}-\cite{ant2}) If the discrete spectrum of $\mathcal A,$ that we denote by $\sigma_{\rm dis}(\mathcal A),$ is not empty,
 \beq\label{s.59bxz}
 \mathcal A \geq  {\rm min}\,\sigma_{\rm dis}(A) >0.
 \ene
 Moreover, if  $\sigma_{\rm dis}(A)$ is empty
  \beq\label{s.59cb}
 \mathcal A \geq  \frac{(4\pi)^2}{(T(E_0))^2} >0.
 \ene
 
  \end{enumerate}
 \end{theorem}
 \begin{proof} We first prove that $\mathcal M$ is closed. Suppose that $\gamma_n \in \mathcal M$ converges to $ \gamma \in \Gamma.$ Then, for $n$ large enough $\gamma_n \in \mathcal N.$  Let us prove that $ \gamma \in \mathcal N.$ Otherwise, as by Theorem~ \ref{theo.s2b}, $(\mathcal B R_0)_\pm(\gamma)$ is compact, we would have that $(\mathcal I-(\mathcal B R_0)_\pm(\gamma))^{-1}$ would be a bounded operator on
  ${\hat{C}_{\alpha}(I_0, L^2_{\rm o}(S_{1}))} .$ However, by the stability of bounded invertibility theorem (Theorem 1.16 in page 196 0f \cite{kato}),  $(\mathcal I-(\mathcal B R_0)_\pm(\gamma_n))^{-1}$ would be a bounded operator on
  ${\hat{C}_{\alpha}(I_0, L^2_{\rm o}(S_{1}))} ,$ for $n$ large enough, in contradiction with the assumption that $\gamma_n \in \mathcal N$ for $n$ large enough. Here we used that by Theorem~\ref{theo.s2b} $ (\mathcal B R_0)_\pm(\gamma)$ is continuos in the operator norm of   ${\hat{C}_{\alpha}(I_0, L^2_{\rm o}(S_{1}))} .$ Further, if 
  $\gamma_n \in \mathcal M$ converges to $\gamma$ where
  $$
 \gamma \in \cup_{l\in \mathbb N} \left\{  \frac{(4\pi l)^2}{T(E_0)^2 }
 , \frac{(4\pi l)^2}{T(E_{\rm min})^2 }\right\},
  $$
  then, $ \gamma \in \mathcal M$ by the definition of $\mathcal M.$ This proves that $\mathcal M$ is closed.
 Let us prove  that $\mathcal M$ is of measure zero. It is enough to prove that 
 \beq\label{s.59bx}
F_l:= \mathcal N \cap \left(  \frac{(4\pi l)^2}{T(E_0)^2 }
 , \frac{(4\pi l)^2}{T(E_{\rm min})^2 }\right ), \qquad l=1,\dots
 \ene
 has measure zero. Let us take sequences $\{a_n^{(l)}\},$ and  $\{b_n^{(l)}\},$ where  $\{a_n^{(l)}\}$  is monotonic decreasing and  $\{b_n^{(l)}\}$ is monotonic increasing, $\ds \frac{(4\pi l)^2}{T(E_0)^2 } < a_n^{(l)} < b_n^{(l)} <
 \frac{(4\pi l)^2}{T(E_{\rm min})^2 }, $  $\ds \lim_{n\to \infty}a_n^{(l)}=  \frac{(4\pi l)^2}{T(E_0)^2 },$ and
 $\ds \lim_{n\to \infty}b_n^{(l)}=  \frac{(4\pi l)^2}{T( E_{\rm min})^2 }.$ We have,
 \beq\label{s.59c}
 F_l = \cup_{n=1}^\infty  \left( \mathcal N \cap [a_n^{(l)}, b_n^{(l)}]\right).
 \ene 
  By   Theorem~\ref{theo.s2b} $(\mathcal B R_0)_\pm(\gamma))$ is compact in  $\hat{C}_{\alpha}(I_0, L^2_{\rm o}(S_{1})),$ for $ \gamma \in \rho(\mathcal A_0) \cup \Gamma,$ and it is continuous in operator norm in  $\hat{C}_{\alpha}(I_0, L^2_{\rm o}(S_{1})),$ for $ \gamma \in \rho(\mathcal A_0) \cup [a_n^{(l)}, b_n^{(l)}].$ Hence,
 by \eqref{s.1b}, \eqref{nueva}, Theorem 3, and Remark 4 in Section 8 of Chapter 1 of \cite{ya}, and of Lemma 5 in Section 1 of Chapter 4 of \cite{ya}, $ \mathcal N \cap [a_n^{(l)}, b_n^{(l)}]$ is of measure zero. Then ,by \eqref{s.59c} $F_l$ is of measure zero. This proves that  $\mathcal M$ is of zero measure. We now prove (b). Suppose that $\gamma \in \Gamma$ belongs  to $\Sigma_p(\mathcal A).$ Then, for some $ \psi \in D[\mathcal A]= D[\mathcal A_0],$ with $ \psi \neq 0,$
\beq\label{s.60}
(\mathcal A_0 -\gamma) \psi= f,
\ene
where we denote,
\beq\label{s.61}
f(\theta,E):= (\mathcal B \psi)(\theta,E)= 4 \pi  |\varphi'(E)|\, v(\theta,E)\, \int_{U_0(x)}^{E_0} \psi(\theta(x,\lambda), \lambda) \, d\lambda, \qquad x= x(\theta,E).
\ene
Recall that $\mathcal B$ is defined in \eqref{2.81b}. By \eqref{2.51} and Sobolev's imbedding theorem (Theorem 5.4 in pages 97-98 of \cite{adams}) for $E$ fixed $\psi(\theta,E)$ is continuously differentiable in $\theta$ and
\beq\label{s.61b}
|\psi(\theta,E)|+|\partial_\theta \psi(\theta,E)| \leq C \| \psi(\cdot, E)\|_{H_2((0,1))}, \qquad \theta \in S_1.
\ene
Further, assuming that $ U_0(x_1) \leq U_0( x_2),$ by \eqref{2.36} for any $ 0 <\delta <1,$ 
\beq\label{s.61c}\begin{array}{l}
 \left|\int_{U_0(x_2)}^{E_0} d\lambda [ \psi(\theta(x_1,\lambda), \lambda)-\psi(\theta(x_2,\lambda), \lambda)]  \right| \leq 
\\ [5pt] C  \int _{U_0(x_2)}^{E_0} \, d\lambda \frac{1}{(\lambda-U_0(x_2))
^{(1-\delta)/2}} |\int_{x_1}^{x_2} \frac{1}{(\lambda-U_0(y))^{\delta/2}}  \,dy|
\| \psi(\cdot, \lambda)\|_{H_2((0,1))}. 
\end{array} 
\ene
Further, by H\"older's inequality with $1/p+1/q=1,$  $ p>1,$ and  $\delta >0,$ such that, $   \delta q <1,$
we have, for $ \lambda > U_0(x_2),$
\beq\label{s.61jj}
 \left|\int_{x_1}^{x_2} \frac{1}{(\lambda -U_0(y))^{\delta/2}}  \,dy\right |\leq |x_1-x_2|^{1/p} \left[\left |\int_{x_1}^{x_2} \frac{1}{(\lambda -U_0(y))^{\delta q/2}}  \,dy \right |\right]^{1/q}\leq C   |x_1-x_2|^{1/p}.
\ene
Moreover, by  Schwarz's inequality, 
\beq\label{s.61cdd}\begin{array}{c}
\int _{U_0(x_2)}^{E_0}\, d \lambda \ds \frac{1}{(\lambda-U_0(x_2))^{(1-\delta)/2}}   \| \psi(\cdot, \lambda)\|_{H_2((0,1))}\leq
\left[   \int _{U_0(x_2)}^{E_0} \, d\lambda\frac{1}{(\lambda-U_0(x_2))^{(1-\delta)}}   \right]^{1/2} \\[5pt]\left[ \int _{U_0(x_2)}^{E_0}\, d\lambda     \| \psi(\cdot, \lambda)\|_{H_2((0,1))}^{2}\right]^{1/2} < \infty,
\end{array}
\ene
where in the last inequality we used \eqref{2.51}. By \eqref{s.61c}, \eqref{s.61jj}, and\eqref{s.61cdd},
\beq\label{s.61xxxx}
 \left|\int_{U_0(x_2)}^{E_0} d\lambda\,  [ \psi(\theta(x_1,\lambda), \lambda)-\psi(\theta(x_2,\lambda), \lambda)]  \right| 
\leq C |x_1-x_2|^{1/p}.
\ene
Furthermore, by \eqref{2.51}, \eqref{s.61b} and Schwarz's inequality.
 \beq\label{s.61e}
 \left|    \int_{U_0(x_1)}^{U_0(x_2)} \psi(\theta(x,\lambda) \, d\lambda    \right |\leq C \sqrt{x_1-x_2}.
 \ene
  Then, by \eqref{s.60},  \eqref{s.61xxxx}, \eqref{s.61e}, \eqref{der.1}, \eqref{der.1a}, it follows that $f \in {\hat{C}_{\alpha}(I_0, L^2_{\rm o}(S_{1}))}, $
  with $\alpha$ as in Proposition~\ref{prop2.3}. Further, 
\beq\label{s.62}
\lim_{ \varepsilon \downarrow 0}  R_0(\gamma\pm i \varepsilon) (\mathcal A_0 -\gamma) \psi=
\psi  \pm i \lim_{ \varepsilon \downarrow 0} \varepsilon R_0(\gamma\pm i \varepsilon) \psi = \psi,
\ene 
 where in the last equality in \eqref{s.62} we used \eqref{2.76}. Hence, by \eqref{s.60}, \eqref{s.61}, and \eqref{s.62}
 \beq\label{s.63}
\left( \mathcal B R_0\right)_\pm(\gamma) f=  \lim_{\varepsilon \downarrow 0} \mathcal B R_0(\gamma\pm i \varepsilon) (\mathcal A_0-\gamma) \psi= \mathcal B \psi= f.
 \ene
 This proves that $\gamma \in \mathcal N,$ as  $ f\neq 0,$  since if $f=0$  by \eqref{s.60} $\gamma$ would be an eigenvalue of $\mathcal A_0,$ but this is imposible as $\mathcal A_0$ has no eigenvalues. The multiplicity of the eigenvalues of $\mathcal A$ in $\Gamma$ is finite, because the multiplicity  of an eigenvalue $\gamma$ of $\mathcal A$ in $\Gamma$ is at most equal to the multiplicity of the eigenvalue $1$ of $G_\pm(\gamma),$  that  is finite because as $G_\pm(\gamma)$ is compact   $1$ is an eigenvalue of finite multiplicity of  $G_\pm(\gamma)$.   
Let us prove  (c). By  \eqref{2.81c} and  Theorem  7.9 of \cite{hrs} zero belongs to the resolvent set of $\mathcal A.$  Then, (c) follows since by \eqref{2.81c} we have that $(4\pi)^2/ T(E_0)^2$ is the bottom of the essential spectrum of $\mathcal A.$
This concludes the proof of the theorem.
 \end{proof}
 
 \subsection{The absolutely continuous spectrum of the Antonov operator}
 In this subsection we  identify the absolutely continuous spectrum of  the Antonov operator $\mathcal A.$ 
 By the second resolvent equation,
 \beq\label{s.65}
 R_{\mathcal A}(z)= R_{\mathcal A_0}(z)+  R_{\mathcal A}(z) \mathcal B  R_{\mathcal A_0}(z), \qquad z \in \rho(\mathcal A) \cap \rho(\mathcal A_0).
 \ene
 Note that, as by  Theorem 5.19 of \cite{hrs}  $ \mathcal B  R_{\mathcal A_0}(z)$ is compact in $\mathcal H$ for $z\in \rho(\mathcal A_0),$ the operator $\mathcal I-\mathcal B  R_{\mathcal A_0}(z)$ is  invertible for  $z\in \rho(\mathcal A_0),$  with bounded inverse, unless $1$ is an eigenvalue of $\mathcal B  R_{\mathcal A_0}(z).$ However, by \eqref{s.65} in this case $R_{\mathcal A_0}(z)$ would not be invertible, but this is impossible for $z \in \rho(\mathcal A_0).$ Then, we have,
 \beq\label{s.66}
 R_{\mathcal A}(z)= R_{\mathcal A_0}(z)  \left(\mathcal  I-\mathcal B R_{\mathcal A_0}(z) \right)^{-1}, \qquad   \qquad z \in \rho(\mathcal A) \cap \rho(\mathcal A_0).
\ene
We denote
\beq\label{s.83b}
{\hat{C}_{\alpha,{\rm b}}(I_0, L^2_{\rm o}(S_{1}))} :=\left\{ f \in {\hat{C}_{\alpha}(I_0, L^2_{\rm o}(S_{1}))}: \frac{1}{|\varphi'|} f\, \text{\rm is bounded}\right\}.
\ene
Observe that, in the  case of the polytrope \eqref{1.12} with $k=1,$ and for  the King steady state  \eqref{1.12b}, $ {\hat{C}_{\alpha,{\rm b}}(I_0, L^2_{\rm o}(S_{1}))} = {\hat{C}_{\alpha,}(I_0, L^2_{\rm o}(S_{1}))}.$ Furthermore, for  $\mu \in\rho(\mathcal A_0)\cup\sigma_{\rm ess}(\mathcal A)\setminus\mathcal M$ the operator  $ (\mathcal I-(\mathcal B R_0)_\pm(\mu))^{-1}$ is a bijection in
  $ {\hat{C}_{\alpha,{\rm b}}(I_0, L^2_{\rm o}(S_{1}))}$, as we prove below. Suppose that $ f\in {\hat{C}_{\alpha,{\rm b}}(I_0, L^2_{\rm o}(S_{1}))},$ and denote,
  $$
  g:= (\mathcal I-(\mathcal B R_0)_\pm(\mu))^{-1} f. 
  $$ 
Hence, by  Theorem~\ref{theo.s2b}, as 
$$
  \frac{1}{|\varphi'|}  g=  \frac{1}{|\varphi'|} f+  \frac{1}{|\varphi'|} (\mathcal B R_0)_\pm(\mu) g, 
$$
we have that $ \frac{1}{|\varphi'|}  g$ is bounded. Further if $ f \in {\hat{C}_{\alpha,{\rm b}}(I_0, L^2_{\rm o}(S_{1}))},$ then,
$$
  f=(\mathcal I-(\mathcal B R_0)_\pm(\mu))^{-1}  (\mathcal I-(\mathcal B R_0)_\pm(\mu ))f.
$$ 
Further, $(\mathcal I-(\mathcal B R_0)_\pm(\mu))f \in   {\hat{C}_{\alpha,{\rm b}}(I_0, L^2_{\rm o}(S_{1}))}.$ Then, $(\mathcal I-(\mathcal B R_0)_\pm(\mu))^{-1} $ is onto  $ {\hat{C}_{\alpha,{\rm b}}(I_0, L^2_{\rm o}(S_{1}))}.$  Moreover, we have,  $ {\hat{C}_{\alpha,{\rm b}}(I_0, L^2_{\rm o}(S_{1}))} \subset \mathcal H.$
 
Let $\alpha,$ be as in Proposition~\ref{prop2.3} and $\mathcal M$ be as in Theorem~\ref{eigen}. For   $ f \in {\hat{C}_{\alpha, {\rm b}}(I_0, L^2_{\rm o}(S_{1}))} $   and $\mu \in\sigma_{\rm ess}(\mathcal A)
\setminus\mathcal M,$ we define,

\beq\label{s.67}
\ds f_{\mu\pm  i\varepsilon}:=  (\mathcal I-(\mathcal B R_0)_\pm(\mu \pm i\varepsilon))^{-1}f, \qquad \varepsilon \geq 0.
\ene
\begin{proposition}\label{prop.stone}
Let $\alpha$ be as in Proposition~\ref{prop2.3}. Then,  for every $f, g \in  {\hat{C}_{\alpha, {\rm b}}(I_0, L^2_{\rm o}(S_{1}))} $  and for every $\mu \in\sigma_{\rm ess}(\mathcal A)
\setminus\mathcal M,$
\beq\label{s.68}\begin{array}{l}
\lim_{\varepsilon \downarrow 0}\frac{1}{2\pi i}\left( \left[R_{\mathcal A}(\mu+i\varepsilon)- R_{\mathcal A}(\mu- i\varepsilon)\right] f, g \right)_{\mathcal H}= \sum_{l=1}^\infty  \ds \chi_{(\beta_{l, {\rm min}}, \beta_{l, {\rm max}})}(\mu) \ds\frac{T(\lambda_l(u))}{|\varphi'(\lambda_l(\mu))|} \\ p_l(\lambda_l(\mu)) 
 
 (\mathbf Ff_{\mu \pm i0})_{l}(\lambda_l(\mu)) \overline{(\mathbf Fg_{\mu\pm i0})_{l}}(\lambda_l(\mu)),
\end{array}
\ene
where the limit is uniform for  $\mu $  in compact sets in $\sigma_{\rm ess}(\mathcal A)
\setminus\mathcal M.$ Note that as $f_{\mu\pm i0}(\lambda), $ and  $g_{\mu\pm i0}(\lambda)$ 
  are   jointly locally H\"older continuous in $\mu$ and $\lambda$ we can take the traces at   $\lambda =\lambda_l(\mu)) $ in \eqref{s.68}. 
\end{proposition}
\begin{proof} By the first resolvent equation,
\beq\label{s.69a}
 R_{\mathcal A}(\mu+i\varepsilon)-R_{\mathcal A}(\mu-i\varepsilon)= 2i \varepsilon  R_{\mathcal A}(\mu+i\varepsilon)
 R_{\mathcal A}(\mu-i\varepsilon).
\ene
Then, by  \eqref{2.55bb}, \eqref{2.57}, \eqref{s.66}, \eqref{s.67}, and \eqref{s.69a},
\beq\label{s.69}\begin{array}{l}
\frac{1}{2\pi i}\left( \left[R_{\mathcal A}(\mu+i\varepsilon)- R_{\mathcal A}(\mu- i\varepsilon)\right] f, g \right)_{\mathcal H}= \sum_{l=1}^\infty \int_{I_0} d\lambda   \ds\frac{T(\lambda)}{|\varphi'(\lambda)|}   \ds\frac{\varepsilon}{\pi} 
\frac{1}{ (\frac{(2\pi l)^2}{T^2(\lambda)}-\mu )^2+\varepsilon^2 }  
 \\
 (\mathbf Ff_{\mu\pm i\varepsilon} )_{l}(\lambda)  \overline{(\mathbf F g_{\mu\pm i\varepsilon})_{l}}(\lambda).
\end{array}
\ene
Changing the variable of integration as in \eqref{sp.1}-\eqref{s.43} we obtain,
\beq\label{s.70b}\begin{array}{l}
\frac{1}{2\pi i}\left( \left[R_{\mathcal A}(\mu+i\varepsilon)- R_{\mathcal A}(\mu- i\varepsilon)\right] f, g \right)_{\mathcal H}=  \sum_{l=1}^\infty \int_{\beta_{l,{\rm min}}}^{\beta_{l,{\rm max}}} d\beta   \ds\frac{T(\lambda_l(\beta))}{|\varphi'(\lambda_l(\beta))|}   \ds\frac{\varepsilon}{\pi} 
\frac{1}{ (\beta-\mu )^2+\varepsilon^2 }  p(\lambda_l(\beta)) 
 \\
 (\mathbf Ff(\mu\pm i\varepsilon) )_{l}(\lambda_l(\beta))  \overline{(\mathbf Fg(\mu\pm i\varepsilon))_{
 l}}(\lambda_l(\beta)) \, d\beta_.
\end{array}
\ene
Moreover, by  Theorem~\ref{theo.s2b}, \eqref{s.44}, \eqref{s.46}, \eqref{s.70b},  items (a) and (b) of Proposition~\ref{propder}, and  Theorem 7 in Subsection 3 of Section 2 of Chapter 1 of \cite{ya} ( see also Theorem 5 and Remark 6 in  Chapter  IV of 
\cite{kur}, Theorem 6 in Section 2 of  Chapter 1 of \cite{ya}, and also Section 19 of Chapter 2 of \cite{mu})
\beq\label{s.71}\begin{array}{l}
\lim_{\varepsilon \downarrow 0}\frac{1}{2\pi i}\left( \left[R_{\mathcal A}(\mu+i\varepsilon)- R_{\mathcal A}(\mu- i\varepsilon)\right] f, g \right)_{\mathcal H}= \sum_{l=1}^\infty  \ds \chi_{(\beta_{l, {\rm min}}, \beta_{l, {\rm max}})}(\mu) \ds\frac{T(\lambda(u))}{|\varphi'(\lambda_l(\mu))|}  p_l(\lambda_l(\mu)) 
 \\
 (\mathbf Ff_{\mu\pm i0} )_{l}(\lambda_l(\mu))  \overline{(\mathbf Fg_{\mu\pm i0})_{l}}(\lambda_l(\mu)),
\end{array}
\ene
with the limit uniform for $ \mu$ in compact sets in $\sigma_{\rm ess}(\mathcal A)
\setminus\mathcal M.$ This concludes the proof of the proposition.
\end{proof}
\begin{theorem} \label{theoproj}
\begin{enumerate}
\item[\rm (a)]
The projector $P_{\rm ac}(\mathcal A)$ onto the subspace of absolute continuity of the Antonov operator$\mathcal A $ is given by,
 \beq\label{s.72}
 P_{\rm ac}(\mathcal A)= E_{\mathcal A}(\sigma_{{\rm ess}}(\mathcal A)\setminus \mathcal M),
 \ene
 where $\mathcal M$ is the closed set of measure zero defined in Theorem~\ref{eigen}.
 \item[\rm(b)]
 The absolutely continuous spectrum of $\mathcal A$ is given by,
 \beq\label{s.72b}
 \sigma_{\rm ac}(\mathcal A)= \sigma_{\rm ess}(\mathcal A)=  \cup_{l\in \mathbb N}
\left\{  \frac{(2\pi l)^2}{T(E)^2 }
 : E \in [E_{\rm min}, E_0]\right\}.
\ene
\item[\rm (c)]
The intersection of the singular spectrum of $\mathcal A$ with the absolutely continuous spectrum of $\mathcal A$
is contained in  $\mathcal M,$
\beq\label{2.72c}
\sigma_{\rm sing}(\mathcal A)\cap \sigma_{\rm ac}(\mathcal A) \subset \mathcal M.
\ene
 \end{enumerate}
 \end{theorem}
 \begin{proof}
 Let $\alpha$ be as in Theorem~\ref{theo.s2b},
 and let $f, g \in {\hat{C}_{\alpha, {\rm b}}(I_0, L^2_{\rm o}(S_{1}))}. $  Recall that  ${\hat{C}_{\alpha, {\rm b}}(I_0, L^2_{\rm o}(S_{1}))}$ was defined in \eqref{s.83b}. Then, by Stone's formula (see Theorem VII.13, and the comment in page 264 of \cite{rs1})  and Theorem~\ref{eigen}, for any open interval   $\Delta:=
 (a,b) \subset \sigma_{\rm ess}(\mathcal A) \setminus \mathcal M, $ 
 
 \beq\label{s.73}
 \left( E_{\mathcal A}(\Delta) f,g \right)_{\mathcal H}= \lim\_{\varepsilon \downarrow 0} \frac{1}{2\pi i} \int_{\Delta} \, d\mu \left( \left[R_{\mathcal A}(\mu+i\varepsilon)- R_{\mathcal A}(\mu- i\varepsilon)\right] f, g \right)_{\mathcal H}.
 \ene
Introducing the uniform limit \eqref{s.68} into  the integral in \eqref{s.73}  we get,
 \beq\label{s.74}\begin{array}{l}
  \left( E_{\mathcal A}(\Delta) f,g \right)_{\mathcal H}=  \int_{\Delta} \, d\mu \sum_{l=1}^\infty  \ds \chi_{(\beta_{l, {\rm min}}, \beta_{l, {\rm max}})}(\mu) \ds\frac{T(\lambda_l(\mu))}{|\varphi'(\lambda_l(\mu))|}  p_l(\lambda_l(\mu)) 
 \\
 (\mathbf Ff_{\mu\pm i0} )_{l}(\lambda_l(\mu))  \overline{(\mathbf Fg_{\mu\pm i 0})_{l}}(\lambda_l(\mu)).
\end{array}
\ene
 As the measure in the right-hand side of \eqref{s.74} is absolutely continuous, we have $ E_{\mathcal A}(\Delta) f \subset \mathcal H_{\rm ac}(\mathcal A)$ for $f \in  {\hat{C}_{\alpha, {\rm b}}(I_0, L^2_{\rm o}(S_{1}))}. $  Since  $ {\hat{C}_{\alpha,b}(I_0, L^2_{\rm o}(S_{1}))}$ is dense in $\mathcal H$ and as  $\mathcal H_{\rm ac}(\mathcal A)$ is closed, we get
 \beq\label{s.75}
 E_{\mathcal A}(\Delta) \mathcal H \subset \mathcal H_{\rm ac}(\mathcal A),
 \ene 
 for any open interval  $\Delta:=
 (a,b) \subset \sigma_{\rm ess}(\mathcal A) \setminus \mathcal M. $ But then, \eqref{s.75} follows for any Borel set 
$\Delta \subset \sigma_{\rm ess}(\mathcal A) \setminus \mathcal M.$ Moreover, as $\mathcal M$ is a closed set of measure zero it can not support absolutely continuous spectrum, and then, by \eqref{s.75} equation \eqref{s.72} holds. Let us now prove \eqref{s.72b}. We already know that $\sigma_{\rm ess}(\mathcal A)\setminus \mathcal M \subset \sigma_{\rm ac}(\mathcal A).$ Consider $ \beta \in \sigma_{\rm ess}(\mathcal A) \cap \mathcal M.$ Suppose that $ \beta$ is not a limit point of $\sigma_{\rm ess}(\mathcal A)\setminus \mathcal M .$ Then, there  would be a neighborhood 
of $\beta$ that contains no point of $\sigma_{\rm ess}(\mathcal A)\setminus \mathcal M.$ This would imply that there is a an open set of positive measure contained in $\sigma_{\rm ess}(\mathcal A)\cap \mathcal M.$ However, this is
not possible since $\mathcal M$ has measure zero. Hence, every point in $\sigma_{\rm ess}(\mathcal A) \cap \mathcal M$ is a limit point of    $\sigma_{\rm ess}(\mathcal A)\setminus \mathcal M \subset \sigma_{\rm ac}(\mathcal A),$ and as the absolutely continuous spectrum is closed, $\sigma_{\rm ess}(\mathcal A)\cap \mathcal M \subset  \sigma_{\rm ac}(\mathcal A).$
Finally, as the complement in the spectrum of   $\sigma_{\rm ess}(\mathcal A)$ is discrete, \eqref{s.72b} follows. This proves (b). Let us prove (c). By \eqref{s.59b} 
$\Sigma_p(\mathcal A) \cap \Gamma \subset \mathcal N.$ Then, $\overline{\Sigma_p(\mathcal A)} \cap \sigma_{\rm ac}(A)\subset \mathcal M.$ 
By \eqref{s.72} $\mathcal H_{\rm sc}(A) \subset E_{\mathcal A}(\mathcal M),$ and then, 
$\sigma_{\rm sc}(A) \subset \mathcal M.$ This completes the proof of (c) and of the theorem
 \end{proof}
\section{The generalized Fourier maps}
\label{fouma}
In this section we construct the generalized Fourier maps of the Antonov operator.
We define first the generalized Fourier maps  $\mathcal F_\pm$    for finite linear combinations of the type,
\beq\label{s.75b}
f= \sum_{j=1}^N E_{\mathcal A}(O_j) f_j,
\ene
 where  $f_j \in {\hat{C}_{\alpha, {\rm b}}(I_0, L^2_{\rm o}(S_{1}))} $  and the $O_j$ are  disjoint intervals, $O_j \cap O_k=\emptyset $ for $j \neq k,$  with $O_j  \subset \sigma_{\rm ac}(\mathcal A)\setminus \mathcal M,$ for $j=1,\dots, N.$  The set of all the functions as in \eqref{s.75b} is dense in $\mathcal H _{\rm ac}(\mathcal A).$  Further, we take   $\alpha$  as in Proposition~\ref{prop2.3}. For $f$ as in \eqref{s.75b} we define,

\beq\label{s.75c}\begin{array}{l}
\left( \mathcal F_\pm f  \right)_l(\mu):= \sum_{j=1}^N  \chi_{O_j}(\mu) \ds\sqrt{\frac{T(\lambda_l(\mu)) |p(\lambda_l(\mu))|}{|\varphi'(\lambda_l(\mu )|}  } \ds \chi_{(\beta_{l, {\rm min}}, \beta_{l, {\rm max}})}(\mu) \\
\left( \mathbf F\left(\mathcal I-  (\mathcal B R_0)_\pm(\mu))^{-1} f_j\right)(\lambda_l(\mu)) \right)_{l}.
\end{array}
\ene
Note that, by \eqref{s.75c}, for any interval $\Delta  \subset \sigma_{\rm ac}(\mathcal A)\setminus \mathcal M ,$ and any $f$ as in \eqref{s.75b},
\beq\label{s.75d}
\chi_\Delta (\mu)( \mathcal F_\pm f)(\mu)= (\mathcal F_\pm E_{\mathcal A}(\Delta) f)(\mu).
\ene
Then, for every $f,g \in {\hat{C}_{\alpha, {\rm b}}(I_0, L^2_{\rm o}(S_{1}))} $  we can write \eqref{s.74} as
\beq\label{s.76}
\left( E(\Delta) f,g \right)_{\mathcal H}= \left(\chi_{\Delta} \mathcal F_\pm f, \mathcal F_\pm g\right)_{\mathcal H_{\rm sp}},
\ene
with $\mathcal H_{\rm sp}$ defined in \eqref{sp.4}.
As \eqref{s.76} holds for every open interval, it also holds for every Borel   set $ O \ \in \sigma_{\rm ac}(\mathcal A)\setminus \mathcal M$, i.e.
\beq\label{s.77}
\left( E(O) f,g \right)_{\mathcal H}= \left(\chi_{O} \mathcal F_\pm f, \mathcal F_\pm g\right)_{\mathcal H_{\rm sp}}.
\ene
 Further, \eqref{s.77} extends by linearity to all functions of the form \eqref{s.75b}.
In  particular,
 \beq\label{s.78}
\
\left( P_{\rm ac}(\mathcal A) f,g \right)_{\mathcal H}= = \left(\chi_{(\sigma_{\rm ac}(\mathcal A)\setminus \mathcal M)} \mathcal F_\pm f, \mathcal F_\pm g\right)_{\mathcal H_{\rm sp}}=
 \left( \mathcal F_\pm f, \mathcal F_\pm g\right)_{\mathcal H_{\rm sp}}.
\ene
Note that \eqref{s.78} implies that the operators $\mathcal F_\pm$ are densely defined and bounded from 
$\mathcal H_{\rm ac}(\mathcal A)$ into $\mathcal H_{\rm sp}.$ We extend then by continuity to bounded operators from $\mathcal H_{\rm ac}(A)$ into
$\mathcal H_{\rm sp}$ and we denote the extensions by the same symbol  $\mathcal F_\pm.$ Further, we define $\mathcal F_\pm$ by zero on $\mathcal H_{\rm s}(\mathcal A),$
\beq\label{s.79}
\mathcal F_\pm f=0, \qquad f \in \mathcal H_{\rm s}(\mathcal A).
\ene
Then, by \eqref{s.78} the generalized Fourier maps $\mathcal F_\pm$ are partially isometric from $\mathcal H$ into
$\mathcal H_{\rm sp}$ with initial subspace $\mathcal H_{\rm ac}(\mathcal A).$  In the theorem below we prove that they are actually onto $\mathcal H_{\rm sp}.$ Our spectral theorem in terms of {\it trace maps}, that we obtained at the end of Subsection~\ref{subspec}, plays a key role for this purpose
\begin{theorem} \label{propgf}
The generalized Fourier maps $\mathcal F_\pm$ are partially isometric from $\mathcal H$ onto $\mathcal H_{\rm sp}.$ Furthermore, the  initial subspace of  $\mathcal F_\pm$  is $\mathcal H_{\rm ac}(\mathcal A)$ and the final subspace is $\mathcal H_{\rm sp},$
\beq\label{s.80}
\mathcal F_\pm^\ast \mathcal F_\pm= P_{\rm ac}(\mathcal A),\qquad  \mathcal F_\pm \mathcal F_\pm^\ast =\mathcal I.
\ene
Moreover, for any  Borel function $\phi,$
\beq\label{s.81}
\phi(\mathcal A_{\rm ac}) = \mathcal F_\pm ^\ast \, \phi(\mu) \, \mathcal F_\pm,
\ene
where $\mathcal A_{\rm ac}$ is the absolutely continuous part of $\mathcal A,$ and by $\phi(\mu)$ we denote the operator of multiplication by the function  $\phi(\mu),$ that is to say,
\beq\label{s.82}
\phi(\mu) \{f_l\}=  \{  \varphi(\mu) f_l(\mu)\}.
\ene
In particular,
\beq\label{s.82b}
\mathcal A_{\rm ac}= \mathcal F_\pm ^\ast \,\mu \, \mathcal F_\pm.
\ene
\end{theorem}
\begin{proof}
We already know that the generalized Fourier maps $\mathcal F_\pm$ are partially isometric from $\mathcal H$ into
$\mathcal H_{\rm sp}$ with initial subspace $\mathcal H_{\rm ac}(\mathcal A).$ It remains to prove that they are onto 
$\mathcal H_{\rm sp}.$  For this purpose, note that by \eqref{sp.19},  \eqref{sp.20}, and \eqref{s.75c} for any interval  $\tilde{\Delta} \subset \sigma_{\rm ess}(\mathcal A) \setminus \mathcal M $ contained in the interior of $\Delta_n,$ and  any $f \in {\hat{C}_{\alpha, {\rm b}}(I_0, L^2_{\rm o}(S_{1}))} $ with $\alpha$ as in Proposition~ \ref{prop2.3},
\beq\label{s.83}
\left(\mathcal F_\pm E_{\mathcal A }(\tilde{\Delta}) f\right)(\beta)= \chi_{\tilde{\Delta}}(\beta) \mathcal L_{n}(\beta) \mathbf F(\mathcal I-(\mathcal B R_0)_\pm(\beta))^{-1}f.
\ene
 Recall that  ${\hat{C}_{\alpha, {\rm b}}(I_0, L^2_{\rm o}(S_{1}))}$ was defined in \eqref{s.83b}
Let us prove that for almost every $ \beta $ in $\tilde{\Delta}$ 
\beq\label{s.83c}
 \mathcal L_{n}(\beta) \mathbf F(\mathcal I -(\mathcal B R_0)_\pm(\beta))^{-1}  {\hat{C}_{\alpha,{\rm b}}(I_0, L^2_{\rm o}(S_{1}))} =\mathbb C_{m_n}.
 \ene
   Since   $\varphi'(E_{l_{m,j}}(\beta)), T(E_{l_{m,j}}(\beta),$ and $p_{l_{m,j}}(\beta))$  are different from zero  for $ \beta  > \beta_{\rm min}$ and $j=1,\dots, m_{n},$  by \eqref{sp.20}  it is enough to prove that
  \beq\label{s.83d}
 \mathcal L_{n,{\rm r}}(\beta) \mathbf F(\mathcal I -(\mathcal B R_0)_\pm(\beta))^{-1}  {\hat{C}_{\alpha, {\rm b}} (I_0, L^2_{\rm o}(S_{1}))} =\mathbb C_{m_n},
 \ene
  for almost every $ \beta$ in $\tilde{\Delta},$ where  $\mathcal L_{n,{\rm r}}(\beta)$ is the operator  from $\mathcal H_{{\rm sp}, m_n}$  into $\mathbb C_{m_n}$  given by, 

   \beq\label{s.83e}
\begin{array}{r}
\mathcal L_{n, {\rm r}}(\beta) \{f_l\}:=  \left( f_{l_{n,1}}(\beta), 
\dots,  f_{l_{n,m_n}}(\beta) \right),\qquad  \{f_l\} \in \mathcal H_{{\rm sp}, m_n}. 
\end{array}
\ene
 Moreover, as $(\mathcal I -(\mathcal B R_0)_\pm(\beta))^{-1}$ is a bijection on   ${\hat{C}_{\alpha, {\rm b}} (I_0, L^2_{\rm o}(S_{1}))}$  (see the proof below \eqref{s.83b}) it is enough to prove that,
\beq\label{s.83ez}
 \mathcal L_{n,r}(\beta) \mathbf F {\hat{C}_{\alpha, {\rm b}}(I_0, L^2_{\rm o}(S_{1}))} =\mathbb C_{m_n}
 \ene
for almost every $ \beta$ in $\tilde{\Delta}.$ For this purpose we introduce the following Sobolev space. For $ m=1,\dots, $ let $H_{m.0}(I_0, L^2_{\rm o}(S_{1}))$ be the Sobolev space that is the completion of 
$C^\infty_0(I_0, L^2_{\rm o}(S_{1}))$    in the norm,
$$
\| f\|_{\ds H_{m, 0}(I_0, L^2_{\rm o}(S_{1}))}:=\left[ \sum_{j=0}^m\left \| \frac{\partial^j}{\partial E^j} f\right\|^2_{\ds L^2(I_0, L^2_{\rm o}(S_{1}))}\right]^{1/2}.
$$
For the polytropes we denote $ H_{\rm b}(I_0, L^2_{\rm o}(S_{1})):= H_{m, 0}(I_0, L^2_{\rm o}(S_{1}))$ 
  where $m$ is the smallest integer $m \geq k.$ Further, we designate  $ H_{b}(I_0, L^2_{\rm o}(S_{1})):=  H_{1,0}(I_0, L^2(_{1,{\rm o}}))$
for the  King steady states. Note that in all the cases  $ H_{\rm b}(I_0, L^2_{\rm o}(S_{1})) \subset  \hat{C}_{\alpha, {\rm b}}(I_0, L^2_{\rm o}(S_{1})).$ It is enough to prove,
\beq\label{s.85}
Q(\beta):= \mathcal L_{n,r}(\beta) \mathbf F H_{\rm b}(I_0, L^2_{\rm o}(S_{1})) =\mathbb C_{m_n}
 \ene
 for almost every $ \beta$ in $\tilde{\Delta}.$ For this purpose, let $P(\beta)$ be the projector onto $Q(\beta).$ Let $e_j, j=1,\dots m_n,$ be  the canonical basis of $\mathbb C_{m_n}.$ Assume that   \eqref{s.85} does not hold for almost every $\beta$  in $\tilde{\Delta}.$ Then, there has to be  a $j$ such that
  $$
  g_j(\beta):= (\mathcal I-P(\beta))e_j,
   $$
is different from zero in a set of positive measure $ O \subset \tilde{\Delta}.$ Let us prove that $g_j$ is a measurable function. For this purpose, we prove that $\mathcal I-P(\beta)$ is strongly measurable. Note that $ D_n(\beta):= \mathcal L_{n,r}(\beta) \mathbf F$ is bounded from $ H_{\rm b}(I_0, L^2_{\rm o}(S_{1}))$ into $\mathbb C_{m_n},$ and continuous in operator norm.
Remark  that  $\mathcal I-P(\beta)$ is the projector onto the kernel of $D(\beta)^\ast$ as an operator from $\mathbb C_{m_n}$ into $ H_{\rm b}(I_0, L^2_{\rm o}(S_{1})).$ Let us denote 
$$
K(\beta):= D(\beta) D(\beta)^\ast.
$$
Let $E_{K(\beta)}( \lambda)$ be the spectral family of $K(\beta)$ as a selfadjoint operator in  $ \mathbb C_{m_n}$ and let 
$R_{K(\beta)}(z)= ( K(\beta)-z )^{-1} $ be the resolvent of $K(\beta).$ Then, by Stone's formula  (see Theorem VII.13, and the comment in page 264 of \cite{rs1})  
$$
(\mathcal I-P(\beta)) f = E_{K(\beta)}(\{0\})f=  \lim_{\varepsilon \downarrow 0} \frac{1}{\pi i} \int_{-1}^0 [R_{K(\beta)}(\lambda+i\varepsilon)-
R_{K(\beta)}(\lambda-i\varepsilon)]f d\lambda, \qquad f \in \mathbb C_{m_n}.
 $$ 
It follows that  $\mathcal I-P(\beta)$ is  strongly measurable. Hence,  $g_j$ is measurable, and then, $g_j \in L^2(\tilde{\Delta}, \mathbb C_{m_n}).$ Moreover, as $g_j(\beta)$ is orthogonal to $Q(\beta)$
$$
\left( g_j, \mathcal L_{n,r}(\beta) \mathbf F f\right)_{L^2(\tilde{\Delta}, \mathbb C_{n_m})}=0, \qquad f \in  H_{\rm b}(I_0, L^2_{\rm o}(S_{1})).
$$
Further, as  the set $ \mathcal L_{n,r}(\beta) \mathbf F f,$ with $f \in  H_{\rm b}(I_0, L^2_{\rm o}(S_{1}))
$ is dense in $L^2(\tilde{\Delta}, \mathbb C_{n_m}),$ the function $g_j(\beta)$ is zero for almost every $ \beta \in \tilde{\Delta}.$ This contradiction proves that 
\eqref{s.85} holds, and then,

\beq\label{s.85b}
 \mathcal L_{n}(\beta) \mathbf F(\mathcal I -(\mathcal B R_0)_\pm(\beta))^{-1}   H_{\rm b}(I_0, L^2_{\rm o}(S_{1}))=\mathbb C_{m_n}.
 \ene
 Suppose now that some $ h \in \mathcal H_{\rm sp}$ is orthogonal to the range  of $\mathcal F_\pm.$ Let  $\{f_l\}_{l=1}^\infty$ be an orthonormal basis in $ H_{\rm b}(I_0, L^2_{\rm o}(S_{1})).$ Then, for every  interval  $\tilde{\Delta}
 \subset \sigma_{\rm ac}(\mathcal A)$ contained in the interior of $\Delta_n,$ 
\beq\label{s.86}
\left( \mathcal F_\pm E_{\mathcal A}(\tilde{\Delta}) f_l,h \right)_{\mathcal H_{\rm sp}}= \int_{\tilde{\Delta}} \left( \mathcal L_{n}(\beta) \mathbf F(\mathcal I -(\mathcal B R_0)_\pm(\beta))^{-1}  f_l,  h\right)_{\mathbb C_{n_m}}=0.
\ene
Hence,
$$\left(  \mathcal L_{n}(\beta) \mathbf F(\mathcal I -(\mathcal B R_0)_\pm(\beta))^{-1}   f_l,  h\right)_{\mathbb C_{n_m}}=0, \qquad  \text{\rm a.e} \, \beta \in \Delta_n,   l=1,\dots
$$
and by \eqref{s.85}
$h(\beta)=0,$  for almost every $\beta \in \Delta_n$. As this holds for every $\Delta_n$ we get that $h=0.$ 
This completes the proof that the $\mathcal F_\pm$ are onto $\mathcal H_{\rm sp}$ and that \eqref{s.80} holds. 
Furthermore, by \eqref{s.71},
\eqref{s.75c},  Stone's formula (see Theorem VII.13, and the comment in page 264 of \cite{rs1}),  the spectral theorem (see Section five of Chapter six of \cite{kato} and Theorem VIII.6 in page 263 of \cite{rs1})  for any intervals
 $\Delta_j\in \sigma_{\rm ac}(\mathcal A),$ for $j=1,2,$  any $f_j\in   \hat{C}_{\alpha}(I_0, L^2_{\rm o}(S_{1})) $   with compact support in $I_0$,  and for any bounded Borel  function $\phi$
\beq\label{s.86b}
\left( \phi(\mathcal A) E_{\mathcal A}(\Delta_1) f_1,   E_{\mathcal A}(\Delta_2) f_2\right)_{\mathcal H}=
\left( \phi(\cdot) \mathcal F_\pm   E_{\mathcal A}(\Delta_1) f_1, \mathcal F_\pm E_{\mathcal A}(\Delta_2)f_2\right)_{\mathcal H_{\rm sp}}.
\ene
Moreover, since the ser of functions \eqref{s.75b} is dense in $\mathcal H_{\rm ac}(\mathcal A),$ we have
\beq\label{s.87b}
\left( \phi(A)  f_1,  f_2\right)_{\mathcal H}=
\left( \phi(\cdot) \mathcal F_\pm f_1, \mathcal F_\pm f_2\right)_{\mathcal H_{\rm sp}}, \qquad f_1, f_2 \in \mathcal H_{\rm ac}(\mathcal A).
\ene
Equation \eqref{s.81} in the case where $\phi$ is  bounded follows from \eqref{s.87b}. In the general case where $\phi $ is not bounded,  \eqref{s.81} follows from \eqref{s.81} in the bounded case, approximating $\phi$ by 
$ \phi_n$ where $ \phi_n(\lambda)= \phi(\lambda)$ if $|\phi(\lambda)| \leq n$ and $ \phi_n(\lambda)=0,$ if 
$|\phi(\lambda)|  >n,$ for $n=1,\dots.$
\end{proof}
\section{The wave operators}\label{wave}
In this section we prove the existence and the completeness of the wave operators. We recall that the wave operators are said to be complete if their range is $\mathcal H_{\rm ac} (\mathcal A).$ See Definition 1 in Section 3 of Chapter 2 of \cite{ya}. We also prove  Birman's invariance principle of the wave operators. We first prove that the weak Abelian wave operators exist and and are isometric with final subspace the absolutely continuous subspace of $\mathcal A.$ The existence and the completeness of the wave operators follows from this result. Along the way, we also prove the stationary formulae for the wave operators. 
We begin defining the weak Abelian wave operators   as in equation 11 in page 76 of \cite{ya}
\beq\label{wa.2}
\Omega_\pm:= {\rm w-} \lim_{\varepsilon \downarrow 0} \int_0^\infty w_\varepsilon(t) P_{\rm ac}(\mathcal A)
e^{\pm it \mathcal A} e^{\mp it\mathcal A_0} dt,
\ene
if the limit exits, where we take $w_\varepsilon(t)= 2 \varepsilon e^{-2\varepsilon t},$     and  ${\rm w-} \lim_{\varepsilon \downarrow 0}$ means the weak limit as $\varepsilon \downarrow 0.$ This is equivalent to,
\beq\label{wa.2b}
\left(\Omega_\pm f, g\right)_{\mathcal H}:=  \lim_{\varepsilon \downarrow 0} \int_0^\infty w_\varepsilon(t) 
\left(e^{\pm it \mathcal A} e^{\mp it\mathcal A_0} f, g\right)_{\mathcal H}\, dt,
\ene
for $f\in \mathcal H,$ and $ g\in \mathcal H_{\rm ac}(\mathcal A).$

In the following theorem we prove that the weak Abelian wave operators exist, that they are given by the stationary formulae, and that they are isometric with final subspace $\mathcal H_{\rm ac}(A).$
  
\begin{theorem}\label{pwa.1}
The weak Abelian wave operators  \eqref{wa.2}  exist, and
\beq\label{wa.3a}
\Omega_\pm = \mathcal F_\pm^\ast \mathcal F.
\ene
Moreover the $\Omega_\pm$ are isometric with final subspace $\mathcal H_{\rm ac}(\mathcal A).$
\end{theorem}
\begin{proof}  Note that as $\mathcal F$ is unitary and the $\mathcal F_\pm$ are partially isometric with final subspace $\mathcal H_{\rm ac}(\mathcal A)$ it is immediate from  \eqref{wa.3a} that  the $\Omega_\pm$ are isometric with final subspace $\mathcal H_{\rm ac}(\mathcal A).$ We proceed to prove   \eqref{wa.3a}.
Since $ e^{\pm it \mathcal A} e^{\mp it\mathcal A_0}$ is unitary it is enough to prove that the limits in \eqref{wa.2b} exist in  dense sets. Then, we can replace $g$ in \eqref{wa.2b} by $ P_{\rm ac}(\mathcal A)g, $ with $g \in \hat{C}_{\alpha, {\rm b}}(I_0, L^2_{\rm o}(S_{1})).$   Recall that  ${\hat{C}_{\alpha, {\rm b}}(I_0, L^2_{\rm o}(S_{1}))}$ was defined in \eqref{s.83b}.Furthermore, we can approximate  $ P_{\rm ac}(\mathcal A)g $ by
finite sums, $\sum_{n=1}^NE_{ \mathcal A}(\tilde{\Delta}_n) g, $ where $\tilde{\Delta}_n$  are  bounded  closed  intervals contained in the interior of $\Delta_n$ with  $\Delta_n$ as in \eqref{sp.9}, \eqref{sp.10} and moreover, with    
  $\tilde{\Delta}_n \cap \mathcal M=\emptyset$ and $ \tilde{\Delta}_j \cap \tilde{\Delta}_n= \emptyset,$ for
  $j\neq n.$
 In the  same way, we can replace $ f$ in \eqref{wa.2b} by $\sum_{n=1}^N E_{ \mathcal A_0}(\tilde{\Delta}_n) f, $ with 
 $ f \in \hat{C}_{\alpha, {\rm b}}(I_0, L^2_{\rm o}(S_{1})).$ 
Then, in order that the weak Abelian wave operators exists it is enough to prove that
\beq\label{wa.2c}
\left(\Omega_\pm E_{\mathcal A_0}(\tilde{\Delta}_j) f, E_{\mathcal A}(\tilde{\Delta}_n) g \right)_{\mathcal H}:=  \lim_{\varepsilon \downarrow 0} \int_0^\infty w_\varepsilon(t) 
\left(e^{\pm it \mathcal A} e^{\mp it\mathcal A_0}   E_{\mathcal A_0}(\tilde{\Delta}_j) f, E_{\mathcal A}(\tilde{\Delta}_n) g \right)_{\mathcal H}\, dt.
\ene
Further, as in the proof of Lemma 1 in page 92 of \cite{ya} we prove, taking the Fourier  transform that \eqref{wa.2c} is equivalent to
\beq\label{wa.3}\begin{array}{r}
\left( \Omega_\pm  E_{\mathcal A_0}(\tilde{\Delta}_j) f,  E_{\mathcal A}(
\tilde{\Delta}_n) g \right)_{\mathcal H}= \lim_{\varepsilon \downarrow  0}   \frac{\varepsilon}{\pi} \int_{\mathbb R}\,\left( R_{\mathcal A_0}(\beta\pm i\varepsilon) E_{\mathcal A_0}(\tilde{\Delta}_j) f,\right. \\ \left. R_{\mathcal A}(\beta\pm i\varepsilon)  E_{\mathcal A}(\tilde{\Delta}_n) g \right)_{\mathcal H}\, d\beta.
\end{array}
\ene
Then, it is enough to prove that \eqref{wa.3} holds and that,
\beq\label{wa.3b}
\left( \Omega_\pm  E_{\mathcal A_0}(\tilde{\Delta}_j) f,  E_{\mathcal A}(
\tilde{\Delta}_n) g \right)_{\mathcal H}= \left( \mathcal F  E_{\mathcal A_0}(\tilde{\Delta}_j) f, \mathcal F_\pm  E_{\mathcal A}(
\tilde{\Delta}_n) g\right)_{\mathcal H_{\rm sp}}.
\ene
 We proceed to prove \eqref{wa.3} and \eqref{wa.3b}. We first prove that when $j\neq n$ the limit in the right-hand side of \eqref{wa.3} is zero. Remark that as $\tilde{\Delta}_j \cap \tilde{\Delta}_n=\emptyset,$ also the right-hand side of \eqref{wa.3b} is zero. Take $L_0$ such that $\tilde{\Delta}_j\cup\tilde{\Delta}_n \subset (-L_0/2, L_0/2).$ Then, 
  \beq\label{wa.3bxx}\begin{array}{c}
 \frac{\varepsilon}{\pi} \int_{|\beta| \geq L  } \left| \left( R_{A_0}(\beta\pm i\varepsilon) E_{\mathcal A_0}(\tilde{\Delta}_j) f,\right. \right.  \left. \left. R_{\mathcal A}(\beta\pm i\varepsilon)  E_{\mathcal A}(\tilde{\Delta}_n) g \right)_{\mathcal H}\right |\, d\beta \leq \\ 
C \int_{|\beta| \geq L}  \frac{1}{|\beta|^2} d\beta \leq \frac{C}{L}, \qquad L \geq L_0.
\end{array}
\ene
 By the spectral theorem (see Section five of Chapter six of \cite{kato} and Theorem VIII.6 in page 263 of \cite{rs1})
$$
\|R_{\mathcal A}(\beta\pm i\varepsilon)  E_{\mathcal A}(\tilde{\Delta}_n) \| \leq C, \qquad \beta \in  \mathbb R \setminus \Delta_n.
$$
Then, by Schwarz's inequality and the spectral theorem,
\beq\label{wa.3ax}\begin{array}{c}
 \frac{\varepsilon}{\pi} \int_{[\mathbb R \setminus \Delta_n]\cap [-L,L]}  \left|\left( R_{A_0}(\beta\pm i\varepsilon) E_{\mathcal A_0}(\tilde{\Delta}_j) f,\right. \right.  \left. \left. R_{\mathcal A}(\beta\pm i\varepsilon)  E_{\mathcal A}(\tilde{\Delta}_n) g \right)_{\mathcal H}\right |\, d\beta \leq \\
C   \int_{[\mathbb R \setminus \Delta_n]\cap [-L,L]}  \left[ \int_{\tilde{\Delta}_j} \frac{\varepsilon^2}{(\beta-\lambda)^2+\varepsilon^2} \frac{d}{d \lambda}
(E_{\mathcal A_0}(\lambda)f,f)_{\mathcal H} \, d\lambda \right]^{1/2}\, d\beta . 
\end{array}
\ene
Hence, by Lebesgue dominated convergence theorem,
\beq\label{wa.3c}
\ds \lim_{\varepsilon \downarrow  0}   \frac{\varepsilon}{\pi} \int_{[\mathbb R \setminus \Delta_n]\cap [-L,L]} \left|\left( R_{A_0}(\beta\pm i\varepsilon) E_{\mathcal A_0}(\tilde{\Delta}_j) f,\right. \\ \left. R_{\mathcal A}(\beta\pm i\varepsilon)  E_{\mathcal A}(\tilde{\Delta}_n) g \right)_{\mathcal H}\right|\, d\beta=0.
\ene
By \eqref{wa.3bxx} and \eqref{wa.3c},
\beq\label{wa.3cc}
\ds \lim_{\varepsilon \downarrow  0}   \frac{\varepsilon}{\pi} \int_{\mathbb R \setminus \Delta_n} \left |\left( R_{A_0}(\beta\pm i\varepsilon) E_{\mathcal A_0}(\tilde{\Delta}_j) f,\right.\right. \\ \left. \left.R_{\mathcal A}(\beta\pm i\varepsilon)  E_{\mathcal A}(\tilde{\Delta}_n) g \right)_{\mathcal H}\right|\, d\beta=0.
\ene
We similarly prove,
\beq\label{wa.3ccee}\begin{array}{l}
\ds \lim_{\varepsilon \downarrow  0}   \frac{\varepsilon}{\pi} \int_{\mathbb R \setminus \Delta_j} \left |\left( R_{A_0}
(\beta\pm i\varepsilon) E_{\mathcal A_0}(\tilde{\Delta}_j) f,
  R_{\mathcal A}(\beta\pm i\varepsilon)  E_{\mathcal A}(\tilde{\Delta}_n) g \right)_{\mathcal H}\right |\, d\beta=0.
\end{array}
\ene
By \eqref{wa.3cc} and \eqref{wa.3ccee} the limit in the right-hand side  of \eqref{wa.3} is zero for $j \neq n.$
We consider now the case $j=n.$ 
Since $\mathcal M$ has measure zero,
\beq\label{wa.3cd}
   \frac{\varepsilon}{\pi} \int_{ \mathcal M} \left( R_{A_0}(\beta\pm i\varepsilon) E_{\mathcal A_0}(\tilde{\Delta}_j) f,\right. \\ \left. R_{\mathcal A}(\beta\pm i\varepsilon)  E_{\mathcal A}(\tilde{\Delta}_j) g \right)_{\mathcal H}\, d\beta=0.
\ene
Moreover, by  \eqref{sp.19} and \eqref{s.66}
\beq\label{wa.3d}\begin{array}{l}
\ds \frac{\varepsilon}{\pi} \int_{ \mathbb R \setminus \mathcal M}\,d\beta\,\left( R_{\mathcal A_0}(\beta\pm i\varepsilon) E_{\mathcal A_0}(\tilde{\Delta}_j) f,\right.  \left. R_{\mathcal A}(\beta\pm i\varepsilon)  E_{\mathcal A}(\tilde{\Delta}_j) g \right)_{\mathcal H}= \int_{  \mathbb R \setminus \mathcal M} \, d\beta \\[10pt]
 \int_{\tilde{\Delta}_j} d\gamma \,\frac{1}{\pi} \frac{\varepsilon} {(\beta-\gamma)^2+\varepsilon ^2}  (  \mathcal L_{j}(\gamma) \mathbf F f,   \mathcal L_{j}(\gamma) \mathbf F \left( I-\mathcal B R_{\mathcal A_0}(\beta\pm i\varepsilon) \right)^{-1}    E_{\mathcal A}(\tilde{\Delta}_j) g )_{\mathbb C_{m_j}}.         
 \end{array}
 \ene
Moreover, by  Theorem~\ref{theo.s2b}, \eqref{s.44}, \eqref{s.46}, \eqref{wa.3d}, item (a)  of Proposition~\ref{propder}, and Theorem 7 in Subsection 3 of Section 2 of Chapter 1 of \cite{ya} ( see also Theorem 5 and Remark 6 in  Chapter  IV of 
\cite{kur}, Theorem 6 in  Section 2 of Chapter 1 of \cite{ya}, and also Section 19 of Chapter 2 of \cite{mu})
\beq\label{wa.3e}\begin{array}{l}
\ds\lim_{\varepsilon \downarrow  0}   \frac{\varepsilon}{\pi} \int_{ \mathbb R \setminus \mathcal M}\,\left( R_{\mathcal  A_0}(\beta\pm i\varepsilon) E_{\mathcal A_0}(\tilde{\Delta}_j) f,\right.  \left. R_{\mathcal A}(\beta\pm i\varepsilon)  E_{\mathcal A}(\tilde{\Delta}_j) g \right)_{\mathcal H}\, d\beta=    \int_{  \tilde{\Delta}_j} \,d\beta \\
 \left(  \mathcal L_{j}(\beta) \mathbf F E_{\mathcal A_0}(\tilde{\Delta}_j) f , \right.  \left.\mathcal L_{j}(\beta) \mathbf F  \left( I-\mathcal (B R_{\mathcal A_0})(\beta\pm i 0) \right)^{-1}    E_{\mathcal A}(\tilde{\Delta}_j) g \right)_{\mathbb C_{m_j}}.        
 \end{array}
 \ene
 Further, by \eqref{sp.19}, \eqref{s.75c},  and \eqref{wa.3e}
 \beq\label{wa.3f}\begin{array}{l}
\ds \lim_{\varepsilon \downarrow  0}   \frac{\varepsilon}{\pi} \int_{\mathbb R \setminus \mathcal M }\,\left( R_{\mathcal A_0}( \beta\pm i\varepsilon) E_{\mathcal A_0}(\tilde{\Delta}_j) f,\right.  \left. R_{\mathcal A}(\beta\pm i\varepsilon)  E_{\mathcal A}(\tilde{\Delta}_j) g \right)_{\mathcal H}\, d\beta=\\
\left( \mathcal F  E_{\mathcal A_0}(\tilde{\Delta}_j)f , \mathcal F_\pm E_{\mathcal A}(\tilde{\Delta}_j  ) g\right)_{\mathcal H_{\rm sp}}.
\end{array}
\ene
By \eqref{wa.3cd}  and \eqref{wa.3f} equations \eqref{wa.3} and \eqref{wa.3b} hold. This completes the proof of the theorem
\end{proof}
In the following theorem we prove our results on the existence and the completeness of  the wave operators.
\begin{theorem}\label{waop}
The wave operators,
\beq\label{wa.6}
W_\pm= {\rm s-} \lim_{t\to \pm \infty} e^{it \mathcal A} e^{-it\mathcal A_0},
\ene
where $ {\rm s-} \lim_{t\to \pm \infty} $ means the strong limit as $t\to \pm \infty,$  exist and are complete. They are isometric from $\mathcal H$ onto $\mathcal H_{\rm ac}(\mathcal A).$ Moreover, the intertwining relations hold,
\beq\label{wa.7}
\phi(\mathcal A_{\rm ac}) W_\pm= W_\pm \phi(\mathcal A_0),
\ene
for every Borel function $\phi.$ Furthermore, the stationary formulae for the wave operators are valid,
\beq\label{wa.8}
W_\pm= \mathcal F^\ast_\pm \mathcal F.
\ene
\end{theorem} 
\begin{proof}
By Theorem~\ref{pwa.1} the $\Omega_\pm$ are isometric from $\mathcal H$ onto  $\mathcal H_{\rm ac}(\mathcal A).$
Then, by Corollary 2 in page 74 and the remarks in page 76 of \cite{ya}, the wave operators $W_\pm $   exist,  are complete, and coincide with the $\Omega_\pm.$ They are isometric from $\mathcal H$ onto $\mathcal H_{\rm ac}(\mathcal A).$ Furthermore, by \eqref{wa.3a} the stationary formulae \eqref{wa.8} hold. Finally, by Theorem 4 in page 69 of \cite{ya} the intertwining relations \eqref{wa.7} are valid for any bounded Borel function $\phi.$
 In the general case where $\phi $ is not bounded,  \eqref{wa.7} follows from \eqref{wa.7} in the bounded case, approximating $\phi$ by 
$ \phi_n$ where $ \phi_n(\lambda)= \phi(\lambda)$ if $|\phi(\lambda)| \leq n$ and $ \phi_n(\lambda)=0,$ if 
$|\phi(\lambda)|  >n,$ for $n=1,\dots.$
\end{proof}

We proceed to prove Birman's invariance principle of the wave operators. For this purpose we make explicit the dependence of the wave operators on $\mathcal A$ and $\mathcal A_0,$   and denote by $W_\pm(\mathcal A, \mathcal A_0)$ the wave operator for the pair $\mathcal A, \mathcal A_0.$ By  \eqref{2.81c} either,
\beq\label{wa.9}
\sigma(\mathcal A_0)= \sigma_{\rm ess}(\mathcal A)=\sigma_{\rm ac}(\mathcal A)=\left[ (2\pi / T(E_{0}))^2, \infty\right),
\ene
or
\beq\label{wa.10}
\sigma(\mathcal A_0)= \sigma_{\rm ess}(\mathcal A)= \sigma_{\rm ac}(\mathcal A)= \cup_{l=1}^N  J_l \cup [\alpha_{N+1},\infty),
\ene
where $J_l= [\alpha_l, \gamma_l]$, for $l=1,\dots, N$  with   $\alpha_1:= (2\pi / T(E_{0}))^2 < \alpha_2, \dots < \alpha_N,$ and $  \gamma_1 < \gamma_2 < \dots <\gamma_N.$ Further,  $ \alpha_l < \gamma_l,$ for $l=1,\dots,N.$ 
Moreover, $J_l \cap J_k = \emptyset, $
for $l \neq k,$ with $l,k= 1,\dots,N.$ Furthermore, $\alpha_{N+1} > \gamma_N .$
The case  \eqref{wa.9} is valid if and only if  the {\it no gap condition}, $T(E_{\rm max}) > 2
T(E_{\rm min})$ holds. In the case \eqref{wa.9} we define,
\beq\label{wa.11}
\Lambda:= \left(\left (2\pi / T(E_{0})\right)^2, \infty\right),
\ene
and in the case \eqref{wa.10}  we designate,
\beq\label{wa.12}
\Lambda:= \cup_{N=1}^N (\alpha_l, \gamma_l) \cup (\alpha_{N+1}, \infty).
\ene 
We have that (see page 357 of \cite{kato}), 
$$
E_{\mathcal A}(\mathbb R \setminus \sigma(\mathcal A))= 0.
$$
Then, to define $\phi(\mathcal A)$ is is enough to defined $\phi$ on $\sigma(\mathcal A).$  Let $\phi$ be a real valued Borel function defined on $\sigma(\mathcal A)$   Then, with $\Lambda$ as in \eqref{wa.11} or \eqref{wa.12} $\sigma(\mathcal A)\setminus \Lambda $ is countable. It follows that, the Lebesgue measure of  $ \sigma(\mathcal A) \setminus \Lambda$ and of $ \phi( \sigma(\mathcal A) \setminus \Lambda)$ are zero. Let us represent $\Lambda$ in \eqref{wa.11} or in \eqref{wa.12} as 
\beq\label{wa.13}
\Lambda= \cup_{n=1}^Q \Lambda_n,
\ene
where the $\Lambda_n,$ for $n=1,\dots, Q$ are  nonintersecting intervals, and $Q$ is a positive integer, or $Q=\infty.$ Note that the $\Lambda_n$ do not need to coincide with the intervals in \eqref{wa.11} or in \eqref{wa.12}. Suppose that $\phi$ is absolutely continuous in each of the $\Lambda_n, n=1,\dots,$ and that $\phi'(\beta) >0,$ for almost every $\beta$ in $\Lambda_n.$
Then, by Lemma 1 in page 87 of \cite{ya}
\beq\label{wa.14}
\mathcal H_{\rm ac}(\phi(\mathcal A))= \mathcal H_{\rm ac}(\mathcal A),
\ene
and
\beq\label{wa.15}
\mathcal H_{\rm ac}(\phi(\mathcal A_0))= \mathcal H_{\rm ac}(\mathcal A_0)= \mathcal H.
\ene
  Note that in the notation of page 86 of \cite{ya} we are taking $M=\sigma(A).$
  
  Birman's invariance principle is the following theorem.  
  \begin{theorem} \label{bir}
  Suppose that $\phi$ is a real valued Borel function defined in $\sigma(\mathcal A),$ that is absolutely continuous on each of the $\Lambda_n$ in \eqref{wa.13} and that $\phi'(\beta) >0,$ for almost every $\beta \in \Lambda_n,$ for $n=1,\dots, Q.$ Then, the wave operators $W_\pm(\phi(\mathcal A), \phi(\mathcal A_0))$ exist and,
  \beq\label{wa.16}
  W_\pm(\phi(\mathcal A), \phi(\mathcal A_0))= W_\pm(\mathcal A, \mathcal A_0).
  \ene
  In particular, the wave operators $W_\pm(\phi(\mathcal A), \phi(\mathcal A_0))$ are complete. They are isometric with final subspace $\mathcal H_{\rm ac}(\phi(\mathcal A))= \mathcal H_{\rm ac}(\mathcal A).$
  \end{theorem}
  \begin{proof} The Abelian wave operators, $\Theta_\pm(\mathcal A, \mathcal A_0)$ are defined by the following limit in page 76 of \cite{ya},
  \beq\label{wa.17}
  \lim_{\varepsilon \downarrow 0} \int_0^\infty\varepsilon e^{-\varepsilon t} \| e^{i\pm t\mathcal A}  e^{\mp it \mathcal A_0 t} f- \Theta_\pm(\mathcal A,\mathcal A_0) f\|^2_{\mathcal H} dt, \qquad f \in \mathcal H.
  \ene
   By the remarks in page 76 of \cite{ya}, as by Theorem~\ref{waop} the wave operators $W_\pm(\mathcal A, \mathcal A_0)$ exist, the Abelian wave operators exist and,
  \beq\label{wa.18}
  \Theta_\pm(\mathcal A,\mathcal A_0)= W_\pm(\mathcal A, \mathcal A_0).
  \ene
  Then, by Theorem 1 in page 110 of \cite{ya} the Abelian wave operators $\Theta_\pm(\phi(\mathcal A),\phi(\mathcal A_0))$ exist and, furthermore,
  \beq\label{wa.19}
  \Theta_\pm(\phi(\mathcal A),\phi(\mathcal A_0))=   \Theta_\pm(\mathcal A,\mathcal A_0)= W_\pm(\mathcal A, \mathcal A_0).
  \ene
  Further,  by Theorem~\ref{waop} the wave operators $W_\pm(\mathcal A, \mathcal A_0)$ are isometric. Then, it follows from \eqref{wa.19} that the wave operators $  \Theta_\pm(\phi(\mathcal A),\phi(\mathcal A_0))$ are isometric. Hence,  by the remarks in page 76 of \cite{ya} the wave operators  $W_\pm(\phi(\mathcal A), \phi(\mathcal A_0))$ exist and
  \beq\label{wa.20}
   W_\pm(\phi(\mathcal A), \phi(\mathcal A_0))= \Theta_\pm(\phi(\mathcal A),\phi(\mathcal A_0)).
  \ene
  Equation \eqref{wa.16} follows from \eqref{wa.19} and \eqref{wa.20}. Finally, by \eqref{wa.16} the wave operators $  W_\pm(\phi(\mathcal A), \phi(\mathcal A_0))$ are isometric and complete with final subspace  $\mathcal H_{\rm ac}(\phi(\mathcal A))=\mathcal H_{\rm ac}(\mathcal A)$ since this is true for the wave operators $W_\pm(\mathcal A,\mathcal A_0).$
  \end{proof}
The scattering operator is defined as follows,
\beq\label{wa.21}
S:= W_+^\ast W_-.
\ene 
 By Theorem~\ref{waop} $S$  is unitary on $\mathcal H.$
 \section{Landau damping and large time dynamics}\label{ladab}
 In this section we obtain our results in Landau damping and in the long time asymptotic behaviour of the 
 phase-space density. We first prepare results that we need. We begin by considering the quadratic form associated to $\tilde{\mathcal A}.$  
Recall that the quadratic from of $\tilde{\mathcal  A}_0$ is given by \eqref{new.12}.
We define the quadratic form
 \beq\label{wa.23}
  \tilde{a}(f,g)= \tilde{a}_0(f,g)-\left( \tilde{\mathcal B} f, g\right)_{\tilde{\mathcal H}}, f,g\in D[\tilde{a}]= D[\tilde{a}_0]:= D[\tilde{\mathcal D}_o].
  \ene
   As $\tilde{\mathcal B}$ is selfadjoint and bounded, $\tilde{a}$ is  closed, symmetric, and bounded below. By the first representation theorem (see Theorem 2.1 in page  322 of \cite{kato})
  $\tilde{a}$ is the quadratic form of $\tilde{\mathcal A}.$ Further, by \eqref{2.81a} and item (c) of Theorem~\ref{eigen} $\tilde{\mathcal A} \geq \delta >0.$ Hence, also $ \tilde{a} \geq \delta  >0.$ By the second representation theorem (Theorem 2.23 in page 331 of \cite{kato} )
  \beq\label{wa.23b}
 D[\sqrt{\tilde{\mathcal A}}]= D[\tilde{a}]= D[\tilde{\mathcal D}_0],
 \ene
  and
  \beq\label{wa.24}
  \tilde{a}(f,g)= \left(  \sqrt{\tilde{A}} f, \sqrt{\tilde{A}}  g\right)_{\tilde{\mathcal H}}, \qquad f, g \in  D\left[\sqrt{\tilde{A}}\right]= D[\tilde{a}]= D[\tilde{\mathcal D}_{\rm o}].
  \ene 
   Recall that the gravitational potential $U\tilde{\mathcal D}_0 f$ induced by a phase-space density $f\in D[\tilde{\mathcal D}_0]$  is given by,
  \begin{eqnarray}\label{wa.25}
( U \tilde{\mathcal D}_0 f)(x)= 2\pi \int_{\mathbb R}|x-y| \rho(y)\, dy=  2\pi \int_{-R_0}^{R_0}|x-y| \rho(y)\, dy,\\
\rho (x)= \int (\tilde{\mathcal D}_0 f)(x,v)\, dv,\label{wa.25b}
\end{eqnarray}  
  Further, we used that  the support of $\rho$ is contained in $[-R_0,R_0].$   At the beginning of page 688 of \cite{hrs} it is proved,
 \beq\label{wa.26b}
 \int_{-R_0}^{R_0} \, \rho(x) \,dx=0.
 \ene
It follows from the second equality in \eqref{wa.25},  \eqref{wa.26b}, and using that $\rho$ is an even  function that the support of $U \tilde{\mathcal D}_0 f $ is contained
in $[-R_0,R_0].$
Furthermore, using \eqref{wa.26b} it is proved at the beginning of page 688 of \cite{hrs},
  \beq\label{wa.26}
\partial_x ( U \tilde{\mathcal D}_0 f)(x)=4 \pi \int_{-R_0}^x \rho (y) \,dy, 
  \ene
 and that  $\partial_x ( U \tilde{\mathcal D}_0 f)(x) =0,$ for  $x \in \mathbb R \setminus [-R_0,R_0].$

 Moreover,  it is proven in equation (A.4) of \cite{hrs} that 
  \beq\label{wa.27}
  \partial_x (U\tilde{\mathcal D}_{\rm 0} f)(x)=  4 \pi \int_{\mathbb R} v f(x,v) dv.
  \ene
We denote by $\mathcal K$ the following operator defined on $\tilde{\mathcal H},$
\beq\label{wa.28}
\mathcal K f:= \partial_x U\left( \tilde{\mathcal D}_0 \frac{1}{\sqrt{\tilde{\mathcal A}}}f\right), \qquad f \in \tilde{\mathcal H}.
\ene
In the following proposition we prove that $\mathcal K$ is a compact operator from  $\tilde{\mathcal H}$  into $L^2((-R_0,R_0)).$
 \begin{proposition}\label{comp}
 The operator $\mathcal K$ defined in \eqref{wa.28} is compact from  $\tilde{\mathcal H}$ into $L^2((-R_0,R_0)).$
 \end{proposition}
 \begin{proof}
 Let $\{f_n\}$ be a bounded sequence in $\tilde{\mathcal H}.$ Since $D[\tilde{\mathcal A}^{1/2}]= D[\tilde{\mathcal D}_0],$ the operator $\tilde{\mathcal D}_0  \tilde{\mathcal A}^{-1/2}$ is bounded. Then, for some constant $C,$
 \beq\label{wa.28b}
 \|\tilde{\mathcal D}_0  \tilde{\mathcal A}^{-1/2} f_n\|_{\tilde{\mathcal H}_{\rm even}} \leq C, \qquad n=1,\dots.
 \ene
 Recall that for $n=1,\dots$ the support of $\partial_x U(\tilde{\mathcal D}_0 \tilde{\mathcal A}^{-1/2} f_n)$ is contained in $[-R_0, R_0].$  Moreover, by equation (A.3) of \cite{hrs} and \eqref{wa.28b}
 \beq\label{wa.31}
 \begin{array}{l}
\left \|\partial_x (U\tilde{\mathcal D}_0   \tilde{\mathcal A}^{-1/2} f_n) \right\|_{L^2((-R_0,R_0))}+ \left\|\partial^2_{x}  (U\tilde{\mathcal D}_0   \tilde{\mathcal A}^{-1/2} f_n)\right \|_{L^2((-R_0,R_0))}\leq \\
 C\left\| \tilde{\mathcal D}_0   \tilde{\mathcal A}^{-1/2} f_n\right\|_{\tilde{\mathcal H}_{\rm even}} \leq C. 
\end{array}
 \ene
Then, by \eqref{wa.31}
  $\{\partial_x U(\mathcal D    \tilde{\mathcal A}^{-1/2} f_n)\}$ is a bounded sequence in the Sobolev space $H_1((-R_0,R_0)).$ Hence, by the   Rellich-Kondrachov local compactness theorem (see Theorem 6.2 in page 144 of \cite{adams}) there is a subsequence of   $\{\partial_x U(\mathcal D    \tilde{\mathcal A}^{-1/2} f_n)\}$ that is convergent in $L^2((-R_0, R_0)).$ This proves that $\mathcal K$ is compact.
  \end{proof}

  Observe that if $f \in \mathcal H_{\rm ac}(\tilde{\mathcal A}),$ as is well known, $e^{-it \tilde{\mathcal A}} f$ tends weakly to zero as $ t \to \pm \infty.$ This is immediate since by the spectral theorem (see Section five of Chapter six of \cite{kato} and Theorem VIII.6 in page 263 of \cite{rs1})

 $$
 \left( e^{-i t \tilde{\mathcal A}} f,g\right)_{\tilde{\mathcal H}}= \int_{\sigma_{\rm ac}(\tilde{\mathcal A)}}\, e^{-i t \lambda}
\frac{d}{d\lambda}\left(E_{\tilde{\mathcal A}})(\lambda)f, g\right)_{\tilde{\mathcal H}} d\lambda, \qquad g \in 
\tilde{\mathcal H}.
$$
Since $\frac{d}{d\lambda}\left(E_{\mathcal A}(\lambda)h, q\right)_{\mathcal H}$ is integrable, it follows from the Riemann- Lebesgue lemma that
\beq\label{wa.33}
\lim_{t \to \pm \infty}  \left( e^{-i t \tilde{\mathcal A}} f,g\right)_{\tilde{\mathcal H}}=0. \qquad g \in \tilde{\mathcal H}.
\ene
Now we state our results in the gravitational Landau damping.
\begin{theorem}\label{damping}
Let $\delta f$ be a solution to the linearized gravitational Vlasov-Poisson system  \eqref{2.5}-\eqref{2.7}. Let us define $\delta f_\pm$ as in \eqref{2.9}, \eqref{2.10}, \eqref{2.21b}, and \eqref{2.21c}. Assume that  $\delta f_+(t,x,v) \in D[\tilde{\mathcal D}_{\rm o}^\dagger],$ and that  $\delta f_+(0,x,v)=0.$  Further, suppose that  $\delta f_-(t,x,v)  \in D[\tilde{\mathcal A}] \cap \tilde{\mathcal H}_{\rm ac}(\tilde{\mathcal A}).$ Let $(U\delta f)(t,x)= (U\delta f_+)(t,x)$ 
be the gravitational potential induced by $\delta f.$ Then, the gravitational force $({\mathbf F}\delta f)(t,x): =- \partial_x(U \delta f)(t,x)= -\partial_x( U\delta f_+)(t,x)$
 satisfies.
\begin{enumerate}
\item[\rm (a)]
\beq\label{wa.34}
 \lim_{t \to \pm \infty} \|(\mathbf F\delta f)(t,  \cdot )\|_{\ds L^2((-R_0,R_0))}=0,
\ene

\item[\rm (b)]
\beq\label{wa.36}
 \lim_{t\to \pm \infty} \|\partial_t (\mathbf F\delta f)(t,\cdot)\|_{\ds L^2((-R_0, R_0))}=0.
 \ene

 \end{enumerate}
\end{theorem}
\begin{proof}
As $\delta f_-(t,x,v) \in D[\mathcal A]$ we have that   $\delta f_-(t,x,v)$ is a strong solution to \eqref{2.19}. Moreover, since  $\delta f_+(0)=0,$ it follows from \eqref{2.12} that $ \partial_t\delta f_-(0, x,v)=0.$ Hence, $ \delta f_-(t,x,v)$ is given by,

 \beq\label{wa.32}
 \delta f_-= \cos\left(\sqrt{\tilde{\mathcal A}}t\right)f_0,
 \ene
 where $ f_0(x,v):= \delta f_-(0,x,v).$ Further, by  \eqref{2.11} and \eqref{wa.32}
\beq\label{wa.32b}
\delta f_+(t)= \tilde{\mathcal D}_o\left( - \tilde{\mathcal A}^{-1/2} \sin\left(\sqrt{\tilde{\mathcal A}}\,t\right)f_0\right).
\ene

By \eqref{wa.28} and \eqref{wa.32b}
\beq\label{wa.32bc}
\mathbf F(t):= \mathcal K  \sin(\sqrt{\tilde{\mathcal A}} \,t)f_0= \mathcal K \frac{1}{2i} \left(e^{i\sqrt{\tilde{\mathcal A}}\,t}- e^{-i\sqrt{\tilde{\mathcal A}}\,t}\right) f_0,
\ene
As by \eqref{wa.33}   $e^{\pm i \sqrt{\tilde{A}}\, t}  f_0 $ tends weakly to zero as $t \to \pm \infty$  and 
$\mathcal K$ is compact  from $\tilde{\mathcal H}$ into  $L^2([-R_0,R_0]),$  it follows from \eqref{wa.32bc} that \eqref{wa.34} holds.
Let us now prove (b). By the first equality in  \eqref{wa.32bc},
\beq\label{wa.40}
\partial_t \mathbf F= \mathcal K \sqrt{\tilde{\mathcal A}}\cos(\sqrt{\mathcal A}\,t)f_0  .
\ene  
Further, 
 \beq\label{wa.40zzx}
\sqrt{\tilde{\mathcal A}} \cos(\sqrt{\mathcal A}\,t) f_0= \frac{1}{2}\ds  \left (e^{i\sqrt{\tilde{\mathcal A}}\,t}+ e^{-i\sqrt{\tilde{\mathcal A}}\,t}\right)  \sqrt{\tilde{\mathcal A}} f_0.
 \ene
 By \eqref{wa.33} and \eqref{wa.40zzx}, we have that  $  \cos(\sqrt{\mathcal A}\,t)  \sqrt{\tilde{\mathcal A}} f_0$ tends weakly to zero in $\tilde{\mathcal H}$ as $t\to \pm \infty.$ Then, as $\mathcal K$ is compact from $\tilde{\mathcal H}$ into  $L^2([-R_0,R_0]),$ equation \eqref{wa.36} follows from \eqref{wa.40}. 
\end{proof}
In the following corollary we prove that the  potential $U\delta f$ and its time derivative tend to zero as time tends to $\pm \infty.$ Recall that the support of $U\delta f$ is contained in $[-R_0, R_0].$
\begin{corollary}\label{decaypot}
Let $\delta f$ be a solution to the linearized gravitational Vlasov-Poisson system  \eqref{2.5}-\eqref{2.7} that satisfies the assumptions of Theorem~\ref{damping}. Further, let $(U\delta f)(t,x)= (U\delta f_+)(t,x) $ be the gravitational potential induced by $\delta f.$ Then, the following is true.
\begin{enumerate}
\item[\rm (a) ]
\beq\label{new.ccddee}
\left| (U\delta f)(t,x) \right| \leq \sqrt{2R_0}\, \|\mathbf F(\delta f)(t,\cdot)\|_{L^2((-R_0,R_0))}\to 0,\quad  t \to \pm \infty, x \in [-R_0, R_0].
\ene
\item[\rm (b) ]
\beq\label{new.ccddff}\begin{array}{c}
\left| \partial_t (U\delta f)(t,x) \right| \leq \sqrt{2R_0} \,\|\partial _t\mathbf F(\delta f)(t,\cdot)\|_{L^2((-R_0,R_0))}\to 0,  \\t \to \pm \infty, x \in [-R_0,R_0].
\end{array}
\ene
\end{enumerate} 
\end{corollary}\begin{proof}
Since $ (U\delta f)(t, -R_0)=0,$
\beq\label{new.llmm}
U(\delta f)(t,x)= -\int_{-R_0}^x  \mathbf F(\delta f)(t,y)\, dy.
\ene
Then, \eqref{new.ccddee} follows from \eqref{wa.34} and Schwarz's inequality. Equation \eqref{new.ccddff} follows in the same way taking the time derivative of  \eqref{new.llmm} and using \eqref{wa.36}.
\end{proof}

Finally, let   $\delta f$ be a solution to the linearized gravitational Vlasov-Poisson system  \eqref{2.5}-\eqref{2.7}
as in Theorem~\ref{damping}. We define the wave operators in $\tilde{\mathcal H}$ as follows,
\beq\label{wa.41a}
\tilde{W}_\pm:= {\rm s-} \lim_{t\to \pm \infty} e^{it \tilde{\mathcal A}} e^{-it\tilde{\mathcal A}_0}=
\mathbf U^\dagger W_\pm \mathbf U.
\ene
Further, by \eqref{wa.16} and \eqref{wa.41a}
\beq\label{wa.41b}
  W_\pm(\phi(\tilde{\mathcal A}), \phi(\tilde{\mathcal A}_0))= W_\pm(\tilde{\mathcal A}, \tilde{\mathcal A}_0).
  \ene
The wave operators $\tilde{W}_\pm$  are complete since they are unitarily equivalent to the $W_\pm.$ Then,  for any
$f_0 \in \tilde{\mathcal H}_{\rm ac}(\tilde{\mathcal A})$ there are $g_{\pm,0} \in \tilde{\mathcal H}$ such that,   
\beq\label{wa.41} 
 g_{\pm,0}= \tilde{W}_\pm^\ast\, f_0.
 \ene
 We denote,
 \beq\label{wa.42}
 g_\pm(t):=\frac{1}{2}\left( e^{ - i t \sqrt{\tilde{A}_0}} g_{\pm, 0}+    e^{ it \sqrt{\tilde{A}_0}} g_{\mp,0}  \right).
 \ene
 Then, by  \eqref{wa.32} and  Birman's invariance principle \eqref{wa.41b}  with $\phi(\lambda):= \sqrt{\lambda},$
 \beq\label{wa.43}
\lim_{t \to \pm \infty}\| \delta f_-(t) - g_\pm(t)\|_{\tilde{\mathcal H}}=0.
\ene
That is to  say, for  $ t\to \pm \infty,$  the solution  $\delta f_-(t)$ is asymptotic to  the solutions $g_\pm(t)$ to the free Antonov wave equation,
$$
\partial^2_t g_\pm+\tilde{\mathcal A}_0 g_\pm=0.
$$ 
This implies that asymptotically  the odd part of the phase-space distribution function  follows the  orbits of the solutions to  Newton's equation for the gravitational potential $U_0,$ of the  steady state,  in the sense that $\delta f_-$ is  transported along these   orbits.

\section{Scattering for the Antonov wave equation}\label{scat}
In this section we consider the wave operators and the scattering operator for the Antonov wave equation \eqref{2.22}. We follows Section four of Chapter 3 of \cite{ya}. For related approaches see \cite{kato2} and  Section  10 in Chapter XI of\cite{rs3}. We introduce the new variables 
\beq\label{sc.1}
h_1(t):= \partial_t g, \qquad h_2(t)= (\sqrt{\tilde{\mathcal A}} g)(t).
\ene
Then \eqref{2.22} is equivalent to
\beq\label{sc.2}
i \partial h(t)= H h(t),
\ene 
where
\beq\label{sc.3}
h(t):= \begin{pmatrix} h_1(t)\\h_2(t)
\end{pmatrix}
\ene
and
\beq\label{sc4}
H:= \begin{bmatrix} 0 & -i \sqrt{\tilde{\mathcal  A}}\\   i \sqrt{\tilde{\mathcal  A}} &0\end{bmatrix}=\mathcal C
 \begin{bmatrix} \sqrt{\tilde{\mathcal A}} & 0\\  0& 
 -\sqrt{\tilde{\mathcal A}}\end{bmatrix} \mathcal C^\dagger,
 \ene
 with $\mathcal C$ the unitary matrix
 \beq\label{sc.5}
 \mathcal C:=\frac{1}{\sqrt{2}} \begin{bmatrix}  -i&i\\ 1&1\end{bmatrix}.
 \ene
 The operator $H$ is selfadjoint in the Hilbert space,
 \beq\label{sc.6}
 \mathbf H:= \tilde{\mathcal H}\oplus \tilde{\mathcal H},
 \ene
 with domain
 \beq\label{sc.7}
 D[H]:= D\left[\sqrt{\tilde{\mathcal A}}\right]\oplus  D\left[\sqrt{\tilde{\mathcal A}}\right].
 \ene
 The unperturbed Antonov wave equation is given by,
 \beq\label{sc.8}
 \partial_t^2 f+ \tilde{\mathcal A_0} f=0. 
 \ene
 We proceed as with \eqref{2.22}. We introduce the new variables 
\beq\label{sc.9}
h_1(t):= \partial_t f, \qquad h_2(t)= (\sqrt{\tilde{\mathcal A}_0} f)(t).
\ene
Then \eqref{sc.8} is equivalent to
\beq\label{sc.2a}
i \partial h(t)= H_0 h(t),
\ene 
where
\beq\label{sc.3b}
h(t):= \begin{pmatrix} h_1(t)\\h_2(t)
\end{pmatrix}
\ene
and
\beq\label{sc4b}
H_0:= \begin{bmatrix} 0 & -i \sqrt{\tilde{\mathcal  A}_0}\\   i \sqrt{\tilde{\mathcal  A}_0} &0\end{bmatrix}=\mathcal C
 \begin{bmatrix} \sqrt{\tilde{\mathcal A}_0} & 0\\  0& -\sqrt{\tilde{\mathcal  A}_0}\end{bmatrix} \mathcal C^\dagger,
 \ene
 with $\mathcal C$ the unitary matrix \eqref{sc.5}.
 The operator $H_0$ is selfadjoint in the Hilbert space \eqref{sc.6}, with domain
 \beq\label{sc.7b}
 D[H_0]:= D\left[\sqrt{\tilde{\mathcal A_0}}\right]\oplus  D\left[\sqrt{\tilde{\mathcal A}_0}\right].
 \ene
 Note that as $ \tilde{\mathcal H}_{\rm ac}(\sqrt{\tilde{\mathcal A}_0})=\tilde{\mathcal H}_{\rm ac}(\tilde{\mathcal A}_0)= \tilde{\mathcal H},$ we have that $\mathbf H_{\rm ac}( H_0)= \mathbf H.$
Then, by Theorems~\ref{waop}, and \ref{bir}, \eqref{wa.41a}, \eqref{wa.41b}, and Theorem  1 in page 109 of \cite{ya} the wave operators
\beq\label{sc.8b}
W_\pm(H, H_0):= {\rm s-}\lim_{t\to \pm \infty} e^{it H}\, e^{-it H_0}
 \ene
 exist, are isometric, and are complete, i.e.  they are onto $\mathbf  H_{\rm ac}( H).$ Further, by \eqref{sc4}
 \beq\label{sc.9b}
 \mathcal H_{\rm ac}(\mathbf H)=\left\{  h \in \mathbf H: h= \mathcal C g, \, {\rm where} \,g\in \tilde{\mathcal H}_{\rm ac}(\tilde{\mathcal A})\oplus  
 \tilde{\mathcal H}_{\rm ac}(\tilde{\mathcal A)}\right\}.
\ene
 Furthermore
 \beq\label{sc.10}
 W_\pm(\mathbf H, \mathbf H_0)= =\mathcal C
 \begin{bmatrix} W_\pm(\tilde{\mathcal A}, \tilde{\mathcal A}_0) & 0\\  0& W_\mp(\tilde{\mathcal A}, \tilde{\mathcal A}_0)\end{bmatrix} \mathcal C^\ast,
 \ene
and then
\beq\label{sc.11}
 W_\pm(\mathbf H, \mathbf H_0)= \frac{1}{2}
\begin{bmatrix} W_\pm(\tilde{\mathcal A}, \tilde{\mathcal A}_0)+W_\mp(\tilde{\mathcal A}, \tilde{\mathcal A}_0)  & i( W_\mp(\tilde{\mathcal A}, \tilde{\mathcal A}_0)- W_\pm(\tilde{\mathcal A}, \tilde{\mathcal A}_0))\\   i( W_\pm(\tilde{\mathcal A}, \tilde{\mathcal A}_0)- W_\mp(\tilde{\mathcal A}, \tilde{\mathcal A}_0))  & W_\pm(\tilde{\mathcal A}, \tilde{\mathcal A}_0)+W_\mp(\tilde{\mathcal A}, \tilde{\mathcal A}_0)\end{bmatrix}.
 \ene
Further, the scattering matrix,
\beq\label{sc.12}
\mathbf S (\mathbf H, \mathbf H_0):=  W_+(\mathbf H, \mathbf H_0)^\dagger  W_-(\mathbf H, \mathbf H_0), 
\ene
is unitary on $\mathbf H.$ 
The existence and the completeness of the wave operators \eqref{sc.10} implies that for every $ h \in \mathcal H_{\rm ac}(\mathbf H)$ the solution to the Antonov wave equation $e^{-it \mathbf H}h$ behaves as $t\to \pm \infty$
as a solution to the unperturbed Antonov wave equation,    $e^{-it \mathbf H_0} W_\pm(\mathbf H, \mathbf H_0)^\dagger h,$
\beq\label{sc.13}
\lim_{t\to \pm \infty}\| e^{-it \mathbf H}h - e^{-it \mathbf H_0} W_\pm(\mathbf H, \mathbf H_0)^\dagger h\|_{\mathbf H}=0.
\ene
\section{Conclusions}
\label{concl}

We studied the gravitational Vlasov-Poisson system linearized around steady states that are polytropes  and King steady states. These steady states are extensively used in astrophysics to study the dynamics of galaxies   and of
clusters of galaxies. Following a well established method  in astrophysics, we studied the linearized Vlasov-Poisson system, in an equivalent form,  by means of the Antonov wave equation. We obtained  comprehensive results in the spectral and scattering theory of the Antonov  operator. In particular, we identified the absolutely continuous spectrum of the Antonov operator and the part of the singular continuous spectrum that is embedded in the absolutely continuous spectrum. We constructed the generalized Fourier maps, we proved that the wave operators exist and are complete, and we obtained the stationary formulae for the wave operators.

With the help of these results we analyzed  the long time behaviour of the solutions to the linearized  gravitational Vlasov-Poisson system with initial values in the absolutely continuous subspace of the Antonov operator. We obtained a   precise description of the dynamics of  the stars in the galaxies, or of the galaxies in the clusters of galaxies, for large times.  We proved  that  the distribution functions  are asymptotic for large times  to the trajectories of the solutions to Newton's equation with the gravitational potential of the steady state, in the sense that they are transported along these trajectories.  Moreover,  we proved  that the  gravitational Landau damping holds. Namely, we   proved that the gravitational force and its time derivative, as well as the gravitational potential and its time derivative, tend to zero  for large times. Hence, the stars in the galaxies, or the galaxies in the clusters  of galaxies, behave  as matter that moves under the gravitational force of the steady state, up  to a small perturbation due to the selfconsistent gravitational potential that they induce. Further, this perturbation tends  to zero for large times.

\appendix

\section{Appendix}\label{apex}

\renewcommand{\theequation}{\thesection.\arabic{equation}}

\newtheorem{theorem2}{THEOREM}[section]
\renewcommand{\thetheorem}{\arabic{section}.\arabic{theorem}}

\newtheorem{prop2}[theorem2]{Proposition}
\newtheorem{lemma2}[theorem2]{LEMMA}

We begin the appendix  giving the  proof of a proposition about the spaces 
of H\"older continuous functions defined in a bounded open interval  and with values in a Hilbert space.  We include this proof to make the paper selfcontained. 

\begin{prop2}\label{propa.1}
Let $I$ be a bounded open  interval of real numbers, and $\mathcal G$ a separable Hilbert space. Then, for any $\alpha_1, \alpha_2,$ with $ 0< \alpha_1 < \alpha_2 \leq 1,$ we have,
\beq\label{a.2}
C_{\alpha_2}(I, \mathcal G) \subset \hat{C}_{\alpha_1}(I, \mathcal G),
\ene 
with the imbedding continuous.
\end{prop2}
\begin{proof}
By a change of coordinates $x \to ax+b$ with suitable $a, b, \in \mathbb R,$ we reduce the problem 
to the case  $I=(-\pi, \pi).$ As any function in $C_{\alpha_2}(I, \mathcal G)$ extends uniquely to a  function in $C_{\alpha_2}(\overline{I}, \mathcal G)$
we can assume that $f$ is H\"older continuous on $[-\pi, \pi].$ We decompose f as follows,
\beq\label{a.3}
f(x)= f_+(x)+ [f_-(x)-\frac{1}{\pi} f_-(\pi)x]+ \frac{1}{\pi}f_-(\pi)x,
\ene
where $f_\pm (x):= \frac{1}{2} (f(x)\pm f(-x)),$ are, respectively, the even and the odd parts of $f.$
As $\frac{1}{\pi}f_-(\pi)x$ belongs to $\hat{C}_{\alpha_1}(I,\mathcal G),$  $f_+(x),$ and $f_-(x)- \frac{1}{\pi}f_-(\pi)x$ take the 
same value at $\pm \pi$ the problem is reduced to the case of $f \in C_{\alpha_2}(I,\mathcal G)$ that satisfy 
$f(-\pi)= f(\pi).$ We can  use the (C,1) convergence of the Fourier series \cite{gr}. Remark that \cite{gr} considers real valued functions, but the same proof applies in the case of functions with values in a separable Hilbert space. We denote
\beq\label{a.4}
s_n(x)= \frac{a_0}{2}+\sum_{k=1}^n (a_k \cos kx+ b_k \sin kx) ,
\ene
where
\beq\label{a.5} 
a_k:= \frac{1}{\pi} \int_{-\pi}^\pi f(x) \cos kx dx,\qquad  k= 0, \dots,
\ene
and
\beq\label{a.6} 
b_k:= \frac{1}{\pi} \int_{-\pi}^\pi f(x) \sin kx dx,  \qquad k= 1, \dots.
\ene
Further, we designate
\beq\label{a.7}
\sigma_n(x)= \frac{1}{n} \sum_{l=0}^{n-1} s_l(x).
\ene
Then,  according to Section 12.2C of \cite{gr}
\beq\label{a.8}
\sigma_n(x)- f(x)= \frac{2}{\pi} \int_0^\pi \left[\frac{f(x+t)+f(x-t)}{2} -f(x)\right] K_n(t) dt,
\ene
where we have extended $f(x)$ periodically to $\mathbb R,$ and
\beq\label{a.9}
K_n(t):= \frac{(\sin n t/ 2)^2}{2n (\sin t/2)^2}.
\ene
Further,
\beq\label{a.10}
1= \frac{2}{\pi} \int_0^\pi K_n(t) dt.
\ene
Note that $\sigma_n \in C^\infty(I, \mathcal G).$ 
Further, by Theorem 12.3C of \cite{gr}
\beq\label{a.11}
\lim_{n\to \infty} \max_{x\in I} \| \sigma_n(x)-f(x)\|_{\mathcal G}=0.
\ene
Let us denote
\beq\label{a.12}
h_n(x):= \sigma_n(x)- f(x).
\ene
Then, by \eqref{a.8}
\beq\label{a.13}\begin{array}{c}
h_n(x_1)- h_n(x_2)= \ds \frac{2}{\pi} \int_0^\pi \ds \Big [ \frac{f(x_1+t)-f(x_2+t) +(f(x_1-t)-f(x_2-t))}{2} - \\
\ds  (f(x_1)-f(x_2)) \Big ] K_n(t) dt.
\end{array}
\ene
Hence, for any $\varepsilon >0,$ by \eqref{a.10} and as $f \in C_{\alpha_2}(I, \mathcal G),$
\beq\label{a.14}
\frac{1}{|x_1-x_2|^{\alpha_1} }\|h_n(x_1)- h_n(x_2)\|_{\mathcal G} \leq C |\varepsilon|^{\alpha_2-\alpha_1 }, \qquad |x_1-x_2|\leq \varepsilon,
\ene
where the constant $C$ in the right-hand side of \eqref{a.14} is independent of $\varepsilon.$
Moreover, for each fixed $\varepsilon >0,$ by  \eqref{a.11} and \eqref{a.12}, for each $\delta >0$ there is an $N$ such that,
\beq\label{a.15}
\frac{1}{|x_1-x_2|^{\alpha_1}} \|h_n(x_1)- h_n(x_2)\|_{\mathcal G} \leq \frac{1}{\varepsilon^{\alpha_1}} \delta, \qquad  |x_1-x_2|\geq \varepsilon,  n \geq N.
\ene
Given any $\rho>0$ we can take $\varepsilon$ so small that the right hand side of \eqref{a.14} is smaller than $\rho /2$ and then take $N$ so large  that the right hand side of \eqref{a.15} is smaller than ${\rho}/2.$ Then,
\beq\label{a.16}
 \frac{1}{|x_1-x_2|^{\alpha_1}} \|h_n(x_1)-h_n(x_2)\|_{\mathcal G}\leq \rho, \qquad n \geq N.
\ene 
By \eqref{a.11} and \eqref{a.16} $\sigma_n$ converges to $f$ in the norm of $C_{\alpha_1}(I,\mathcal G).$ This completes the proof of \eqref{a.2}. Moreover, it is immediate from the definition of $C^{\alpha}(I, \mathcal G)$
that $ C^{\alpha_2}(I, \mathcal G)\subset C^{\alpha_1}(I, \mathcal G)$ and,
$$
\|f\|_{C_{\alpha_1}(I, \mathcal G)} \leq C \|f\|_{C_{\alpha_2}(I, \mathcal G)}, \qquad   f\in C_{\alpha_2}(I, \mathcal G).
$$
Moreover, as $ \hat{C}_{\alpha_1}(I, \mathcal G) \subset  C_{\alpha_1}(I, \mathcal G)$ we get,
\beq\label{a.16b}
\|f\|_{\hat{C}_{\alpha_1}(I, \mathcal G)} \leq C \|f\|_{C_{\alpha_2}}(I, \mathcal G), \qquad   f\in C_{\alpha_2}(I, \mathcal G).
\ene
This proves that the imbedding of  $C_{\alpha_2}(I, \mathcal G)$  into  $\hat{C}_{\alpha_1}(I, \mathcal G)$ is continuous.
 \end{proof}
 
 In the following proposition we obtain properties of the period function and of the angle variable. Recall that the the angle $\theta(x,E)$ was defined in \eqref{2.36} for $ x_-(E)\leq x  \leq x_+(E),$ and  $E > E_{\rm min}.$  Moreover,  $\theta (x_-(E), E)=0,$ 
 and $ \theta(x_+(E),E) =1/2.$  To simplify the statement of the next proposition we agree to extend the definition of $ \theta(x,E)$ as follows: $\theta(x, E)=0,$ for $ x \leq x_-(E),$ and $ \theta (x,E)= 1/2,$ for $x\geq x_+(E).$
\begin{prop2}\label{propder}
We have that:

\begin{enumerate}
\item[{\rm (a)}]
There is a constant
$C$ such that.
\beq\label{2.42abco}
T'(E)\leq C \frac{1}{\sqrt{E-E_{\rm min}}}, \qquad  E\in ( E_{\rm min}, E_0].
\ene
\item[{\rm (b)}]
 The period functions  $T(E)$ has a continuous second   derivative for $ E \in (E_{\rm min}, E_0].$

\item[\rm (c)]  For every $\delta >0$ such that $E_{\rm min}+\delta < E_0$ there is a constant $C$ so  that
\beq\label{2.42abc}
\left | \theta(x, E_1)- \theta(x, E_2)\right | \leq C \sqrt{|E_1-E_2|}, \qquad  E_1, E_2 \in [E_{\rm min}+\delta, E_0].
\ene
\end{enumerate}
\end{prop2}
\begin{proof} 
  Recall  that $T(E)$ is continuously differentiable for $ E \in  ( E_{\rm min}, E_0].$ Further, in Theorem 2.1 of \cite{cw} it is proved
\beq\label{2.42abc1}
T'(E)= \frac{2}{E-E_{\rm min}} \int_0^{x_+(E)}  M(y)\frac{1}
 {\sqrt{2(E-U_0(y))}}  dy,
 \ene
 where,
 \beq\label{2.42abc2}
 M(x):= \frac{ U_0'(x)^2 -2 (U_0(x)-E_{\rm min}) U''_0(x)}{
  U_0'(x)^2 } .
  \ene
Moreover, by \eqref{1.11}, \eqref{1.12}, and \eqref{1.12b}
 \beq\label{mmnn}
 \rho_0(x):= \int f_0(x,v) dv= \ds 2 \int_0^{\sqrt{2(E_0-U_0(x))}} \varphi(E(v,x)) \,dv.
 \ene
 Then,
 \beq\label{mmnn.b}
 \rho_0'(x)=2 \int_0^{\sqrt{2(E_0-U_0(x))}} \varphi'(E)\, U'_0(x) \,dv.
 \ene
 By \eqref{mmnn.b}
 \beq\label{mmnn.cc}
 |U'''_0(x)|= 4\pi |\rho'_0(x)| \leq C |U'_0(x)|=O(|x|), \qquad x \to 0,
 \ene
 where we used   that   $U'_0(0)=0$ and Taylor expansion.
  We  recall that    $\partial^2_xU_0(0)=4\pi \rho_0(0) >0.$ Then, by \eqref{mmnn.cc} and Taylor expansion, 

\beq\begin{array}{c}\label{zzqq}
U_0(x)= U_0(0)+ \frac{1}{2} U''_0(0) x^2+ O(|x|^4), \qquad x \to 0, \\[3pt]
U'_0(x)= U''_0(0)x+O(|x|^3), \qquad x \to 0,\\
U''_0(x)= U''_0(0)+O(x^2), \qquad x \to 0,\\
U'''_0(x)=O(|x|), \qquad x \to 0.
\end{array}
\ene
Hence, by \eqref{zzqq} and as $E_{\rm min}:= U_0(0)$ we get,
 \beq\label{2.42abc4}
 \left(\frac{M(y)}{ U_0'(y)}\right)=O(|y|),  \partial_y \left(\frac{(M(y)}{ U_0'(y)}\right)= O(1), \qquad  y \to 0.
 \ene

Then, using \eqref{2.42abc1}, \eqref{2.42abc4}  and integrating by parts,
  \beq\label{2.42abc5}\begin{array}{l}
T'(E)= -\frac{2}{E-E_{\rm min}} \ds\int_0^{x_+E)}  \frac{M(y)}{U'_0(y)}  \partial_y\sqrt{2(E-U_0(y)} \, dy =\\[8pt]
  \ds\frac{2}{E-E_{\rm min}} \int_0^{x_+(E)} \left[\partial_y\left( \frac{M(y)}{U_0(y)'}\right)\right]  \sqrt{2(E-U_0(y))}\, dy.
 \end{array}
 \ene  
 Furthermore, we used that  $U_0(x_+(E))=E.$ By the second expression for $T'(E)$ in \eqref{2.42abc5}
 \beq\label{s.45aadd}
 \left| T'(E)  \right| \leq \frac{2^{3/2}}{\sqrt{E-E_{\rm min}}}  \int_0^{R_0} \left|\partial_y\left( \frac{M(y)}{U_0(y)'}\right)\right| \, dy.
 \ene
  Equation\eqref{2.42abco} follows from the   second equality in \eqref{2.42abc4} and \eqref{s.45aadd}. This proves (a). We now prove (b).   By  the second equality in \eqref{2.42abc5} we get,
  \beq\label{s.45}\begin{array}{l}
  T(E)''= 
  -\frac{2}{(E-E_{\rm min})^2} \ds\int_0^{x_+(E)} \left[\partial_y \left(\frac{M(y)}{U_0(y)'}\right)\right]  \sqrt{2(E-U_0(y))} dy+\\[9pt]
  \ds \frac{2}{(E-E_{\rm min})} \int_0^{x_+(E)} \left[\partial_y \left(\frac{M(y)}{U_0(y)'}\right)\right]  \frac{1}{\sqrt{2(E-U_0(y))}} dy.
 \end{array}
 \ene  
By \eqref{s.45} $ T(E)$ is  twice continuously differentiable for $ E > E_{\rm min}.$ This proves (b).

 Let us now prove (c). Let us take $ x_-(E) < x <x_+(E).$
To  compute $ \partial_{E} \theta(x,E)$ we proceed in a similar way as in the computation  of the derivative of the period $T(E)$ given in the proof of  Theorem 2.1 in \cite{cw}. Let us denote,
\beq\label{2.42f}
e:= E-E_{\rm min},  \quad G(x):= U_0(x)- E_{\rm min}, \quad \gamma(x,e):= 2(e-G(x)) = 2(E-U_0(x)),
\ene  
\beq\label{2.42g}
I(x,e):= \int_{x_-(e+E_{\rm min})}^x \, \sqrt{\gamma(y,e)}\, dy,
\ene
\beq\label{2.42h}
J(x,e):= \int_{x_-(e+E_{\rm min})}^x \left( \gamma(y,e) -2e \right ) \sqrt{\gamma(y,e)} \,dy.
\ene
 We have,
\beq\label{2.42i}
\partial_e I(x,e)=  \int_{x_-(e+E_{\rm min})}^x \,\frac{1} {\sqrt{\gamma(y,e)}}\, dy,
\ene
and
\beq\label{2.42j}
\partial_e J(x,e)= I(e,x)- 2e \partial_e I(x,e).
\ene
Note that, as $ U_0(x_-(e+E_{\rm min}))= U_0(x_-(E))= E,$ we have that $ \gamma(x_-(e+E_{\rm min},e)=0.$
Moreover  since, 
\beq\label{2.42k}
J(x,e)= \frac{2}{3} \int_{x_-(e+E_{\rm min})}^x \frac{\left( \gamma(y,e) -2e \right )}{(- U'_0(y))} \partial_y\gamma(y,e)^{3/2} dy
\ene
integrating by parts in y  we get,
\beq\label{2.42l}\begin{array}{c}
J(x,e)= 
\ds\frac{2}{3}  \frac{\left( \gamma(x,e) -2e \right )}{(- U'_0(x))} \gamma(x,e)^{3/2} -\\[12pt]
 \ds\frac{2}{3}\int_{x_-(e+E_{\rm min})}^x  \gamma(y,e)^{3/2}       \frac{( U_0'(y)^2 -(U_0(y)-E_{\rm min}) U''_0(y))}{
  U_0'(y)^2 }  dy.
  \end{array}
\ene
Derivating once \eqref{2.42j} with respect to $e$, and  \eqref{2.42l} twice with respect to $e$  we get,
\beq\label{2.42m}\begin{array}{c}
2 e \partial^2_e I(x,e)=\ds 2\frac{U_0(x)-E_{\rm min}}{(-U'_0(x))} \frac{1}{\sqrt{\gamma(x,e)}}+\\[12pt]
\ds \int_{x_-(e+E_{\rm min})}^x        \frac{( U_0'(y)^2 -2 (U_0(y)-E_{\rm min}) U''_0(y))}{
 \sqrt{\gamma(y,e)} U_0'(y)^2 }  dy.
 \end{array}
 \ene
Further, observe that, \eqref{2.36} and \eqref{2.42i}
\beq\label{2.42n}
\theta(x, E)= \frac{1}{T(E)} \partial_e I(x,E). 
\ene
Hence, by \eqref{2.42i}, \eqref{2.42m}, and \eqref{2.42n},
\beq\label{2.42o}
\partial_{E} \theta(x,E)= \left[-\frac{1}{T(E)^2} T'(E)  \int_{x_-(e+E_{\rm min})}^x \,\frac{1} {\sqrt{\gamma(y,e)}}
 dy+ a(x,E)\right],
 \ene
 where
 \beq\label{2.42p} \begin{array}{l} \ds
  a(x,E):= \frac{1}{ 2 T(E)  (E-E_{\rm min}) } \left( 2\frac{U_0(x)-E_{\rm min}}{(-U'_0(x))} \frac{1}{\sqrt{\gamma(x,e)}}+
  \right.\\ \left.\ds \int_{x_-(e+E_{\rm min})}^x        \frac{ M(y)}{
 \sqrt{\gamma(y,e)} }  dy\right). 
 \end{array}
  \ene
Remark that,
$$
 \lim_{x \to 0} 2\frac{U_0(x)-E_{\rm min}}{(-U'_0(x))}=0.
 $$
 Moreover (see \eqref{2.42abc4}), 
 $$
 M(y)= O(y^2), \qquad y \to 0.
 $$
 Then, it follows from \eqref{2.42o},
 \beq\label{2.42p1}
  \left| \partial_{E} \theta(x,E)\right| \leq C \frac{1}{\sqrt{|E- U_0(x)|}}, \qquad   E_{\rm min}+\delta \leq E \leq E_0,  x_-(E) < x <x_+(E).
\ene
Equation \eqref{2.42abc} follows integrating  \eqref{2.42p1}. This proves (c).
\end{proof}

In the following proposition we obtain properties of the functions  $x(\theta,E)$ and $v(\theta, E).$
\begin{prop2}\label{derx}
We recall that $ x(\theta, E),$ and  $v(\theta, E),$ are defined in \eqref{2.41}.
There is a constant $C$ such that.
\begin{enumerate}
\item[{\rm (a)}]
\beq\label{der.1ab}
\left |\partial_{E} x(\theta,E)\right |\leq C \frac{1}{\sqrt{E-E_{\rm min}}}, \qquad E \in (E_{\rm min}, E_0], \theta \in \mathbb S_1,
\ene
\item[{\rm (b)}]
\beq\label{der.1}
\left | x(\theta,E_1)- x(\theta,E_2)\right |\leq C \sqrt{|E_1-E_2|}, \qquad E_1, E_2 \in (E_{\rm min}, E_0], \theta \in \mathbb S_1,
\ene
\item[{\rm (c)}]
\beq\label{der.1aaa}
\left|\partial_{E} v(\theta,E)\right|\leq C \frac{1}{\sqrt{E-E_{\rm min}}}, \qquad E \in (E_{\rm min}, E_0], \theta \in \mathbb S_1.
\ene
\item[{\rm (d)}]\beq\label{der.1a}
\left|v(\theta,E_1)- v(\theta,E_2)\right|\leq C \sqrt{|E_1-E_2|}, \qquad E_1,E_2 \in (E_{\rm min}, E_0], \theta \in \mathbb S_1.
\ene
\end{enumerate}

\end{prop2} 
\begin{proof} A  result  similar to \eqref{der.1ab} and \eqref{der.1}  in the case of a spherically symmetric gravitational Vlasov--Poisson system with an external potential was proved in Lemma 3.5 of \cite{hrss}. The proof in our case is similar. We some details of the proof of \cite{hrss}  since we use the proof of \eqref{der.1ab} and \eqref{der.1} to prove \eqref{der.1aaa} and \eqref{der.1a}.
The solution to the characteristic equations \eqref{2.1} $(X(t, E), V(t,E)$  defined in \eqref{2.40b}, 
satisfies,
\beq\label{der.2}
X(t,E)= x_-(E)- \int_s^t ds \int_0^s U_0'( X(q,E)) \,dq.
\ene
Then,
\beq\label{der.3}
\partial_E X(t,E)=  x'_-(E)- \int_s^t ds \int_0^s U_0''(X(q,E))  \partial_E X(q,E) \, dq.
\ene
Let $(Q(t,E), P(t,E))$ be the solution to the system
\beq\label{der.4}
  P(t,E)=Q'(t,E), P'(t,E)= - U_0''(X(t,E))  Q(t,E)  ,\,\hbox{\rm with } \,Q(0,E)=1,  P(0,E)=0.
 \ene
  Then,
  \beq\label{der.5}
  Q(t,E)=1- \int_0^t ds  \int_0^q  U_0''(X(q,E))  Q(q,E)\, dq.
  \ene
By \eqref{der.5} and Gronwall's inequality,
\beq\label{der.7}
  \left| Q(t,E) \right| \leq C, \qquad (t,E)\in [0,T(E_0)]\times [E_{\rm min}, E_0].
   \ene
   Note that by \eqref{der.3} and \eqref{der.5} 
   \beq\label{der.8}
  \partial_E X(t,E)=Q(t,E)  x'_-(E) .
   \ene
Moreover, by \eqref{2.41}, \eqref{der.7}, and \eqref{der.8}
\beq\label{der.11}
\left|\partial_E x(\theta, E)\right|\leq   C ( \sqrt{E- E_{\rm min} }+  | x'_-(E)|).
\ene
By   \eqref{1.20} and \eqref{mmnn.cc}
%
\beq\label{der.13xx}
|x'_-(E)|=  \frac{1}{\sqrt{2 U''_0(0)}} \frac{1}{\sqrt{E-E_{\rm min}}} (1+O(|x_-(E)|^{1/2})), \qquad  x_-(E)\to 0.
 \ene
Finally, since  $ E \to E_{\rm min},$ as $x_-(E)\to 0,$ we have that \eqref{der.1ab} follows from  \eqref{der.11} and \eqref{der.13xx}. This proves (a). Item (b) follows from (a).
We now  prove (c). By \eqref{der.2}
\beq\label{der.13a}
V(t,E)= - \int_0^t U_0'( X(q,E)) \, dq.
\ene
Then,
\beq   \label{der.13aa}
\partial_E V(t,E)= -\int_0^t U_0''( X(q,E))  \partial_E X(q,E) \,  dq.
\ene
Since $X(t,E)$ is bounded, it follows from\eqref{der.7},  \eqref{der.8} and \eqref{der.13xx}, 
\beq\label{der.13b}
\left|\left(\partial_E V\right)( \theta T(E), E)\right|  \leq C  \frac{1}{\sqrt{E-E_{\rm min}}}, \qquad E \in (E_{\rm min}, E_0].
\ene 
Moreover, by \eqref{2.41}
\beq\label{der.13c}
\partial_E v (\theta, E)=  \theta T'(E) V' (\theta T(E), E)+ \partial_E V(\theta T(E), E).
\ene 
Further, as $V' (\theta T(E), E)$ is bounded for $(\theta, E) \in S_1 \times [E_{\rm min}, E_0]$ it follows from \eqref{2.42abco}, \eqref{der.13b}, and \eqref{der.13c} that \eqref{der.1aaa} holds. This completes the proof of (c).
Item (d) follows fom (c).
 \end{proof} 
 
 In the following proposition we obtain a further property of the angle variable.
\begin{prop2} \label{last}For every $\delta >0$ such that  $E_{\rm min}+ \delta < E_0$ and for every $ 0 < \alpha < 1/2$  there is a constant $C$ such that
\beq\label{der.14}
\left|\theta(x_1,E)- \theta(x_2,E)\right| \leq C  |x_1-x_2|^\alpha  ,  E \in[ E_{\rm min}+ \delta, E_0], \qquad x \in [x_-(E), x_+(E)].
\ene
\end{prop2}
\begin{proof}  Assume that $ x_1 < x_2.$ By \eqref{2.36} and H\"older's inequality,
\beq\label{der.15}
\begin{array}{c}
\left|\theta(x_1,E)- \theta(x_2,E)\right| \ds \leq  C   |x_1-x_2|^{\frac{1}{p}} 
\left[ \int_{x_1}^{x_2} \frac{1}{|E-U_0(y)|^{\frac{q}{2}} }   
 dy\right]^{\frac{1}{q}}, \\[5pt]
\ds  \frac{1}{p}+\frac{1}{q}=1,2 < p \leq \infty.
 \end{array}
\ene
Moreover, there is a $\varepsilon >0$ such that $ |U'_0(x_\pm)(E)| \geq \varepsilon,$ for $E \in  [E_{\rm min}+ \delta, E_0].$  Then, there is a constant $C$ such that,
\beq\label{der.17}
\left[ \int_{x_1}^{x_2} \frac{1}{|E-U_0(y)|^{\frac{q}{2}} } \right]^{\frac{1}{q}}   \leq \left[ \int_{x_-(E)}^{x_+(E)} \frac{1}{|E-U_0(y)|^{\frac{q}{2}} }    dy\right]^{\frac{1}{q}}  \leq C , \qquad   E \in[ E_{\rm min}+ \delta, E_0].
\ene
 Equation \eqref{der.14} follows from \eqref{der.15} and \eqref{der.17}.
\end{proof}

\end{document}